\documentclass[preprint]{imsart}
\usepackage{amsmath,amssymb}
\usepackage{parskip,csquotes}
\usepackage{marvosym}
\usepackage[hang,small,bf]{caption}
\usepackage{mathtools}
\usepackage{verbatim, float}
\usepackage{algorithm,algpseudocode}
\usepackage{xpatch}

\algrenewcommand\alglinenumber[1]{{\sffamily\footnotesize#1}}

\makeatletter
\xpatchcmd{\algorithmic}{\itemsep\z@}{\itemsep=2ex plus2pt}{}{}
\makeatother

\usepackage[colorlinks]{hyperref}
\hypersetup{
	colorlinks,%
	citecolor=blue,%
	filecolor=black,%
	linkcolor=red,%
	urlcolor=blue
}

\usepackage{enumerate}

\RequirePackage[OT1]{fontenc}
\usepackage[letterpaper,margin=1in]{geometry}

\usepackage{amsthm}
\usepackage{amsmath,empheq}
\usepackage{amssymb}
\usepackage{thmtools} 
\usepackage{subcaption}
\usepackage{textcomp}

\usepackage{graphicx}
\usepackage{verbatim}
\usepackage{array, float}

\usepackage{fontenc}
\usepackage[toc,page]{appendix}
\usepackage[nottoc]{tocbibind}
\usepackage[english]{babel}
\usepackage{marvosym}

\usepackage[pagewise]{lineno}

\usepackage{mathtools}

\allowdisplaybreaks
\usepackage{bbm}
\usepackage[noabbrev,capitalize]{cleveref}
\usepackage{listings}

\newtheorem{theorem}{Theorem}

\newtheorem{lemma}{Lemma}
\newtheorem{corollary}{Corollary}
\newtheorem{remark}{Remark}

\newtheorem{prop}{Proposition}

\newtheorem{definition}{Definition}
\newtheorem{observation}{Observation}
\newtheorem{Assumption}{Assumption}

\numberwithin{equation}{section}
\numberwithin{example}{section}
\numberwithin{theorem}{section}
\numberwithin{lemma}{section}
\numberwithin{corollary}{section}
\numberwithin{prop}{section}
\numberwithin{definition}{section}
\numberwithin{remark}{section}
\numberwithin{Assumption}{section}

\usepackage{tikz}
\usepackage{tikz-cd}
\usetikzlibrary{positioning}
\tikzset{
	treenode/.style = {shape=rectangle, rounded corners,
		draw, align=center,
		top color=white, bottom color=cyan!20},
	root/.style     = {treenode, font=\Large, bottom color=yellow},
	env/.style      = {treenode, font=\ttfamily},
	con/.style      = {treenode, font=\ttfamily, bottom color=green!25},
	nocon/.style    = {treenode, font=\ttfamily, bottom color=red!30},
	dummy/.style    = {circle,draw}
}

\startlocaldefs

\newcommand{\hh}{\hat{h}_{\textnormal{PL}}}

\newcommand{\eat}[1]{}
\newcommand{\mM}{\mathcal{M}^\star}

\DeclarePairedDelimiter{\floor}{\lfloor}{\rfloor}

\DeclareMathOperator*{\argmax}{\arg\!\max}

\renewcommand{\bar}[1]{\overline{#1}}
\renewcommand{\hat}[1]{\widehat{#1}}
\renewcommand{\tilde}[1]{\widetilde{#1}}
\newcommand{\E}{\mathbb{E}_N}
\newcommand{\EE}{\mathbb{E}}

\renewcommand{\P}{\mathbb{P}_N}
\newcommand{\lra}{\leftrightarrow}

\newcommand{\scS}{\cS^*}
\newcommand{\R}{\mathbb{R}}
\newcommand{\by}{\mathbf{y}}
\newcommand{\bby}{\mathbf{Y}}
\newcommand{\PB}{\mathbb{P}_{N,\textnormal{bip}}}
\newcommand{\htmp}{\hat{\theta}_{\mathrm{MP}}}

\newcommand{\cU}{\mathcal{U}}
\newcommand{\cV}{\mathcal{V}}
\newcommand{\PPe}{\mathbb{P}_{\bbeta,\textrm{edge}}}
\newcommand{\defn}{\coloneqq}

\newcommand{\tcT}{\tilde{\boldsymbol{\mathcal{T}}}}
\newcommand{\bc}{\mathbf{c}}
\newcommand{\tf}{\tilde{f}}
\newcommand{\A}{\mathbf{A}}
\newcommand{\FG}{\mathfrak{G}}
\newcommand{\tms}{\tilde{\boldsymbol{\sigma}}^{(N)}}
\newcommand{\ud}{\overline{d}_N}

\newcommand{\mz}{\mathbf{P}}
\newcommand{\ns}{\mathrm{rank}}
\newcommand{\M}{\mathcal{M}}

\newcommand{\G}{\mathbf{G}}

\newcommand{\mca}{\mathcal{A}}
\newcommand{\mcb}{\mathcal{B}}

\newcommand{\tI}{\tilde{\mathcal{I}}}
\newcommand{\De}{\Delta}
\newcommand{\st}{t^\star}

\newcommand{\Q}{\mathbf{Q}}

\newcommand{\cD}{\mathcal{D}}
\newcommand{\cE}{\mathcal{E}}
\newcommand{\cF}{\mathcal{F}}

\newcommand{\ind}{\mathbbm{1}}

\newcommand{\tp}{p}
\newcommand{\tq}{q}
\newcommand{\io}{\iota}
\newcommand{\cT}{\boldsymbol{\mathcal{T}}}
\newcommand{\B}{\mathcal{B}}
\newcommand{\cS}{\mathcal{S}}

\newcommand{\abk}{\mathbf{a}^k}
\newcommand{\bbk}{\mathbf{b}^{k_1}}
\newcommand{\mbk}{\mathbf{m}^{k_2}}
\newcommand{\obk}{\mathbf{o}^{k_2}}
\newcommand{\hbm}{\hat{B}_{\textnormal{PL}}}
\newcommand{\hbem}{\hat{\beta}_{\textnormal{PL}}}
\newcommand{\lmn}{\lambda_{\textnormal{min}}}
\newcommand{\fs}{f_{\star}}
\newcommand{\tB}{\tilde{B}}
\newcommand{\tlh}{\tilde{h}}
\newcommand{\tto}{\tilde{t}_1}
\newcommand{\ttt}{\tilde{t}_2}
\newcommand{\bbeta}{\boldsymbol{\beta}}
\newcommand{\PP}{\mathbb{P}}

\@addtoreset{proofpart}{theorem}
\renewcommand\thesection{\arabic{section}}
\makeatletter
\newcommand*{\rom}[1]{\expandafter\@slowromancap\romannumeral #1@}
\patchcmd{\@uclcnotmath}
{\ref}
{\ref\@nonchangecase\autoref}
{}{}
\makeatother

\def\argmax{\mathop{\rm argmax}}

\newcommand{\tih}{\tilde{h}}
\newcommand{\el}{\ell}

\newcommand{\vrh}{\varrho}

\newcommand{\mfm}{\mathfrak{m}}

\newcommand{\tcE}{\tilde{\mathcal{E}}}

\newcommand{\bQ}{\mathbf{Q}}

\newcommand{\tph}{\tilde{\phi}}

\newcommand{\tal}{\tilde{\alpha}}

\newcommand{\Rg}{\R^{\geq 0}}

\newcommand{\hbpm}{\hat{\beta}_{1,\mathrm{PL}}}
\newcommand{\ism}{\mathbb{P}_{N,B_0}^{\textnormal{IS}}}

\newcommand{\cI}{\mathcal{I}}
\newcommand{\cJ}{\mathcal{J}}
\newcommand{\fL}{\mathfrak{N}}
\newcommand{\fM}{\mathfrak{M}}

\newcommand{\tQ}{\tilde{\Q}}
\newcommand{\N}{\mathbb{N}}

\newcommand{\tcS}{\tilde{\cS}}
\newcommand{\tk}{\tilde{k}}

\newcommand{\tbt}{\tilde{\beta}}

\newcommand{\ms}{\boldsymbol{\sigma}^{(N)}}

\newcommand{\fa}{\lVert \A_N\rVert_F^2}

\newcommand{\mc}{\mathbf{c}^{(N)}}

\newcommand{\si}{\sigma}

\newcommand{\Une}{U_{N,\textrm{edge}}}
\newcommand{\Vne}{V_{N,\textrm{edge}}}

\newcommand{\Phb}{\Phi_{\bbeta}}
\newcommand{\vphb}{\varphi_{\bbeta}}
\newcommand{\ps}{p^\star}
\newcommand{\psb}{\ps_{\bbeta}}

\newcommand{\FB}{\mathfrak{B}}
\newcommand{\sk}{k^*}

\newcommand{\cC}{\mathcal{C}^*}
\newcommand{\tcC}{\tilde{\mathcal{C}}^*}

\newcommand{\msk}{\mathcal{K}^*}

\newcommand{\mA}{\mathcal{A}^*}

\newcommand{\ib}{\mathbf{i}^p}
\newcommand{\jb}{\mathbf{j}^q}
\newcommand{\mcH}{\mathcal{H}}

\newcommand{\EEB}{\mathbb{E}_{\bbeta,\mathrm{edge}}}

\newcommand{\FR}{\mathfrak{R}}
\newcommand{\mH}{\mathcal{H}^*}

\newcommand{\Tne}{T_{N,\mathrm{edge}}}

\newcommand{\mU}{\mathfrak{U}}
\newcommand{\mV}{\mathfrak{V}}

\newcommand{\tb}{\tilde{\pi}}

\newcommand{\sech}{sech}

\definecolor{LightCyan}{rgb}{0.88,1,1}
\definecolor{Gray}{gray}{0.9}

\endlocaldefs

\usepackage{abstract}
\RequirePackage[numbers]{natbib}
\begin{document}
	
	\renewcommand{\abstractname}{}    
	\renewcommand{\absnamepos}{empty}
	
	\begin{frontmatter}
		
		\title{Pivotal CLTs for Pseudolikelihood via Conditional Centering in Dependent Random Fields}
		\runtitle{Pseudolikelihood and conditional centering}

		\begin{aug}
			
			\author{\fnms{Nabarun} \snm{Deb}\ead[label=e1]{nabarun.deb@chicagobooth.edu}}
			
			\runauthor{N. Deb}
			
			\address{University of Chicago\\ \printead{e1}}
			
		\end{aug}
		\vspace{0.1in}
		\begin{abstract}
			In this paper, we study fluctuations of conditionally centered statistics of the form 
			$$N^{-1/2}\sum_{i=1}^N c_i(g(\si_i)-\E[g(\si_i)|\si_j,j\neq i])$$
			where $(\si_1,\ldots ,\si_N)$ are sampled from a dependent random field, and $g$ is some bounded function. Our first main result shows that  under weak smoothness assumptions on the conditional means (which cover both sparse and dense interactions), the above statistic converges to a Gaussian \emph{scale mixture} with a random scale determined by a \emph{quadratic variance} and an \emph{interaction component}. We also show that under appropriate studentization, the limit becomes a pivotal Gaussian. We leverage this theory to develop a general asymptotic framework for maximum pseudolikelihood (MPLE) inference in dependent random fields. We apply our results to Ising models with pairwise as well as higher-order interactions and exponential random graph models (ERGMs). In particular, we obtain a joint central limit theorem for the inverse temperature and magnetization parameters via the joint MPLE (to our knowledge, the first such result in dense, irregular regimes), and we derive conditionally centered edge CLTs and marginal MPLE CLTs for ERGMs without restricting to the ``sub-critical" region. Our proof is based on a method of moments approach via combinatorial decision-tree pruning, which may be of independent interest.
		\end{abstract}
		
		\begin{keyword}[class=MSC]
			\kwd[Primary ]{82B20}
			\kwd[; secondary ]{82B26}
		\end{keyword}
		
		\begin{keyword}
			\kwd{Decision trees, Exponential random graph model, Fa\`{a} Di Bruno formula, Gaussian scale mixtures, Ising model,  Method of moments}
		\end{keyword}
	\end{frontmatter}
	
	\section{Introduction}\label{sec:intro}
	
	Dependent random fields---and especially \emph{network models}---are now routine in applications ranging from social and economic interactions to spatial imaging and genomics (see \cite{Fienberg2010a,Fienberg2010b} for surveys). Data from such models often exhibits significant deviations from classical Gaussian approximations. A natural class of statistics to analyze in such models are \emph{conditionally centered} averages (see \cite{Comets1998,Janzura2002,gaetan2004central}), where one recenters the observations by their mean, given all other observations. Crucially, such conditionally centered CLTs are closely tied to maximum \emph{pseudolikelihood} estimators (MPLEs) through the MPLE score (see \cite{jensen1994asymptotic,hofling2009estimation,ekeberg2013improved}). This connection is practically important because in many graphical/Markov random field models (such as \emph{Ising models}, \emph{exponential random graph models} (ERGMs), etc.), computing the MLE is impeded by an \emph{intractable normalizing constant}, whereas pseudolikelihood replaces the joint likelihood with a product of tractable conditional models, scales to large networks, and is widely usable in practice.
	
	However, most existing theory for conditionally centered statistics and for MPLE focuses on \emph{local dependence} --- e.g., bounded degree or sparse neighborhoods --- and does not cover \emph{realistic dense regimes} in which every node may have many connections (which scale with the size of the network). This paper bridges that gap by developing a general limit theory for conditionally centered statistics under weak and verifiable assumptions. Our results accommodate both sparse and dense interactions, as well as regular and irregular network connections. In particular, we deliver valid studentized inference for pseudolikelihood in network/Markov random field settings. As examples, we obtain new CLTs for conditionally centered averages and pseudolikelihood estimators in Ising models (with pairwise and tensor interactions), and exponential random graph models, \emph{without imposing sparsity, regularity, or high temperature  restrictions}. 
	
	To be concrete, let $\B$ denote a Polish space. For $N\ge 1$, suppose  $\ms:=(\si_1,\ldots ,\si_N)\sim \P$ where $\P$ is a probability measure supported on $\B^N$. Let $g:\B\to [-1,1]$ be a bounded function. Also let $\mc:=(c_1,\ldots ,c_N)\in\R^N$. We are interested in studying the fluctuations of the following conditionally centered weighted average of $g(\si_i)$'s:
	\begin{equation}\label{eq:pivotstat}T_N:=\frac{1}{\sqrt{N}}\sum_{i=1}^N c_i\big(g(\si_i)-\E[g(\si_i)|\si_j,j\neq i]\big)\end{equation}
	under $\P$. If $\P$ is a product measure on $\B^N$, then the centering in $T_N$ reduces to $\E[g(\si_i)]$, in which case a limiting normal distribution for $T_N$ can be derived under mild assumptions on $\mc$ using the Lindeberg Feller Central Limit Theorem~\cite{Lindeberg1922}. In the absence of independence, the fluctuations for $T_N$ are known only in very specific cases, mostly restricted to random fields on fixed lattice systems or under strong mixing assumptions (see~\cite{Janzura2002,Comets1998,jensen1994asymptotic,gaetan2004central,jalilian2025central}), or when the dependence is governed by a complete graph model (see \cite{Somabha2021,Sanchayan2025}). However, with the gaining popularity of large network data in modern data science, probabilistic models that facilitate more complex dependence structures have attracted significant attention in both Probability and Statistics; see e.g.,~\cite{Fienberg2010a,Fienberg2010b} for a review. Such models often involve dense interactions and do not satisfy traditional mixing assumptions. Examples include the Ising/Potts model on dense graphs~\cite{Ravikumar2010,Basak2017,Dommers2016}, exponential random graph models~\cite{Ganguly2024,Chatterjee2013,Bhamidi2011}, general pairwise interaction models~\cite{Starr2009,drton2017structure,Ming2016,lee2025clt}, etc. to name a few (see~\cref{sec:examp} for further references). 
	
	The analysis of the statistic $T_N$ (and its variants) is of pivotal importance in the aforementioned models. Their tail probabilities have been exploited in statistical estimation and testing (see~\cite{Chatterjee2007,ghosal2018joint,Daskalakkis2019,DaskalakkisTesting,Mukherjee2018,deb2020detecting}). As mentioned above, the limiting behavior of $T_N$ is inextricably linked to pseudolikelihood estimators which provide a computationally tractable alternative to the MLE. Motivated by these applications, the goal of this paper is to study the fluctuation of $T_N$ in a near \emph{``model-free"} setting. We obtain pivotal limits for $T_N$ under a random (data-driven) studentization (see \cref{theo:CLTmain}) whenever the conditional means satisfy a discrete smoothness condition. This condition accommodates both sparse and dense interactions simultaneously. The studentization involves two components --- the first captures a quadratic variation and the second captures the effect of dependence. As a consequence, we show that $T_N$ converges to a \emph{Gaussian scale mixture} (see \cref{theo:conmain}) when the random scale  converges weakly. As our flagship application, we use our main results to study pseudolikelihood inference in a broad class of models. The flexibility of our main results (Theorems \ref{theo:CLTmain} and \ref{theo:conmain}) ensures that they apply to a plethora of models in one go. Below we highlight our main contributions in further detail.
	
	\subsection{Main contributions}\label{subsec:main-contrib}
	
	1. \textbf{Pivotal and structural limits}
	\begin{itemize}
		\item \emph{Pivotal limit.} In \cref{theo:CLTmain}, we show that there exists two data-driven terms: $U_N$ that captures the quadratic variation and $V_N$ which captures the interaction, such that 
		\[
		\frac{T_N}{\sqrt{U_N+V_N}}\to N(0,1)
		\]
		in the topology of weak convergence, provided the conditional expectations $\E[g(\si_i)|\si_j,j\neq i]$ are smooth with respect to leave-one-out perturbations (see \cref{as:cmean}). This assumption is not tied to a specific model. We illustrate in \cref{sec:mres} using the Ising model that \cref{as:cmean} holds both for sparse and dense interactions, which is the key distinguishing feature of our result with the existing literature. 
		\item \emph{Structural limit.} In the event $(U_N,V_N)$ converges weakly to a distribution $(P_1,P_2)$, $T_N$ converges to a \emph{Gaussian scale mixture}:
		\[
		T_N \;\to\; \sqrt{P_1+P_2}\,Z, \qquad Z\sim\mathcal{N}(0,1)\ \text{independent of }(P_1,P_2).
		\]
		As $(P_1,P_2)$ need not be degenerate, this result is not a consequence of a Slutsky type argument, but instead we prove a joint convergence of $(T_N,U_N,V_N)$. The proof proceeds using a method of moments technique coupled with decision-tree pruning.
		\item \emph{Verifying \cref{as:cmean}.} In \cref{lem:smoothcont}, we also provide a convenient tool for verifying \cref{as:cmean} that is applicable to a broad class of network models. 
	\end{itemize}
	\vspace{0.25em}
	
	2. \textbf{Consequences for pseudolikelihood (MPLE) inference.}
	\begin{itemize}
		\item \emph{Direct import of limits.} Because MPLE is built from local conditional models, its score inherits the conditional-centering structure. Theorems \ref{theo:CLTmain} and \ref{theo:conmain} therefore \emph{transfer} to pseudolikelihood estimators using Z-estimation techniques, yielding a pivotal CLT (see \cref{prop:CLTMPLE}).
		\item \emph{Reality of the mixture.} In \cref{sec:mixturepseudo}, we show the relevance of the Gaussian mixture phenomenon. In an Ising model example on a bipartite-graph, Propositions \ref{prop:bipartg} and \ref{prop:bipartpseudo} show that both $T_N$ and the MPLE have has a Gaussian mixture limit where we identify the mixture components based on the solution of a fixed-point equation. 
	\end{itemize}
	\vspace{0.25em}
	
	3. \textbf{Applications to Ising models: pairwise and higher-order (tensor) interactions.}
	\begin{itemize}
		\item \emph{Generality.} In Theorems \ref{theo:conmainis} and \ref{thm:highordclt}, we obtain the studentized limit of $T_N$ for Ising models under pairwise and under higher order interactions respectively. The only condition required is a certain row summability of the interaction matrix/tensor that is satisfied both in sparse and dense regimes.
		\item \emph{Joint CLTs under irregular interactions.} A fundamental problem in Ising models is the estimation of the inverse temperature $\beta$ and the magnetization parameter $B$. To the best of our knowledge, there are no known CLTs for any estimator of $(\beta,B)$ jointly. In \cref{sec:jointpseudo}, we provide the first \emph{joint} CLT for the inverse temperature and magnetization parameters $(\beta,B)$ using the joint MPLE in dense, \emph{irregular} interaction regimes; see Theorems \ref{thm:jointCLT} and \ref{thm:jointCLThigh} for the pairwise and the higher order interaction cases respectively. 
		\item \emph{Efficiency in approximately regular graphs.} In \cref{sec:marginalpseudo}, we study marginal MPLEs in Ising models, when the interactions are dense and \emph{approximately regular} graphs. In Theorems \ref{thm:margbeta} and \ref{thm:margB}, we prove that the \emph{marginal} MPLEs attain the Fisher-efficient variance, matching the \emph{asymptotic limit of the maximum likelihood estimators} (MLEs). This makes a strong case for MPLEs over MLEs in such regimes as the MLEs are often computationally intractable. To the best of our knowledge, the limit theory for the MPLE was only known for the Curie-Weiss (complete graph) model in the existing literature, whereas our results show that the same limit extends to the much broader regime where the average degree of the underlying graph diverges to $\infty$ (irrespective of the rate).
	\end{itemize}
	\vspace{0.25em}
	
	4. \textbf{Applications to ERGM.}
	\begin{itemize}
		\item \emph{CLT for $T_N$ beyond sub-criticality.} For exponential random graph models (ERGMs) we establish central limit theorems at the level of conditionally centered statistics (see \cref{thm:ergmclt}), under a \emph{variance positivity} condition. Contrary to the existing literature, these results \emph{do not restrict to the well known sub-critical regime}. This is made possible by our main CLT in \cref{theo:CLTmain}, which only requires the smoothness assumption on the conditional means (i.e., \cref{as:cmean}) that is easily verified in ERGMs. In \cref{cor:ergmclt}, we simplify the variance in the sub-critical regime. The same result also applies to the Dobrushin uniqueness regime where the coefficients may take small negative values (not directly covered in the sub-critical regime).
		\item \emph{Marginal MPLE limits beyond sub-criticality.} Using \cref{prop:CLTMPLE}, we then derive studentized CLTs for the marginal MPLE for the coefficient associated with ERGM edges (see \cref{thm:MPLergmCLT}). Once again, we do not restrict to the sub-critical regime. The variance however does simplify considerably in the sub-critical regime which is provided in \eqref{eq:secondconc}.
	\end{itemize}
	
	\subsection{Organization}
	In \cref{sec:mres}, we provide our main result \cref{theo:CLTmain} under \cref{as:cmean}. The same Section also contains a Gaussian scale mixture limit for $T_N$ under added stability conditions. In \cref{sec:MPLE}, we show how the main results can yield a theory for pseudolikelihood based inference. In \cref{sec:howtover}, we provide a convenient analytic technique to verify \cref{as:cmean} that considerably simplifies the verification procedure in many network models. In \cref{sec:examp}, we apply our results to Ising models with pairwise/tensor interactions and ERGMs. In \cref{sec:proof-overview}, we provide a technical road map for proving our main results. Finally the Appendix contains the technical details and proofs.
	
	\section{Main result}\label{sec:mres}
	We begin this section with some notation. Let $\N$ be the set of natural numbers and $[N]$ denote the set $\{1,2,\ldots ,N\}$ for $N\in\N$. We will write $\E$ for expectations computed under $\P$. Given any $\ms=(\si_1,\si_2,\ldots ,\si_N)\in \B^N$ and any set $\cS\subseteq [N]$, let  $\ms_{\cS}:=(\si_{\cS,1},\si_{\cS,2},\ldots ,\si_{\cS,N})$ denote the vector which satisfies:
	\begin{equation}\label{eq:newsig}\si_{\cS,i}=\begin{cases} \si_i & \mbox{if } i\in \cS^c \\ b_0 & \mbox{if } i\in S\end{cases}\end{equation} for all $i\in [N]$, where $b_0$ is an arbitrary but fixed (free of $N$, $\cS$) element in $\B$.
	Define 
	\begin{equation}\label{eq:consig2} t_i\equiv t_i\big(\ms\big):=\E[g(\si_i)|\si_j,j\neq i]
	\end{equation} 
	and for any subset $\cS\subseteq [N]$ set
	\begin{equation}\label{eq:consig}t_i^{\cS}\equiv t_i^{\cS}(\ms):=t_i\big(\ms_{\cS}\big)=\E[g(\si_i)|\sigma_j=b_0\ \mbox{for}\ j\in\cS,\ \si_j\ \mbox{for}\ j\in \cS^c,\ j\neq i],\end{equation}
	where $\ms_{\cS}$ is defined in~\eqref{eq:newsig}. Throughout this paper, we also drop the set notation for singletons, i.e., $\{a\}$ and $a$ will both denote the singleton set with element $a$, as will be obvious from context. With this understanding, and choosing $\cS=\{j\}$ in~\eqref{eq:consig}, we can write $t_i^j=\E[g(\si_i)|\si_k,\ k\neq i,\ \si_j=b_0]$ for $j\neq i$.  Also, we will use $\overset{w}{\longrightarrow}$ to denote weak convergence of random variables and $|A|$ to denote the cardinality of a finite set $A$. Also $\phi$ will denote the empty set throughout the paper. 
	
	We are now in a position to state our main assumptions.
	
	\begin{Assumption}\label{as:coeff}[Uniform integrability of coefficient vector]
		The vector $\mc=(c_1,\ldots ,c_N)$ satisfies the following condition:
		$$\lim_{L\to\infty}\limsup_{N\to\infty}\frac{1}{N}\sum_{i=1}^N c_i^2\ind(|c_i|\geq L)=0.$$
	\end{Assumption}
	The above imposes a uniform integrability condition on the empirical measure $\frac{1}{N}\sum_{i=1}^N \delta_{c_i^2}$. Even in the case where $\P$ is a product measure, to obtain a CLT for $T_N$, it is necessary to assume that the above empirical measure has asymptotically bounded moments. \cref{as:coeff} is a mildly stronger restriction.
	
	\begin{Assumption}\label{as:cmean}[Smoothness of conditional mean]
		For any fixed $N\geq 1, k\geq 2$, there exists a $N\times N\times \ldots \times N$ ($k$-fold) tensor $\Q_{N,k}:=\{\Q_{N,k} (j_1,\ldots ,j_k)\}_{(j_1,\ldots ,j_k)\in [N]^k}$ with non-negative entries, such  that, for any set $\cS=\{j_1,j_2,\ldots ,j_k\}\in [N]^k$ of distinct elements, $\tcS\subseteq [N]$, $\cS\cap \tcS=\phi$, the following holds:
		\begin{equation}\label{eq:cmean1}\bigg|\sum_{D\subseteq \cS\setminus\{j_1\}}(-1)^{|D|} t_{j_1}^{\tcS\cup D}\bigg| \leq \Q_{N,k}(j_1,j_2,\ldots ,j_{k}).
		\end{equation}	
		Further, the tensors $\Q_{N,k}$ satisfy the following property:
		\begin{equation}\label{eq:cmean2}\limsup_{N\to\infty}\max_{\el\in [k]}\max_{j_{\el}\in [N]}\sum_{(\{j_1,j_2,\ldots ,j_{k}\}\setminus \{j_{\el}\})\in [N]^{k-1}} \Q_{N,k}(j_1,j_2,\ldots ,j_k)< \infty.
		\end{equation}
		
		Without loss of generality, we assume for the rest of the paper that  $\Q_{N,k}(j_1,j_2,\ldots ,j_k)$ is symmetric in its last $k-1$ arguments (for every $k\in\N$, $k\geq 2$). This is possible because the left hand side of~\eqref{eq:cmean1} is symmetric about $j_2,\ldots ,j_k$ which means we can replace
		$$\Q_{N,k}(j_1,j_2,\ldots ,j_k)\mapsto \sum_{\sigma\in \mathcal{P}_{k-1}} \Q_{N,k}(j_1,j_{\sigma(1)+1},\ldots ,j_{\sigma(k-1)+1})$$
		where $\mathcal{P}_{k-1}$ is the set of all permutations of $[k-1]$. It is easy to see that under the above transformation, $\Q_{N,k}(j_1,\ldots ,j_k)$ still satisfies~\eqref{eq:cmean2}.
		
	\end{Assumption}
	\noindent \cref{as:cmean} can be interpreted as a boundedness assumption on the discrete derivatives of appropriate conditional means by the elements of a tensor, which is assumed to have bounded row sums (by \eqref{eq:cmean2}). For better comprehension of \cref{as:cmean}, we will use the $\pm 1$-valued Ising model as a working example. It is defined as  
	
	\begin{equation}\label{eq:twospIsing}
		\P^{\textnormal{IS}}(\ms):=\frac{1}{Z_N^{\textnormal{IS}}}\exp\left(\frac{1}{2}(\ms)^{\top}\A_N(\ms)\right),
	\end{equation}
	where each $\sigma_i\in \pm1$, $\A_N$ is a symmetric matrix with non-negative entries and $0$s on the diagonal, and $Z_N^{\textnormal{IS}}$ is the partition function. We choose $\mathcal{B}=[-1,1]\supseteq \{\pm 1\}$ and $b_0=0$. We emphasize that our results hold in much more generality as will be seen in \cref{sec:examp}. Now, as a simple illustration, consider the case $k=2$, $g(x)=x$, $\tcS=\phi$, and $\cS=\{j_1,j_2\}$. Then, under the model \eqref{eq:twospIsing}, the left hand side of \eqref{eq:cmean1} becomes 
	$$|t_{j_1}-t_{j_1}^{j_2}|= \A_N(j_1,j_2)\bigg|\sigma_{j_2}\int_0^1 \sech^2\bigg(\sum_{k\ne j_2} \A_N(j_1,k)\sigma_{k}+s\A_N(j_1,j_2)\sigma_{j_2}\bigg)\,ds\bigg|\le \A_N(j_1,j_2),$$
	where the last inequality uses the fact that $\sech^2(\cdot)$ is bounded by $1$ and $|\sigma_{j_2}|=1$. Therefore, under \eqref{eq:twospIsing}, $\Q_{N,2}(j_1,j_2)$ can be chosen as the entries $\A_N(j_1,j_2)$ of the interaction matrix. Now \eqref{eq:cmean2} reduces to assuming that $\A_N$ has bounded row sums which is a common assumption in this literature (see \cref{sec:ising} for examples). To go one step further, let $k=3$, $g(x)=x$, $\tcS=\phi$, and $\cS=\{j_1,j_2,j_3\}$. In that case, the left hand side of \eqref{eq:cmean1} becomes 
	\begin{align*}
		&\;\;\;\;|t_{j_1}-t_{j_1}^{j_2}-t_{j_1}^{j_3}+t_{j_1}^{j_2,j_3}| \\ &=\A_N(j_1,j_2)\A_N(j_1,j_3)\bigg|\sigma_{j_2}\sigma_{j_3}\int_0^1\int_0^1 (\tanh)''\big(\sum_{k\neq j_2,j_3} \A_N(j_1,k)\sigma_k+\A_N(j_1,j_2)\sigma_{j_2}+\A_N(j_1,j_3)\sigma_{j_3}\big)\,ds\,dt\bigg| \\ &\le \A_N(j_1,j_2)\A_N(j_1,j_3).
	\end{align*}
	In the last inequality we additionally use the fact that $(\tanh''(\cdot))$ is uniformly bounded by $1$. Therefore the entries of the third order tensor $\Q_{N,3}(j_1,j_2,j_3)$ can be chosen as $\A_N(j_1,j_2)\A_N(j_1,j_3)$. Further if we assume that the maximum row sum for $\A_N$ is bounded by some $c>0$, then elementary computations reveal that the maximum row sum for $\Q_{N,3}$ is bounded by $c^2$, which will imply that $\Q_{N,3}$ satisfies \eqref{eq:cmean2}. A similar computation can be carried out for general $k$ as well. In fact, the entries of the $k$-th order tensor $\Q_{N,k}(j_1,j_2,\ldots ,j_k)$ can be chosen as $\A_N(j_1,j_2)\A_N(j_1,j_3)\ldots \A_N(j_1,j_k)$ and the corresponding maximum row sum can be bounded by $c^{k-1}$, up to a multiplicative factor of $k$ (see \cref{sec:howtover} for details).
	
	\noindent \textit{To ease the burden of verifying \cref{as:cmean} for future use, we provide a tractable way to check this assumption in \cref{sec:howtover} that is broadly applicable across a large class of models}.
	
	\begin{theorem}\label{theo:CLTmain}
		Suppose Assumptions \ref{as:coeff} and \ref{as:cmean} hold. Define the random variables 
		\begin{align}\label{eq:randvar}
			U_N:=\frac{1}{N}\sum_{i=1}^N c_i^2 (g(\si_i)^2-t_i^2) \quad \mbox{and} \quad V_N := \frac{1}{N}\sum_{i,j} c_i c_j (g(\si_i)-t_i)(t_j^i-t_j).  
		\end{align}
		We assume that there exists $\eta>0$ such that 
		\begin{align}\label{eq:bddzero}
			\P(U_N+V_N\geq \eta) \to 1, \quad \mbox{as}\,\,  N\to\infty.
		\end{align}
		Then given any sequence of positive reals $\{a_N\}_{N\ge 1}$ such that $a_N\to 0$, we have 
		\begin{align*}
			\frac{T_N}{\sqrt{(U_N+V_N)\vee a_N}}\overset{w}{\longrightarrow} N(0,1).
		\end{align*}
	\end{theorem}
	
	The above result will follow as a consequence of a more general moment convergence result. To state it, we begin with the following Assumption. 
	
	\begin{Assumption}\label{as:empcon}[An empirical convergence condition]
		There exists a bivariate random variable $\mz\defn (P_1,P_2)$ such that the following holds: 
		$$\begin{bmatrix}
			\frac{1}{N}\sum_{i=1}^N c_i^2\left(g(\si_i)^2-t_i^2\right)\\ \frac{1}{N}\sum_{i,j} c_i c_j (g(\si_i)-t_i)(t_j^{i}-t_j)
		\end{bmatrix}\overset{w}{\longrightarrow} \mz.$$
	\end{Assumption}
	
	To understand \cref{as:empcon}, we note that, under \cref{as:coeff}, we have
	\begin{align}\label{eq:boundZ1}
		\bigg|\frac{1}{N}\sum_{i=1}^N c_i^2 (g(\si_i)^2-t_i^2)\bigg|\le \frac{1}{N}\sum_{i=1}^N c_i^2 \le \sup_{N\ge 1} \frac{1}{N}\sum_{i=1}^N c_i^2<\infty.
	\end{align}
	Further under \cref{as:cmean}, we have: 
	$$\bigg|\frac{1}{N}\sum_{i,j} c_i c_j (g(\si_i)-t_i)(t_j^i-t_j)\bigg|\le \frac{2}{N}\sum_{i,j} |c_i||c_j|\Q_{N,2}(j,i).$$
	By \eqref{eq:cmean2}, $\Q_{N,2}$ has uniformly bounded row sums, say, by some constant $c>0$. This implies that the operator norm of $\Q_{N,2}$ is also bounded by c. As a result, by \cref{as:coeff}, we have that 
	\begin{align}\label{eq:boundz2}
		\bigg|\frac{1}{N}\sum_{i,j} c_i c_j (g(\si_i)-t_i)(t_j^i-t_j)\bigg| \le \frac{(2c)}{N}\sum_{i=1}^N c_i^2 \le (2c) \sup_{N\ge 1} \frac{1}{N}\sum_{i=1}^N c_i^2<\infty.
	\end{align}
	The above displays imply that the random sequence in the left hand side of \cref{as:empcon} is  already asymptotically tight. Therefore, by Prokhorov's Theorem, all subsequential limits exist. \cref{as:empcon} simply requires all the subsequential limits to be the same. 
	
	We are now in the position to state the more general form of \cref{theo:CLTmain} which may be of independent interest.  
	
	\begin{theorem}\label{theo:conmain}
		For any $k,k_1,k_2\in\N\cup\{0\}$, under Assumptions~\ref{as:coeff},~\ref{as:cmean}, and~\ref{as:empcon}, the following sequence 
		\begin{equation}\label{eq:limmoment}
			m_{k,k_1,k_2}:=
			\begin{cases}0 & \mbox{ if }k\mbox{  is odd}\\ (k)!!\EE[(P_1+P_2)^{k/2}P_1^{k_1}P_2^{k_2}] & \mbox{ if }k\mbox{ is even}
			\end{cases},
		\end{equation}
		where $(k)!!:=1\times 3\times 5\times \ldots \times (k-1)$ for $k$ even, is well defined. Recall the definitions of $U_N$ and $V_N$ from \eqref{eq:randvar}. Then, for all $k,k_1,k_2\in\N\cup \{0\}$, we have 
		\begin{align}\label{eq:bigconmoment}
			\E T_N^{k}U_N^{k_1}V_N^{k_2} \to m_{k,k_1,k_2}.
		\end{align}
		This implies that there exists a unique probability measure $\rho$ with moment sequence $m_{k,0,0}$. Further $P_1+P_2$ is non-negative almost everywhere and we have 
		$$T_N\overset{w}{\longrightarrow}\rho=\textnormal{Law}(\sqrt{P_1+P_2}Z),$$
		where $Z\sim N(0,1)$ is independent of $\mz=(P_1 , P_2)$.
	\end{theorem}
	Intuitively, \(P_1\) encodes the ``local'' quadratic variance created by conditional centering, while \(P_2\) aggregates the residual variance due to interactions. Let us discuss two special cases of \cref{theo:conmain}. 
	\begin{enumerate}
		\item In the special case where $\mz=(P_1,P_2)$ is degenerate, say $\delta_{(p_1,p_2)}$ for some reals $p_1,p_2$, \cref{theo:conmain} implies that $T_N\overset{w}{\longrightarrow} N(0,p_1+p_2)$. The non-negativity of $p_1+p_2$ in this case is a by-product of the Theorem itself. 
		\item It is indeed possible for $P_1+P_2$ to have a non-degenerate limit law, in which case the unstandardized limit of $T_N$ is a Gaussian scale mixture. A concrete example is provided in \cref{sec:mixturepseudo}. 
	\end{enumerate}
	
	\begin{remark}[Avoiding \cref{as:empcon}]\label{rem:remempcon}
		We note here that in the absence of \cref{as:empcon}, the conclusion of \cref{theo:conmain} holds along subsequences, although these subsequential limits need not be the same (i.e., the limit $\rho$ might depend on the chosen subsequence). Therefore the primary purpose of \cref{as:empcon} is to provide a clean characterization for the limit of $T_N$.
	\end{remark}
	
	\begin{remark}[Comparison with \cite{Comets1998}]\label{rem:compare}
		\cite[Theorem 2.1]{Comets1998} prove a studentized CLT for sums of conditionally centered \emph{local} fields on $\mathbb{Z}^d$ with \emph{fixed} finite neighborhoods. Their proof is based on Stein's method and crucially hinges on the local (not growing) nature of the  random field,  thereby precluding the possibility of any dense interactions. In contrast, \cref{theo:CLTmain} here yields a randomly studentized pivot
		\[
		\frac{T_N}{\sqrt{(U_N+V_N)\vee a_N}}\ \Rightarrow\ \mathcal N(0,1)
		\]
		\emph{without} imposing locality or lattice structure. Moreover, our result \cref{theo:conmain} establishes \emph{joint} convergence of $(T_N,U_N,V_N)$ and identifies the raw limit $T_N\Rightarrow \sqrt{P_1+P_2}\,Z$. Consequently, whenever $U_N+V_N$ has a nondegenerate subsequential limit (see \cref{sec:mixturepseudo} for an example), the present framework pins down the exact Gaussian–mixture law for $T_N$—a conclusion not available from the \cite{Comets1998} studentized result alone, in the absence of additional stable/joint convergence assumptions.
	\end{remark}
	
	\section{Asymptotic normality of maximum pseudolikelihood estimator (MPLE)}\label{sec:MPLE}
	The conditionally centered CLT established in \cref{theo:CLTmain} is intricately connected to asymptotic normality of the maximum pseudolikelihood estimator (MPLE) for random fields. To wit, suppose that $(d\mathbb{P}_{\theta}/d\nu)(\si_i|\si_j,j\neq i)$ denotes the conditional density of $\sigma_i$ given all the other $\sigma_j$s, indexed by some parameter $\theta\in\R^p$, and with respect to some dominating measure $\nu$. Let $\theta_0\in\R^p$ denote the true parameter and let the open set $\Theta$ be the parameter space. The MPLE is defined as 
	\begin{align}\label{eq:MPLE}
		\htmp\in \argmax_{\theta\in\Theta}\sum_i f_i(\theta)\,\, , \,\, \mbox{where} \,\, f_i(\theta):=\log \frac{d\mathbb{P}_{\theta}}{d\nu}(\sigma_i|\sigma_j,j\neq i).
	\end{align}
	The MPLE, introduced by Besag \cite{Besag1974,besag1975statistical}, has since attracted widespread attention in the statistics, probability, and machine learning community over the years; see e.g.~\cite{hofling2009estimation,ekeberg2013improved,Ravikumar2010,Comets1998,Comets1991,jensen1994asymptotic}. A natural approach to obtaining a central limit theory for $\htmp$ proceeds as follows: first, one starts with the score equation 
	$$\sum_i \nabla f_i(\htmp)=0.$$
	By a first order Taylor expansion, and ignoring higher order error terms, the above equation can be rewritten as 
	\begin{equation}\label{eq:bareidea}
		\sqrt{N}(\htmp-\theta_0)\approx \bigg(-\frac{1}{N}\sum_i \nabla^2 f_i(\theta_0)\bigg)^{-1}\big(N^{-1/2}\sum_i \nabla f_i(\theta_0)\big).
	\end{equation}
	It is then reasonable to expect that the asymptotic normality of $\sqrt{N}(\htmp-\theta_0)$ will be driven by the asymptotic normality of $N^{-1/2}\sum_i \nabla f_i(\theta_0)$. The main observation here is that, under enough regularity, 
	\begin{align}\label{eq:concenter}
		\mathbb{E}[\nabla f_i(\theta_0)|\si_j,j\neq i]=\int \nabla\frac{d\mathbb{P}_{\theta}}{d\nu}(\si_i|\si_j,j\neq i)\bigg|_{\theta=\theta_0}\,d\nu(\si_i)=\nabla_{\theta}\left(\int \frac{d\mathbb{P}_{\theta}}{d\nu}(\si_i|\si_j,j\neq i)\,d\nu(\sigma_i)\right)\bigg|_{\theta=\theta_0}=0.\end{align}
	In other words, $\nabla f_i(\theta_0)$s are already conditionally centered which makes \cref{theo:CLTmain} a critical tool for obtaining the Gaussianity of $\htmp$. To provide a further concrete example, consider the two-spin Ising model from \eqref{eq:twospIsing} with an additional magnetization term, i.e.,
	\begin{equation}\label{eq:twospIsingmagnet}
		\ism(\ms):=\frac{1}{Z_{N,B_0}^{\textnormal{IS}}}\exp\left(\frac{1}{2}(\ms)^{\top}\A_N(\ms)+B_0\sum_i \si_i\right),
	\end{equation}
	where, as before, each $\sigma_i\in \pm1$, $\A_N$ is a symmetric matrix with non-negative entries and $0$s on the diagonal, and $Z_{N,B_0}^{\textnormal{IS}}$ is the partition function. Assume that the magnetization parameter $B_0$ is unknown. A simple computation yields that the MPLE $\hbm$ satisfies 
	\begin{align}\label{eq:pseudois}
		\sum_i \big(\si_i-\tanh\big(\sum_{j}\A_N(i,j)\si_j+\hbm\big)\big)=0.
	\end{align}
	As argued earlier in \eqref{eq:bareidea}, a CLT for $\sqrt{N}(\hbm-B_0)$ follows from the CLT of $N^{-1/2}\sum_i (\si_i-\tanh(\sum_j \A_N(i,j)\si_j+B_0))$, which is the subject of \cref{theo:CLTmain}. In the applications to follow, we will show that more complicated instances involving CLTs for vector parameters (e.g. both inverse temperature and magnetization) can also be derived from \cref{theo:CLTmain}. 
	
	We now present a proposition which provides the limit distribution of $\htmp$ under high level conditions. This follows from classical results in M/Z-estimation theory (see e.g.~\cite[Chapter 3]{Newey1994} and \cite[Theorems 5.23 and 5.41]{Vaart1996}).
	\begin{prop}[CLT for MPLE]\label{prop:CLTMPLE}
		Suppose that $\ms\sim\mathbb{P}_{\theta_0}$ where $\mathbb{P}_{\theta}$ is compactly supported in $\R^N$ (the support is free of $\theta$). Each $f_i(\cdot)$ is twice differentiable with continuous derivatives. We assume that $\theta_0$ belongs to the interior of the parameter space $\Theta$ and $\htmp$ as in \eqref{eq:MPLE} exists. We assume the following conditions: 
		\begin{itemize}
			\item[(A1)] For any $r_N\to 0$, we have: 
			$$\sup_{\theta:\lVert \theta-\theta_0\rVert\le r_N}\bigg|\frac{1}{N}\sum_{i=1}^N \nabla^2 f_i(\theta)-\frac{1}{N}\sum_{i=1}^N \nabla^2 f_i(\theta_0)\bigg|\overset{\mathbb{P}_{\theta_0}}{\to} 0.$$
			Further $(N^{-1}\sum_{i=1}^N \nabla^2 f_i(\theta_0))^{-1}=O_{\mathbb{P}_{\theta_0}}(1)$.
			\item[(A2)] There exists invertible $\Sigma_N(\theta_0)\in \R^{p\times p}$ (potentially random) such that $ \Sigma_N(\theta_0)=O_{\mathbb{P}_{\theta_0}}(1)$, such that 
			$$\Sigma_N(\theta_0)^{-1/2}\frac{1}{\sqrt{N}}\sum_{i=1}^N \nabla f_i(\theta_0)\overset{d}{\to} N(0,1).$$
			
			\item[(A3)] $\htmp\overset{\mathbb{P}_{\theta_0}}{\to} \theta_0$. 
		\end{itemize}
		Then we have:
		\begin{align}\label{eq:limtheo}
			\Sigma_N(\theta_0)^{-1/2}\left(\frac{1}{N}\sum_{i=1}^N \nabla^2 f_i(\theta_0)\right)\sqrt{N}(\htmp-\theta_0)\overset{w}{\longrightarrow} N(\mathbf{0}_p,\mathbf{I}_p).
		\end{align}
	\end{prop}
	
	Assumption (A1) above is standard and rather mild. It follows for example, if one can show that $N^{-1}\sum_{i=1}^N \lVert\nabla^3 f_i(\theta)\rVert$ is $O_{\mathbb{P}_{\theta_0}}(1)$ uniformly in a fixed neighborhood around $\theta_0$. As we have assumed compact support on $\mathbb{P}_{\theta_0}$, in many examples, the above third order tensor will turn out to be uniformly bounded. The main obstacle behind proving a CLT for $\sqrt{N}(\htmp-\theta_0)$ is to obtain the CLT in (A2) above. As discussed around \eqref{eq:concenter}, this is where the main result of this paper \cref{theo:CLTmain} plays a crucial role. Earlier attempts at CLTs for pseudolikelihood such as \cite{Comets1998,gaetan2004central,Somabha2021,Sanchayan2025} often restrict to Ising/Potts models with interactions on the $d$-dimensional lattice (for fixed $d$) or Curie-Weiss type interactions where all nodes are connected to all other nodes. On the other hand, the current paper provides CLTs akin to (A2) for a large class of general interactions in one go, without imposing restrictive sparsity or complete graph like assumptions. Moreover, since our CLT is not tied to a specific model, it can go much beyond Ising/Potts models; as illustrated by the exponential random graph model example in \cref{sec:ergm}.
	
	Assumption (A3) in \cref{prop:CLTMPLE} requires $\htmp$ to be consistent. Once again, one can state high level conditions for consistency leveraging classical results; see \cite[Section 2]{Newey1994} and \cite[Theorem 5.7]{Vaart1996}. Since the focus of this paper is on asymptotic normality, a detailed discussion on consistency is beyond the scope of the paper. For the sake of completion, we provide one sufficient condition for consistency which is easy to establish.
	
	\begin{prop}[Consistency of MPLE]\label{prop:ConsisMPLE}
		Suppose that $\ms\sim\mathbb{P}_{\theta_0}$ where $\mathbb{P}_{\theta}$ is compactly supported in $\R^N$ (the support is free of $\theta$). Each $f_i(\cdot)$ is twice differentiable with continuous derivatives. We assume that $\theta_0$ belongs to the interior of the parameter space $\Theta$ and $\htmp$ as in \eqref{eq:MPLE} exists. Let us consider two further assumptions:
		\begin{itemize}
			\item[(B1)] There exist a deterministic $\alpha>0$ such that $$\lmn\left(-\frac{1}{N}\sum_{i=1}^N \nabla^2 f_i(\theta)\right)\ge -\alpha$$
			for all $\theta\in\Theta$ and all large enough $N$. Here $\lmn$ denotes the minimum eigenvalue.
			\item[(B2)] Moreover $N^{-1}\sum_{i=1}^N \nabla f_i(\theta_0)\overset{\mathbb{P}_{\theta_0}}{\longrightarrow} 0$.
		\end{itemize}
	\end{prop}
	In other words, as long as the pseudolikelihood objective is strongly concave and the average of the gradient at $\theta_0$ converges to $0$ in probability, consistency follows. Going back to the Ising model \eqref{eq:twospIsingmagnet}, recall the pseudolikelihood equation from \eqref{eq:pseudois}. Note that the second derivative of the likelihood function is given by 
	$$B\mapsto -\frac{1}{N}\sum_{i=1}^N \sech^2(\beta \sum_j \A_N(i,j)\sigma_j+B).$$
	If we assume that the parameter space for $B$ is compact and $\A_N$ has bounded row sums (akin to \cref{as:cmean}), then condition (B1) follows immediately. Condition (B2) is a by-product 
	of \cref{theo:CLTmain}. This establishes consistency of $\hbm$. Generally speaking, there is no need to necessarily restrict to a compact parameter space, as we shall see in some of the examples later.
	
	\section{How to verify~\cref{as:cmean}?}\label{sec:howtover}
	In this Section, we will demonstrate how \cref{as:cmean} can be verified using simple analytic tools. To set things up, let us introduce an important notation: given any two sets $A,B\subseteq [N]$, such that $A\cap B=\phi$, and any function $\eta:\B^N\to\R$, define
	\begin{align}\label{eq:deldef}
		\De(\eta;A;B)=\sum_{D\subseteq B} (-1)^{|D|}\eta(\ms_{A\cup D})
	\end{align}
	where $\ms_{A\cup D}$ is defined as in~\eqref{eq:newsig}. By convention, we set $\De(\eta;A;\phi)=\eta(\ms_A)$. As an example, observe that $\De(\eta;j_1;\{j_2,j_3\})=\eta(\ms_{j_1})-\eta(\ms_{\{j_1,j_2\}})-\eta(\ms_{\{j_1,j_3\}})+\eta(\ms_{\{j_1,j_2,j_3\}})$. One way to interpret $\De(\eta;A;\phi)$ is a natural mixed partial discrete derivative of the function $\eta(\ms_A)$ along the coordinates in the set $D$. To put the definition of $\De(\cdot;\cdot;\cdot)$ into further perspective, observe that~\eqref{eq:cmean1} in~\cref{as:cmean} can be rewritten as:
	\begin{align}\label{eq:altform}
		\big|\De(t_{j_1};\tcS;\{j_2,\ldots ,j_k\})\big|\leq \Q_{N,k}(j_1,j_2,\ldots ,j_k).
	\end{align}
	We can reduce the problem of verifying \eqref{eq:altform} by making the following crucial observation --- namely that in many random fields the conditional means $t_1,\ldots ,t_N$ can often be written as smooth functions of simpler objects involving the vector $\ms$. As a concrete example, consider the $\pm 1$-valued Ising model described in \eqref{eq:twospIsing} with $b_0=0$ and $g(x)=x$. Through elementary computations, one can check that 
	\begin{align}\label{eq:tempising}
		t_j=\E[\si_j|\sigma_i,i\neq j]=\tanh(m_j),\quad \mbox{where} \quad m_j:=\sum_{i=1}^N \A_N(j,i)\sigma_i.
	\end{align}
	Note that the $m_j$s are linear in the coordinates of $\ms$ and the $\tanh(\cdot)$ is infinitely smooth with bounded derivatives. As controlling the discrete derivatives of the $m_j$s are significantly easier than working directly with the $t_j$s, one can ask the following natural question --- 
	\begin{center}
		\emph{Can one derive \eqref{eq:altform} using the simple structure of $m_j$s and the smoothness of $\tanh(\cdot)$?}
	\end{center}
	
	This phenomenon of expressing the conditional means as smooth transforms of simpler functions is not tied to the specific $\pm 1$-valued Ising model, but extends to many other settings involving higher order tensor interactions (see \eqref{eq:tensorsimpl}), exponential random graph models (see \eqref{eq:conergm}), etc. In the following result, we show this structural observation immediately yields a simple way to verify \cref{as:cmean} across the class of all such models. 
	
	We begin with some notation. Suppose $\{\tQ_{N,k}\}_{N\ge 1,k\ge 2}$ is a sequence of tensors of dimension $N\times N\times \ldots N$ ($k$-fold product), with non-negative entries, which is symmetric in its last $k-1$ coordinates. Given any such sequence and any $(j_1,\ldots ,j_k)\in [N]^k$, define the following recursively 
	\begin{align}\label{eq:newtensor}&\;\;\;\;\;\mathcal{R}[\tQ]_{N,k}(j_1,j_2,\ldots ,j_k)\nonumber \\ &:=\tQ_{N,k}(j_1,j_2,\ldots ,j_k)+\sum_{\substack{D\subseteq \{j_2,\ldots ,j_k\},\\ |D|\leq k-2,\ D\neq \phi}} \mathcal{R}[\tQ]_{N,1+|D|}(j_1,D)\tQ_{N,k-|D|}(\{j_1\},\{j_2,\ldots ,j_k\}\setminus D),\end{align}
	where, by convention, $\mathcal{R}[\tQ]_{N,2}(j_1,j_2)=\tQ_{N,2}(j_1,j_2)$ for $(j_1,j_2)\in [N]^2$. 
	
	\begin{theorem}\label{lem:smoothcont}
		Fix $k\geq 2$. Consider a set of functions $\{b_j(\ms)\}_{j\in [N]}$ such that 
		$$\max_{j\in [N]}\sup_{\ms\in\mathcal{B}^N}|b_j(\ms)|\le M, \quad \mbox{and}\quad \big|\De(b_{j_1};\tcS;\{j_2,\ldots ,j_k\})\big|\leq \tQ_{N,k}(j_1,j_2,\ldots ,j_k).$$
		for some $M>0$ and all $\tcS\subseteq[N]$ such that $\tcS\cap\{j_1,\ldots ,j_k\}=\phi$. 
		Let $f:[-M,M]\to \R$ such that $\sup_{|x|\leq M} |f^{(\ell)}(x)|\leq 1$ for all $0\leq \ell\leq k$, where $f^{(\ell)}(\cdot)$ denotes the $\ell$-th derivative of $f(\cdot)$ with $f^{(0)}(\cdot)=f(\cdot)$. 
		
		\begin{enumerate}
			\item The sequence of function compositions $f\circ b_1,\ldots ,f\circ b_N$ satisfies
			\begin{align}\label{eq:smoothext}
				|\De(f\circ b_{j_1};\tcS;\{j_2,\ldots ,j_k\})|\le C \mathcal{R}[\tQ]_{N,k}(j_1,j_2,\ldots ,j_k),
			\end{align}
			where $C>0$ depends only on $M$ and $k$. 
			
			\item $\mathcal{R}[\tQ]_{N,k}$ is symmetric in its last $k-1$ coordinates. If $\tQ_{N,k}$ satisfies \eqref{eq:cmean2}, then we have 
			\begin{align}\label{eq:newrowsum}
				\limsup_{N\to\infty}\max_{\ell\in [k]}\max_{j_{\ell\in [N]}} \sum_{(\{j_1,j_2,\ldots ,j_k\}\setminus j_{\ell})\in [N]^k} \mathcal{R}[\tQ]_{N,k}(j_1,j_2,\ldots ,j_k)< \infty.
			\end{align}
		\end{enumerate}
	\end{theorem}
	
	\cref{lem:smoothcont} says that if a sequence of functions $\{b_{j_1}(\ms),\ldots ,b_{j_N}(\ms)\}$ satisfies \eqref{eq:altform} ($t_{j_1}$ replaced by $b_{j_1}$) with some tensor sequence $\tQ$, then for any smooth $f(\cdot)$, the sequence $\{f(b_{j_1}(\ms)),\ldots ,f(b_{j_N}(\ms))\}$ satisfies \eqref{eq:altform} with the tensor sequence $\mathcal{R}[\tQ]$. Moreover, if $\tQ$ satisfies the maximum row summability condition in \eqref{eq:cmean2}, so does $\mathcal{R}[\tQ]$. The proof of \cref{lem:smoothcont} proceeds by showing a Fa\'{a} Di Bruno (see \cite{faa1855sullo} and \cref{cl:smoothclaim2}) type identity involving discrete derivatives of compositions of functions. 
	
	In terms of verifying \cref{as:cmean}, the main message of \cref{lem:smoothcont} is the following: 
	
	\begin{itemize}
		\item First show that the conditional means $t_j=\E [g(\si_i)|\si_j,j\neq i]=f(b_j(\ms))$ for some ``smooth" function $f(\cdot)$ and some simple transformations of $\ms$, say $b_j(\ms)$ (an example would be the $m_j$s in \eqref{eq:tempising} for the Ising model case). 
		\item Second, prove $b_j(\ms)$ satisfies \eqref{eq:altform} for some tensor sequence $\Q_{N,k}$ which has bounded maximum row sum in the sense of \eqref{eq:cmean2}. Typically the $b_j(\ms)$ sequence will be some polynomial of degree, say $v$, involving the observations $\ms$. This will immediately force \eqref{eq:altform} to hold for all $k>v$ by simply choosing the corresponding tensors to be identically $0$. The lower order discrete derivatives of such polynomial functions can be easily calculated and bounded, often using closed form expressions (as we shall explicitly demonstrate in the Ising case below).
		\item The final step is to apply \cref{lem:smoothcont} with the above functions $f(\cdot)$ and $b_j(\cdot)$, which will readily yield \cref{as:cmean}.
	\end{itemize}
	
	\paragraph{Application in Ising models.} In the Ising model case, by \eqref{eq:tempising}, recall that  $t_j=\tanh(m_j)$ where $m_j=\sum_{i=1}^N \A_N(i,j)\si_i$. As the $m_j$s are linear in the coordinates of $\ms$, we have 
	$$\De(m_{j_1};\tcS;\{j_2,\ldots ,j_k\})=0$$
	for all $k\ge 3$ and $\tcS$ such that $\tcS\cup \{j_2,\ldots ,j_k\}=\phi$. For $k=2$, we have 
	$$|\Delta(m_{j_1};\tcS;\{j_2\})|=\big|m_{j_1}^{\tcS}-m_{j_1}^{\tcS\cup\{j_2\}}\big|=\big|\A_N(j_1,j_2)\sigma_{j_2}\big|=\A_N(j_1,j_2).$$
	Combining the above observations, we note that 
	$$\big|\De(m_{j_1};\tcS;\{j_2,\ldots ,j_k\})\big|\le \tQ_{N,k}(j_1,\ldots ,j_k),$$
	where 
	$$\tQ_{N,k}(j_1,\ldots ,j_k) := \begin{cases} \A_N(j_1,j_2) & \mbox{if}\, k=2 \\ 0 & \mbox{if}\, k\ge 3\end{cases}.$$
	Therefore, if we assume that the matrix $\A_N$ has bounded row sums, then the sequence of tensors $\tQ_{N,k}$ will automatically have bounded row sums. Recall from above that $t_j=\tanh(m_j)$. As $\tanh(\cdot)$ has all derivatives bounded, by \cref{lem:smoothcont}, $(t_1,\ldots ,t_N)$ will satisfy \cref{as:bddrowsum} with 
	$$\Q_{N,k}(j_1,j_2,\ldots ,j_k)=\mathcal{R}[\tQ]_{N,k}(j_1,j_2,\ldots ,j_k)=\sum_{r=2}^{k} \mathcal{R}[\tQ]_{N,k-1}(\{j_1,j_2,\ldots ,j_k\}\setminus \{j_r\})\tQ_{N,2}(j_1,j_r).$$
	A simple induction then shows we can choose 
	$$\Q_{N,k}(j_1,\ldots ,j_k)=(k-1)\prod_{r=2}^k \A_N(j_1,j_r).$$
	The fact that $\Q_{N,k}$ as constructed above has a bounded row sum, follows from \cref{lem:smoothcont} itself, provided $\A_N$ has bounded row sums. 
	
	\begin{remark}[Broader implications]
		We emphasize that the above argument is not restricted to Ising models with pairwise interactions. It applies verbatim to many other graphical/network models. We provide two further illustrations involving Ising models with tensor interactions (see \eqref{eq:tensorsimpl}) and exponential random graph models (see \eqref{eq:conergm}).
	\end{remark}

	\section{Main Applications}\label{sec:examp}
	In this Section, we provide applications of our main results by deriving CLTs for conditionally centered spins and limit theory for a number of pseudolikelihood estimators. We will focus on the Ising model with pairwise interactions (in \cref{sec:ising}) and general higher order interactions (in \cref{sec:highising}). We will also apply our results to the popular exponential random graph model in \cref{sec:ergm}. 
	\subsection{Ising model with pairwise interactions}\label{sec:ising}
	The Ferromagnetic Ising model is a discrete/continuous Markov random field which was initially introduced as a mathematical model of Ferromagnetism in Statistical Physics, and has received extensive attention in Probability and Statistics~(c.f.~\cite{Basak2017,Berthet2019, Bresler2019, Chatterjee2007, Cha2011, Comets1991, deb2020fluctuations, deb2020detecting, AmirAndrea2010, eldan2018taming, Ellis1978, Gheissari2018, giardina2015annealed, jain2019mean, kabluchko2019fluctuations, liu2017log, Lowe2018, Mukherjee2018, mukherjee2019testing,  Ravikumar2010, Sly2014,Rados2019} and references therein).  Writing $\ms:=(\sigma_1,\cdots,\sigma_N)$, the Ising model with pairwise interactions can be described by the following sequence of probability measures:
	\begin{equation}\label{eq:model}
		\P\big\{\,d\ms\big\}\coloneqq\frac{1}{Z_N(\beta,B)}\exp\left(\frac{\beta}{2}(\ms)^{\top}\A_N\ms+B\sum\limits_{i=1}^N \sigma_i\right)\prod_{i=1}^N \vrh(\,d\sigma_i),
	\end{equation}
	where $\vrh$ is a non-degenerate probability measure, which is \emph{symmetric about $0$} and supported on $[-1,1]$ with the set $\{-1,1\}$ belonging to the support. 
	Here $\A_N$ is a $N\times N$ symmetric matrix with non-negative entries and zeroes on its diagonal, and $\beta\in \R$, $B\in \R$ are \emph{unknown parameters} often referred to in the Statistical Physics literature as \textit{inverse temperature} (Ferromagnetic or anti-Ferromagnetic depending on the sign of $\beta$) and \textit{external magnetic field} respectively. As the dependence on $\A_N$ in~\eqref{eq:model} is through a quadratic form, we can also assume without loss of generality that $\A_N$ is symmetric in its arguments. The factor $Z_N(\beta,B)$ is the normalizing constant/partition function of the model. The most common choice of the coupling matrix $\A_N$ is the adjacency matrix $\G_N$ of a graph on $N$ vertices, scaled by the average degree $\ud:=\frac{1}{N}\sum_{i,j=1}^N G_N(i,j)$.
	
	As mentioned in \eqref{eq:pseudois}, the asymptotic distribution of pseudolikelihood estimators under model \eqref{eq:model} is tied to the asymptotic behavior of $T_N$ in \eqref{eq:pivotstat} with $g(x)=x$. Therefore, in this section, we first present a general CLT for $T_N$ under model~\eqref{eq:model} which will be then leveraged to yield several new asymptotic properties of pseudolikelihood estimators. We begin with the following assumptions.
	
	\begin{Assumption}[Bounded row/column sum]\label{as:bddrowsum}
		$\A_N$ satisfies $$\limsup_{N\to\infty}\max_{1\leq i\leq N} \sum_{j=1}^N \A_N(i,j)< \infty.$$
	\end{Assumption}
	The above assumption \emph{does not impose} any sparsity assumptions. For instance, if $\A_N=\G_N/d_N$ where $\G_N$ is the adjacency matrix of a $d_N$-regular graph, \cref{as:bddrowsum} is automatically satisfied whether $d_N\to\infty$ (dense case) or $\sup_N d_N<\infty$ (sparse case). Therefore both the Curie-Weiss model \cite{Ellis1976,Qi-Man2019} ($\G_N$ is the complete graph) and the Ising model on the $d$-dimensional lattice \cite{Comets1998,gaetan2004central} satisfy this criteria. \cref{as:bddrowsum} will ensure that $T_N$ satisfies \cref{as:cmean} which is required to apply our main results.
	
	\begin{theorem}\label{theo:conmainis}
		Suppose $(\sigma_1,\ldots ,\sigma_N)$ is an observation drawn according to~\eqref{eq:model}. Recall the definitions of $U_N$ and $V_N$ from \eqref{eq:randvar}. 
		Then under Assumptions~\ref{as:coeff},~\ref{as:bddrowsum}~and~\eqref{eq:bddzero}, the following holds:
		
		\begin{equation}\label{eq:limmomentis}
			\frac{T_N}{\sqrt{(U_N+V_N) \vee a_N}}\overset{w}{\longrightarrow} \mathcal{N}(0,1), 
		\end{equation}
		for any strictly positive sequence $a_N\to 0$.
	\end{theorem}
	
	There are three key features of \cref{theo:conmainis} which will help uncover new asymptotic phenomena.
	
	\noindent (i). \textbf{No regularity restrictions}: Unlike some existing CLTs for $\sum_{i=1}^N \si_i$ in Ising models (see \cite{xu2023inference,deb2020fluctuations}) which assumes that the underlying graph $\G_N$ is ``approximately" regular, \cref{theo:conmainis} shows that no regularity assumption is needed to study asymptotic distribution of the conditionally centered statistic $T_N$. This flexibility will allow us to \emph{obtain the first joint CLTs for the pseudolikelihood estimator of $(\beta,B)$ in \cref{sec:jointpseudo}}.
	
	\noindent (ii). \textbf{No dense/sparse assumptions}: \cref{theo:conmainis} also does not impose any dense/sparse restrictions on the nature of interactions, unlike e.g. \cite{Comets1991,gaetan2004central} which requires sparse interactions. As a by-product, we are able to show (in \cref{sec:marginalpseudo}) that for dense regular graphs (much beyond the Curie-Weiss model), the asymptotic distribution of the pseudolikelihood estimator attains the Cramer-Rao information theoretic lower bound.
	
	\noindent (iii). \textbf{Anti-Ferromagnetic case $\beta<0$.} \cref{theo:conmainis} also allows for $\beta<0$. This helps us produce an example (in \cref{sec:mixturepseudo}) where the asymptotic distribution of the pseudolikelihood estimator for the magnetization parameter is not Gaussian but instead a Gaussian scale mixture. To the best of our knowledge, this phenomenon has not been observed before.
	
	\subsubsection{Joint pseudolikelihood CLTs for irregular graphons}\label{sec:jointpseudo}
	In this Section, we study the joint estimation of the inverse temperature and magnetization parameters, $\beta$ and $B$, respectively, under model \eqref{eq:model}. From \cite{Comets1998,Chatterjee2007,bhattacharya2018inference}, it is known that under mild assumptions $\beta$ is estimable at a $\sqrt{N}$ rate if $B$ \emph{is known}, and similarly $B$ is estimable at a $\sqrt{N}$ rate if $\beta$ \emph{is known}. The joint estimation of $(\beta,B)$ has been studied most comprehensively in \cite{ghosal2018joint}. At a high level, they observe that 
	\begin{enumerate}
		\item $\sqrt{N}$ estimation of $(\beta,B)$ jointly is \textbf{possible} if $\A_N$ is \textbf{approximately irregular}.
		\item $\sqrt{N}$ estimation of $(\beta,B)$ jointly is \textbf{impossible} if $\A_N$ is \textbf{approximately regular}.
	\end{enumerate}
	Moreover, in case 1, \cite{ghosal2018joint} shows that the pseudolikelihood estimator (formally defined below) is indeed $\sqrt{N}$-consistent for $(\beta,B)$ jointly. However, to the best of our knowledge, no joint limit distribution theory for the pseudolikelihood has been established yet. The aim of this Section is to provide the first such result. To achieve this, we will adopt the framework from \cite{ghosal2018joint}.
	
	\begin{definition}[Parameter space]
		Let $\Theta\subset\R^2$ denote the set of all parameters $(\beta,B)$ such that $\beta>0, B\neq 0$. 
	\end{definition}
	
	Next we define the joint pseudolikelihood estimator. To wit, note that under model \eqref{eq:model}, we have: 
	\begin{align}\label{eq:condist}
		\P\{\,d\sigma_i|\sigma_j,j\neq i\} = \frac{\exp\left(\sigma_i\big(\beta m_i(\ms)+B\big)\right) \vrh(\,d\sigma_i)}{\left(\int\exp\left(y\big(\beta m_i(\ms)+B\big)\right)\vrh(\,dy)\right)},
	\end{align}
	where
	\begin{equation}\label{eq:avisdef} m_i\equiv m_i(\ms):=\sum_{j=1, j\neq i}^N \A_N(i,j)\si_j.\end{equation}
	In other words, the conditional distribution of $\si_i$ given $\{\si_j,\ j\neq i\}$ is a function of $m_i$. These $m_i$s defined above are usually referred to as \emph{local averages}. For very site $i$, they capture the average effect of the neighbors of the $i$-th observation. Weak limits, concentrations, and tail bounds for $m_i$s have been studied extensively in the literature (see~\cite{gheissari2019ising,Chatterjee2007,deb2020fluctuations,deb2020detecting,bhattacharya2025sharp,bhattacharya2023gibbs}). Based on \eqref{eq:condist}, we note that 
	\begin{equation}\label{eq:candef1}\E[\si_i|\sigma_j,\ j\neq i]=\frac{\int y\exp\left(y\big(\beta m_i+B\big)\right)\vrh(\,dy)}{\int\exp\left(y\big(\beta m_i+B\big)\right)\vrh(\,dy)}=\Xi_{1}'(\beta m_i + B),\,\, \mbox{where}\,\, \Xi(t):=\log\int \exp(ty)\vrh(\,dy).\end{equation}
	
	\begin{definition}[Joint pseudolikelihood estimator]\label{def:plestim}
		Consider the bivariate equation in $(\beta,B)$ given by 
		$$\begin{pmatrix} \sum_{i=1}^N m_i(\si_i-\Xi'(\beta m_i+B)) \\ \sum_{i=1}^N (\si_i-\Xi'(\beta m_i+B))\end{pmatrix} = \begin{pmatrix}
			0 \\ 0
		\end{pmatrix}.$$
		The above equation has a unique solution $(\hbem,\hbm)$ in $\Theta$ with probability tending to $1$ under model \eqref{eq:model} (see \cite[Theorem 1.7]{ghosal2018joint}).
	\end{definition}
	To study the limit distribution theory for $(\hbem,\hbm)$ with an explicit covariance matrix, we need some notion of convergence of the underlying matrix $\A_N$. We use the notion of convergence in cut norm which has been studied extensively in the probability and statistics literature (see~\cite{FriezeKannan1999,bc_lpi,borgs2018p,borgsdense1,borgsdense2}).
	\begin{definition}[Cut norm]\label{def:defirst}
		Let $L^1([0,1]^2)$ denote the space of all integrable functions $W$ on the unit square, 
		Let $\mathcal{W}$ be the space of all symmetric real-valued functions in $L^1([0,1]^2)$. Given two functions $W_1,W_2\in \mathcal{W}$, define the cut norm between $W_1, W_2$ by setting	$$d_\square(W_1,W_2):=\sup_{S,T}\Big|\int_{S\times T} \Big[W_1(x,y)-W_2(x,y)\Big]dx dy\Big|.$$
		In the above display, the supremum is taken over all measurable subsets $S,T$ of $[0,1]$.  
		
		\noindent Given a symmetric matrix $\Q_N$, define a function $W_{\Q_N}\in \mathcal{W}$ by setting
		\begin{align*}
			W_{\Q_N}(x,y)=& \Q_N(i,j)\text{ if }\lceil Nx\rceil =i, \lceil Ny\rceil =j.
		\end{align*}
		
		We will assume throughout the paper that 
		the sequence of matrices $\{N \A_N\}_{N\ge 1}$  converge in cut norm,~i.e. for some $W\in \mathcal{W}$,
		\begin{align}\label{eq:cut_con}
			d_{\square}(W_{N\A_N},W) \rightarrow 0.
		\end{align} 
	\end{definition}
	As an example, if $\A_N=\G_N/(N-1)$ where $\G_N$ is the adjacency matrix of a complete graph, then the limiting $W$ is the constant function $1$. We note that \eqref{eq:cut_con} is a standard assumption for analyzing models on dense graphs. In particular, if $\A_N$ is the scaled adjacency matrix of a sequence of dense graphs (with average degree of order $N$), it is known that \eqref{eq:cut_con} always holds along subsequences  (see \cite{lovaszszemeredi}). An important goal in the study of Gibbs measures is to characterize the limiting partition function $Z_N(\beta,B)$ (see \eqref{eq:model}) in terms of the limiting graphon $W$ (see e.g.~\cite{augeri2019transportation,Cha2016}). In particular, it can be shown (see \cite[Proposition 1.1]{bhattacharya2023gibbs}) that 
	\begin{align}\label{eq:optimprob}
		\frac{1}{N}Z_N(\beta,B)&\overset{N\to\infty}{\longrightarrow} \sup_{f:[0,1]\to [-1,1]} \bigg(\beta \int_{[0,1]^2} f(x) f(y) W(x,y)\,dx\,dy + B \int_{[0,1]} f(x)\, dx \nonumber \\ &\qquad \qquad \qquad - \int_{[0,1]} ((\Xi')^{-1}(f(x))f(x)-\Xi((\Xi')^{-1}(f(x))))\,dx\bigg).
	\end{align}
	In our main result, we show that the limiting distribution of $(\hbem,\hbm)$ can be characterized in terms of the optimizers of \eqref{eq:optimprob}. As mentioned earlier, by \cite[Theorem 1.11]{ghosal2018joint}, $\sqrt{N}$ convergence of $(\hbem,\hbm)$ requires the limiting $W$ to satisfy an irregularity condition, which we first state below. 
	\begin{Assumption}[Irregular graphon]\label{as:irregraph}
		$W\in\mathcal{W}$ is said to be an irregular graphon if 
		\begin{align}\label{eq:irreg}
			\int_{x\in[0,1]} \left(\int_{y\in [0,1]} W(x,y)\,dy - \int_{x,y\in [0,1]^2} W(x,y)\,dx\,dy\right)^2\,dx>0.
		\end{align}
		In other words, the row integrals of $W$ are non-constant.
	\end{Assumption}
	We are now in position to state the main result of this section. 
	\begin{theorem}\label{thm:jointCLT}
		Suppose $\A_N$ satisfies \cref{as:bddrowsum} and \eqref{eq:cut_con} for some irregular graphon $W$ in the sense of \cref{as:irregraph}. For any $f:[0,1]\to [-1,1]$, define the following matrices: 
		\begin{align}\label{eq:amat}
			\mca_f:=\begin{pmatrix} \int_{[0,1]} f^2(x)\Xi''(\beta f(x)+B)\,dx & \int_{[0,1]} f(x)\Xi''(\beta f(x)+B)\,dx \\ \int_{[0,1]} f(x)\Xi''(\beta f(x)+B)\,dx & \int_{[0,1]} \Xi''(\beta f(x)+B)\,dx\end{pmatrix},
		\end{align}
		and $\mcb_f$ where 
		\begin{equation}
			\begin{aligned}\label{eq:bmat}
				\mcb_f(1,1)&:=\int_{x\in [0,1]} f(x)\Xi''(\beta f(x)+B)\left(f(x)-\beta \int_{y\in [0,1]} f(y)\Xi''(\beta f(y)+B)W(x,y)\,dy\right)\,dx\\
				\mcb_f(1,2)&:=\mcb_f(2,1):=\int_{x\in [0,1]}f(x)\Xi''(\beta f(x)+B)\left(1-\beta\int_{y\in [0,1]}\Xi''(\beta f(y) + B)W(x,y)\,dy\right)\,dx \\ 
				\mcb_f(2,2)&:=\int_{x\in [0,1]}\Xi''(\beta f(x)+B)\left(1-\beta \int_{y\in [0,1]} \Xi''(\beta f(y)+B)W(x,y)\,dy\right)\,dx.
			\end{aligned}
		\end{equation}
		
		Assume that the optimization problem in \eqref{eq:optimprob} has an almost everywhere unique solution $\fs$. Then $\mca_{\fs}$ is invertible and  
		$$\sqrt{N}\begin{pmatrix} \hbem-\beta \\ \hbm-B\end{pmatrix}\overset{w}{\longrightarrow} N\left(\begin{pmatrix} 0 \\ 0\end{pmatrix}, \mca_{\fs}^{-1}\mcb_{\fs}\mca_{\fs}^{-1}\right).$$
	\end{theorem} 
	To the best of our knowledge, \cref{thm:jointCLT} provides the first joint CLT for estimating $(\beta,B)$. A sufficient condition for unique solutions to the optimization problem in \eqref{eq:optimprob} is to assume $\vrh$ is strictly log-concave or equivalently $B$ is large enough (see \cite[Theorem 2.5]{lacker2024mean} and \cite[Lemma 25]{mukherjee2021variational}).
	\subsubsection{Marginal pseudolikelihood CLTs in the Mean-Field regime}\label{sec:marginalpseudo}
	
	As mentioned in \cref{sec:jointpseudo}, when $\A_N$ is ``approximately regular", joint $\sqrt{N}$ estimation of $(\beta,B)$ is no longer possible. However, given one parameter, the other can still be estimated at a $\sqrt{N}$ rate; see \cite{Comets1991,Chatterjee2007,bhattacharya2018inference}. To the best of our knowledge, the CLT for $\hbem$ (respectively $\hbm$) when $B$ (respectively $\beta$) is known, has only been established for the Curie-Weiss model (see \cite[Theorem 1.4]{Comets1991}) and when $\A_N$ is the scaled adjacency matrix of an Erd\H{o}s-R\'enyi graph (see \cite[Theorem 3.1]{Somabha2021}) under light sparsity. The goal of this Section is to complement these existing results by showing universal CLTs for $\hbem$ and $\hbm$ when the other parameter is known, for any sequence of dense regular graphs. Let us first formalize the notion of approximate regularity and denseness of $\A_N$. 
	
	\begin{Assumption}[Approximately regular matrices]\label{as:approxmat} 
		We define an approximately regular matrix  $\A_N$ as one that has non-negative entries, is symmetric and satisfies: 
		
		\begin{equation}\label{eq:asnrsm}
			\lambda_1(\A_N)\overset{N\to\infty}{\longrightarrow} 1,\quad \frac{1}{N}\sum_{i=1}^N \delta_{R_i}\to 1,\ \mbox{where}\ R_i:=\sum_{j=1}^N \A_N(i,j),
		\end{equation}
		where $\lambda_1(\A_N)\geq \lambda_2(\A_N)\geq \ldots \geq \lambda_N(\A_N)$ are the $N$ eigenvalues of $\A_N$ arranged in descending order.
	\end{Assumption}
	
	\begin{Assumption}[Mean-field/denseness  condition]\label{as:mfield}
		The Frobenius norm of $\A_N$ satisfies $$\fa:=\sum_{i,j} \A_N(i,j)^2=o(N).$$
		When the coupling matrix $\A_N$ is the adjacency matrix of a graph $G_N$ on $N$ vertices, scaled by the average degree $\ud:=\frac{1}{N}\sum_{i,j=1}^N G_N(i,j)$, then $\fa=N/\ud$. Therefore in that case, \cref{as:mfield} is equivalent to assuming that $\ud\to\infty$, which implies the graph is dense. 
	\end{Assumption}
	
	Assumptions \ref{as:approxmat} and \ref{as:mfield} cover popularly studied examples in the literature such as scaled adjacency matrices of random/deterministic regular graphs, Erd\H{o}s-R\'enyi graphs, balanced stochastic block models, among others. When $\A_N$ is the scaled adjacency matrix of a graph, the condition $\lambda_1(\A_N)\to 1$ can be dropped as it is  implied by the bounded row sum condition in \cref{as:bddrowsum}, the Mean-Field condition \cref{as:mfield}, and the empirical row sum condition $N^{-1}\sum_{i=1}^N \delta_{R_i}\to 1$ in \eqref{eq:asnrsm}.
	
	In order to present our results when $\A_N$ is approximately regular and dense, we need certain prerequisites. Recall the definition of $\Xi(\cdot)$ from~\eqref{eq:candef1}.  
	\begin{definition}
		Recall the definition of $\Xi$ from \eqref{eq:candef1}. Let \begin{align*}\Theta_{11}:=\{(r,0):0\leq r\leq (\Xi''(0))^{-1}\},&\quad  \Theta_{12}:=\{(r,s):r\geq 0, s\ne 0\}, &\quad \Theta_2:=\{(r,0):r>(\Xi''(0))^{-1}\}.\end{align*} 
		It is easy to check that $\Xi''(0)$ is the variance under $\vrh$, i.e., $\Xi''(0)=\int x^2\,d\vrh(x)$. Finally, let $\Theta_1:=\Theta_{11}\cup \Theta_{12}$. We will refer to $\Theta_1$ as the uniqueness regime and $\Theta_2$ as the non uniqueness regime. The point $\{\Xi''(0)^{-1},0\}$ is called the critical point. 
		The names of the different regimes are motivated by the next lemma which is a slight modification of \cite[Lemma 1.7]{bhattacharya2023gibbs}.
	\end{definition}
	\begin{lemma}\label{lem:fixsol}
		The function $\Xi'(\cdot)$ is one-to-one. For $x$ in the domain of $(\Xi')^{-1}$, consider the function 
		\begin{equation}\label{eq:regobj}
			\phi(x):=\frac{r x^2}{2}+sx-x(\Xi')^{-1}(x)+C((\Xi')^{-1}(x)).
		\end{equation}
		Assume that \begin{equation}\label{eq:secasn} \Xi'''(x)\le 0 \qquad \mbox{for all}\,\,\, x>0 \qquad \mbox{and}\qquad  \Xi'''(x)\ge 0 \qquad \mbox{for all}\,\,\, x<0.
		\end{equation} 
		Then the following conclusions hold:
		\begin{enumerate}
			\item[(a)] If $(r,s)\in \Theta_{11}$, then $\phi(\cdot)$ has a unique maximizer at $t_{\vrh}=0$.
			\item[(b)] If $(r,s)\in \Theta_{12}$, then $\phi(\cdot)$ has a unique maximizer $t_{\vrh}$ with the same sign as that of $s$. Further, $t_{\vrh}=\Xi'(r t_{\vrh}+s)$ and $r\Xi''(r t_{\vrh}+s)<1$.
			\item[(c)] If $(r,s)\in \Theta_2$, then $\phi(\cdot)$ has two maximizers $\pm t_{\vrh}$, where $t_{\vrh}>0$, $t_{\vrh}=\Xi'(r t_{\vrh})$ and $r\Xi''(r t_{\vrh})<1$.
		\end{enumerate} 
		
	\end{lemma} 
	We will use $t_{\vrh}$ as defined in the above lemma throughout the paper, noting that $t_{\vrh}$ also depends on $(r,s)$ which we hide in the notation for simplicity. A remark is in order. 
	
	\begin{remark}[Necessity of~\eqref{eq:secasn}]\label{rem:monotonejust}
		It is easy to construct examples of $\vrh$ for which~\eqref{eq:secasn} does not hold and $\phi(\cdot)$ does not have a unique maximizer for all $(r,s)\in \Theta_{11}$, see e.g.,~\cite[Equation 1.5]{Ellis1976}. In fact, it is not hard to check that Assumption~\eqref{eq:secasn} is a consequence of the celebrated \textit{GHS inequality} (see~\cite{Ginibre1970,Griffiths1970,Newman1975}). Sufficient conditions on $\vrh$ for the GHS inequality and consequently~\eqref{eq:secasn} to hold can be seen in~\cite[Theorem 1.2]{Ellis1976}. Note that when $\rho$ is the Rademacher distribution (which corresponds to the canonical binary Ising model), condition \eqref{eq:secasn} holds.
	\end{remark}
	
	Next we present a CLT result on $T_N$ (see \eqref{eq:pivotstat}) with $g(x)=x$, which forms the backbone of the asymptotics for the pseudolikelihood estimators to follow.
	
	\begin{theorem}[General CLT for regular graphs]\label{theo:regclt}
		Recall the definition of $m_i$ from \eqref{eq:avisdef}. Suppose that \eqref{eq:secasn} and Assumptions \ref{as:coeff}, \ref{as:bddrowsum}, \ref{as:approxmat}, and \ref{as:mfield} hold. Also let $\upsilon_1>0$, $\upsilon_2\in \R$ be constants such that $N^{-1}\sum_{i=1}^N c_i^2\to\upsilon_1$ and $N^{-1}(\mc)^{\top}\A_N\mc\to\upsilon_2$. Then the following conclusion holds for $(\beta,B)\in \Rg\times\R$:
		$$\frac{1}{\sqrt{N}}\sum_{i=1}^N c_i(\si_i-\Xi'(\beta m_i+B))\overset{w}{\longrightarrow}\mathcal{N}\big(0,\Xi''(\beta t_{\vrh}+B)(\upsilon_1-\beta\upsilon_2 \Xi''(\beta t_{\vrh}+B))\big).$$
	\end{theorem}
	
	\cref{theo:regclt} has some interesting implications with regards to two features; namely \emph{universality across a large class of $\A_N$} and \emph{lack of phase transitions}. We discuss them in the following remarks.
	
	\begin{remark}[Universality of fluctuations]\label{rem:comparewithin}
		Suppose that $\A_N$ is the adjacency matrix of a $d_N$-regular graph. When $\mc=\mathbf{1}$, we have  $\upsilon_1=\upsilon_2=1$. Therefore \cref{theo:regclt} implies that $$\frac{1}{\sqrt{N}}\sum_{i=1}^N (\si_i-\Xi'(\beta m_i+B))\overset{w}{\longrightarrow} N(0,\Xi''(\beta t_{\vrh}+B)(1-\beta \Xi''(\beta t_{\vrh}+B))),$$
		whenever $d_N\to\infty$. Therefore the conditionally centered fluctuations exhibit a universal behavior across all such $\A_N$. 
		On the other hand, in the recent paper \cite{deb2020fluctuations}, the authors show the universal asymptotics of the \emph{unconditionally centered} average of spins when $\vrh$ is the counting measure on $\{-1,1\}$ provided $d_N\gg \sqrt{N}$. In fact, the $\sqrt{N}$ threshold there is tight as there exists counterexamples when $d_N\sim\sqrt{N}$ where the universality breaks (see \cite[Example 1.3]{deb2020fluctuations}, \cite{mukherjee2023statistics}). Therefore, \cref{theo:regclt} shows that universality in the conditionally centered fluctuations extends further (up to $d_N\to\infty$) than those for unconditionally centered ones (which stop at $d_n\gg \sqrt{N})$. 
	\end{remark}
	
	\begin{remark}[Non-degeneracy in~\cref{theo:regclt} and (no) phase transition at critical point]\label{rem:nondeg}
		In special cases \cref{theo:regclt} does exhibit degenerate behavior. When $\mc=\mathbf{1}$ as in the previous remark, the limiting variance in \cref{theo:regclt} is $0$ at the critical point $(\beta,B)=\{(\Xi''(0))^{-1},0\}$. In this example, one can show that $N^{-1/4}\sum_{i=1}^N (\si_i-\Xi_i'(\beta m_i+B))$ has a non-degenerate limit. This phase transition behavior however disappears for other choices of $\mc$, such as when $\mc$ is a contrast vector. In particular if $\sum_{i=1}^N c_i=O(N)$, $\lVert \mc\rVert=\sqrt{N}$ and $\max_{i\ge 2} |\lambda_i(\A_N)|=o(1)$ (as is the case with Erd\H{o}s-R\'{e}nyi graphs), \cref{theo:regclt} implies that 
		$$\frac{1}{\sqrt{N}}\sum_{i=1}^N c_i(\si_i-\Xi'(\beta m_i+B))\overset{w}{\longrightarrow} N(0,\Xi''(\beta t_{\vrh}+B)).$$
		Note that the limiting variance is now always strictly positive, even at the critical point. Therefore, under such configurations, the phase transition behavior is no longer observed. More generally, there is no phase transition whenever $\upsilon_2<\upsilon_1$ in \cref{theo:regclt}. 
	\end{remark}
	
	We now move on to the implications of \cref{theo:regclt} in the asymptotic distribution of the pseudolikelihood estimators. 
	
	\paragraph{Limit theory for pseudolikelihood estimators.} We start off with the case when $B$ is known but $\beta$ is unknown. In that case, following \cref{def:plestim}, $\hbem$ is defined as the non-negative solution in $\beta$ of the equation 
	$$\sum_{i=1}^N m_i(\si_i-\Xi'(\beta m_i+B))=0$$
	The following result characterizes the limit of $\hbem$. 
	
	\begin{theorem}\label{thm:margbeta}
		Suppose that \eqref{eq:secasn} and Assumptions \ref{as:bddrowsum}, \ref{as:approxmat}, and \ref{as:mfield} hold. Then provided $\beta>0$, $B\neq 0$, we have: 
		$$\sqrt{N}(\hbem-\beta)\overset{w}{\longrightarrow} N\left(0,\frac{1-\beta\Xi''(\beta t_{\vrh}+B)}{t_{\vrh}^2\Xi''(\beta t_{\vrh}+B)}\right).$$
	\end{theorem}
	Note that the assumption $B\neq 0$ ensures by \cref{lem:fixsol} that $t_{\rho}\neq 0$ and $1-\beta\Xi''(\beta t_{\vrh}+B)>0$. Therefore the limiting distribution in \cref{thm:margbeta} is non-degenerate. 
	
	By \cite[Remark 2.15]{Somabha2021}, it is easy to check that the asymptotic variance matches the asymptotic Fisher information when $\vrh$ is the Rademacher distribution. Therefore, an interesting feature of \cref{thm:margbeta} is that it shows the $\hbem$ is 
	
	(a) Information theoretically efficient at least in the binary Ising model case, and 
	
	(b) The efficiency holds in the entirety of the Mean-Field regime $\ud\to\infty$ \emph{without restricting specifically to Curie-Weiss models}.
	
	We note that the same asymptotic variance was proved for the maximum likelihood estimator (MLE) for the Curie-Weiss model in \cite[Theorem 1.4]{Comets1991}. 
	
	Next we move on to the case where $\beta$ is known but $B$ is unknown. In that case, following \cref{def:plestim}, $\hbm$ is defined as the  solution in $B$ of the equation 
	$$\sum_{i=1}^N (\si_i-\Xi'(\beta m_i+B))=0$$
	The following result characterizes the limit of $\hbm$. 
	
	\begin{theorem}\label{thm:margB}
		Suppose that \eqref{eq:secasn} and Assumptions \ref{as:bddrowsum}, \ref{as:approxmat}, and \ref{as:mfield} hold. Then provided $B\neq 0$, we have: 
		$$\sqrt{N}(\hbm-B)\overset{w}{\longrightarrow} N\left(0,\frac{1-\beta\Xi''(\beta t_{\vrh}+B)}{\Xi''(\beta t_{\vrh}+B)}\right).$$
	\end{theorem}
	
	The implications of \cref{thm:margB} are similar to those of \cref{thm:margbeta}. We once again observe that the pseudolikelihood estimator $\hbm$ is information theoretically efficient. This holds for the entire Mean-Field regime $\ud\to\infty$. 
	
	To conceptualize the full scope of Theorems \ref{theo:regclt}, \ref{thm:margbeta}, and \ref{thm:margB}, we conclude the Section by providing a set of examples featuring popular choices of $\A_N$ on which our results apply. 
	\begin{enumerate}
		\item[(a)]{\bf Regular graphs (deterministic and random):}
		Let $\G_N$ be a $d_N$ regular graph and set $\A_N:=\G_N/d_N$. Then Theorems \ref{theo:regclt}, \ref{thm:margbeta}, and \ref{thm:margB} apply as soon as $d_N\to\infty$.
		\item[(b)]{\bf Erd\H{o}s-R\'enyi graphs:}
		Suppose $\G_N\sim \mathcal{G}(N,p_N)$ be the symmetric Erd\H{o}s R\'enyi random graph with $0<p_N\leq 1$. Define $\A_N(i,j):= \frac{1}{(N-1)p_N}G_N(i,j)$. Then Theorems \ref{theo:regclt}, \ref{thm:margbeta}, and \ref{thm:margB} apply provided $p_N\gg (\log{N})/N$.
		
		\item[(c)]{\bf Balanced stochastic block model: }
		Suppose $\G_N$ is a stochastic block model with $2$ communities of size $N/2$ (assume $N$ is even). Let the probability of an edge within the community be $a_N$, and across communities be $b_N$. This is the well known stochastic block model, which has received considerable attention in Probability, Statistics and Machine Learning (see \cite{deshpande2018contextual,liu2017log,mossel2012stochastic} and references within). If we take $\A_N:=\frac{2}{N(a_N+b_N)}\G_N$, then Theorems \ref{theo:regclt}, \ref{thm:margbeta}, and \ref{thm:margB} hold if $a_N+b_N\gg (\log{N})/N$.
		\item[(d)]{\bf Sparse regular graphons:} 
		Suppose that $W$ be a symmetric measurable function from $[0,1]^2$ to $[0,1]$, such that $\int_{[0,1]}W(x,y)dy=a>0$ for all $x\in [0,1]$. Also let $(U_1,\cdots,U_N)\stackrel{i.i.d.}{\sim}U(0,1)$. For $\gamma\in (0,1]$, let \[\{G_N(i,j)\}_{1\le i<j\le N}\stackrel{i.i.d.}{\sim}Bern\bigg(\frac{W(U_i,U_j)}{N^\gamma}\bigg).\] Such random graph models have been studied in the literature under the name $W$ random graphons (c.f.~\cite{BorgsdenseI,BorgsdenseII,BorgsLPI,BorgsLPII,Lovasz2012}). In this case for the choice $\A_N=\frac{1}{Np_N}G_N$ with  $p_N=a N^{-\gamma}$, Theorems \ref{theo:regclt}, \ref{thm:jointCLT}, and \ref{thm:margB} hold as soon as $\gamma<1$. 
		\item[(e)]{\bf Wigner matrices:}
		This example demonstrates that our techniques apply to examples beyond scaled adjacency matrices. To wit, let $\A_N$ be a Wigner matrix with its entries $\{\A_N(i,j),1\le i<j\le N\}$ i.i.d. from a distribution $F$ scaled by $N\mu$, where $F$ is a distribution on non-negative reals with finite exponential moment and mean $\mu>0$. In this case too, Theorems \ref{theo:regclt}, \ref{thm:margbeta}, and \ref{thm:margB} continue to hold.
	\end{enumerate}
	
	\subsubsection{A Gaussian scale mixture example}\label{sec:mixturepseudo}
	
	The studentized CLTs for $T_N$ (see \cref{theo:CLTmain}) and the pseudolikelihood estimator (see \cref{prop:CLTMPLE}) can also lead to limit distributions which are mixtures of multiple Gaussian components. This can happen when the optimization problem in \eqref{eq:optimprob} (or \eqref{eq:regobj}) admits multiple optimizers. The following result provides an example:
	\begin{prop}\label{prop:bipartg}
		Suppose $\A_N$ is the adjacency matrix of a regular, complete bipartite graph scaled by $N/2$, where the two communities are labeled as $\{1,2,\ldots ,N/2\}$ and $\{1+N/2+1,2+N/2,\ldots ,N\}$ (assume $N$ is even). Let $\mc$ be such that $c_i=1$ or $0$ depending on whether $i\leq N/2$ or $i>N/2$. Suppose that $B>0$ and \eqref{eq:secasn} holds. Then there exists $\beta_0<0$ (depending on $B$), such that for any $\beta<\beta_0$, there exists $t_1$ and $t_2$ (depending on $\beta$, $B$) which are of opposite signs and $\Xi''(\beta t_1 + B)\neq \Xi''(\beta t_2 + B)$, such that
		$$\frac{1}{\sqrt{N}}\sum_{i=1}^N c_i(\si_i-\Xi'(\beta m_i+B))\overset{w}{\longrightarrow} \frac{1}{\sqrt{2}}\left(\xi\times \sqrt{\Xi''(\beta t_1+B)} G_1+ (1-\xi)\times \sqrt{\Xi''(\beta t_2  + B)} G_2\right),$$
		where $\xi$ is a Bernoulli random variable with mean $1/2$, independent of $G_1,G_2\overset{i.i.d.}{\sim}\mathcal{N}(0,1)$.
	\end{prop}
	
	The main intuition behind getting the two component mixture in the limit is as follows. We first note that 
	$$m_i=\frac{2}{N}\sum_{j=N/2+1}^N \si_j, \; \mbox{for}\; 1\le i\le N/2, \qquad m_i=\frac{2}{N}\sum_{j=1}^{N/2} \si_j, \; \mbox{for}\; N/2+1\le i\le N.$$
	Therefore the $m_i$s have a block constant structure across the two communities. This can be leveraged to show that the empirical measure on the $m_i$s over $1\le i\le N/2$ converges to a two-point mixture provided $\beta$ is negative with a large enough absolute value. As a by-product, there will exist $t_1$ and $t_2$ of opposite signs such that 
	$$U_N\approx\frac{1}{N}\sum_{i=1}^{N/2} \Xi''(\beta m_i+B) \overset{w}{\longrightarrow} \frac{1}{2}\delta_{\frac{1}{2}\Xi''(\beta t_1+B)} + \frac{1}{2}\delta_{\frac{1}{2}\Xi''(\beta t_2+B)}.$$
	Moreover it can be show that 
	$$V_N\approx \frac{1}{N}\sum_{i\neq j} c_i c_j\A_N(i,j)\Xi''(\beta m_i+B)\Xi''(\beta m_j+B).$$
	As $c_i c_j \A_N(i,j)=0$ for all $i,j$, it follows that $V_N\approx 0$. Therefore, $U_N+V_N\overset{w}{\longrightarrow} \frac{1}{2}\delta_{(1/2)\Xi''(\beta t_1+B)} + \frac{1}{2}\delta_{(1/2)\Xi''(\beta t_2+B)}$. By the joint convergence of $T_N,U_N,V_N$ in \cref{theo:CLTmain}, the conclusion in \cref{prop:bipartg} will follow. In the same spirit as \cref{prop:bipartg}, we can also construct an example where a pseudolikelihood estimator would have a two component Gaussian scale mixture limit. To achieve this consider a slight modification of \eqref{eq:model} given by 
	\begin{equation}\label{eq:modelbip}
		\PB\big\{\,d\ms\big\}\coloneqq\frac{1}{Z_N(\beta,h,B)}\exp\left(\frac{\beta}{2}(\ms)^{\top}\A_N\ms+ h \sum_{i=1}^{N}c_i \si_i + B\sum\limits_{i=1}^N \sigma_i\right)\prod_{i=1}^N \vrh(\,d\sigma_i),
	\end{equation}
	where $\beta$ is known but $(h,B)$ are unknown, $\A_N$ is the scaled adjacency matrix of a complete bipartite graph, and $c_i$s are defined as in \cref{prop:bipartg}. Following \cref{def:plestim}, the pseudolikelihood estimator is given by $(\hh,\hbm)$ which satisfies the equations 
	$$\begin{pmatrix} \sum_{i=1}^{N/2} (\si_i-\Xi'(\beta m_i+B+h)) \\ \sum_{i=1}^{N/2} (\si_i-\Xi'(\beta m_i+B+h)) + \sum_{i=N/2+1}^N (\si_i-\Xi'(\beta m_i+B))\end{pmatrix}=\begin{pmatrix} 0 \\ 0 \end{pmatrix},$$
	over some \emph{compact set} $K\subseteq \R^2$. The assumption of compactness is made for technical convenience to ensure consistency of $(\hh,\hbm)$.
	
	\begin{prop}\label{prop:bipartpseudo}
		Consider the same setup as in \cref{prop:bipartg}. Assume that $h=0,B\neq 0$ and the point $(0,B)\in K$. Recall $t_1$ and $t_2$ from \cref{prop:bipartg}. Set $\tto:=\Xi''(\beta t_1+B)$ and $\ttt:=\Xi''(\beta t_2+B)$. Define 
		$$H_1:=\begin{pmatrix}\frac{1}{2}\tto & \frac{1}{2} \tto \\ \frac{1}{2} \tto & \frac{1}{2} (\tto+\ttt)\end{pmatrix}^{-1}\begin{pmatrix}\frac{1}{2}\tto & \frac{1}{2}(\tto-\beta \tto \ttt) \\ \frac{1}{2}(\tto-\beta \tto\ttt) & \frac{1}{2}(\tto+\ttt)-\beta \tto \ttt\end{pmatrix}\begin{pmatrix}\frac{1}{2}\tto & \frac{1}{2} \tto \\ \frac{1}{2} \tto & \frac{1}{2} (\tto+\ttt)\end{pmatrix}^{-1}.$$
		Define $H_2$ similarly by switching the roles of $\tto$ and $\ttt$. Then, under \eqref{eq:modelbip}, we have: 
		$$\sqrt{N}\begin{pmatrix} \hh \\ \hbm - B\end{pmatrix} \overset{w}{\longrightarrow} \xi H_1^{1/2} G_1 + (1-\xi) H_2^{1/2} G_2,$$
		where $\xi$ is Rademacher, $G_1,G_2$ are bivariate standard normals. Also $\xi,G_1,G_2$ are independent of each other.
	\end{prop}
	
	To the best of our knowledge, a scaled Gaussian mixture limit of pseudolikelihood estimators in dense graphs has not been observed before. We do believe that more detailed exploration of such phenomenon is an interesting question for future research.
	
	\subsection{Extensions to higher order interactions}\label{sec:highising}
	
	Modern network data often features complex interactions across agents thereby necessitating the development of Ising models with higher ($>2$) order interactions; see e.g.,~\cite{Somabha2021,mukherjee2022estimation,bhattacharya2023gibbs,bhattacharya2024ldp,vanhecke2021solving,sasakura2014ising}. In this Section, we adopt a particular variant of a tensor Ising model (adopted from \cite{bhattacharya2023gibbs}). Let $H=(V(H),E(H))$ be a finite graph with $v:=|V(H)|\ge 2$ vertices labeled $\{1,2,\ldots ,v\}$. Writing $\ms:=(\sigma_1,\cdots,\sigma_N)$, the Ising model can be described by the following sequence of probability measures:
	\begin{equation}\label{eq:modelhighis}
		\P\big\{\,d\ms\big\}\coloneqq\frac{1}{Z_N(\beta,B)}\exp\left(\frac{\beta N}{v}\mathbb{U}_N(\ms)+B\sum\limits_{i=1}^N \sigma_i\right)\prod_{i=1}^N \vrh(\,d\sigma_i),
	\end{equation}
	where the Hamiltonian $\mathbb{U}_N(\ms)$ is a multilinear form, defined by
	\begin{align}\label{eq:U}
		\mathbb{U}_N(\ms):=\frac{1}{N^v}\sum_{(i_1,\ldots,i_v)\in \mathcal{S}(N,v) }\Big(\prod_{a=1}^v \si_{i_a}\Big)\prod_{(a,b)\in E(H)}\A_N(i_a,i_b).
	\end{align}
	Here $\mathcal{S}(N,v)$ is the set of all distinct tuples from $[n]^v$ (so that $|\mathcal{S}(n,v)|=v!\binom{n}{v}$). In particular, if $H$ is an edge, then \eqref{eq:modelhighis} is exactly the same as \eqref{eq:model}. All the parameters $\beta,B,\A_N,\rho$ have the same default assumptions as in the Ising model with pairwise interactions (see \eqref{eq:model}). We reiterate them here for the convenience of the reader. Therefore $\vrh$ is a non-degenerate probability measure, which is \emph{symmetric about $0$} and supported on $[-1,1]$, with the set $\{-1,1\}$ belonging to the support. Further $\A_N$ is a $N\times N$ symmetric matrix with non-negative entries and zeroes on its diagonal, and $\beta\in \R$, $B\in\R$ are \emph{unknown parameters} often referred to in the Statistical Physics literature as \textit{inverse temperature} (Ferromagnetic or anti-Ferromagnetic depending on the sign of $\beta$) and \textit{external magnetic field} respectively. The factor $Z_N(\beta,B)$ is the normalizing constant/partition function of the model. 
	
	Limit distribution theory for the average magnetization $\sum_{i=1}^N \si_i$, coupled with asymptotic theory for the maximum likelihood/pseudolikelihood estimation of $\beta$ and $B$ (marginally) under model \eqref{eq:modelhighis}, has been studied when $\A_N$ is the scaled adjacency matrix of a \emph{complete graph}; see e.g.~\cite{mukherjee2021fluctuations,mukherjee2025MLE,Sanchayan2025}. We note that the proofs of these results heavily rely on the complete graph structure and do not generalize to more general graphs. In a separate line of research $\sqrt{N}$-estimation of $\beta$ and $B$ marginally has been studied under weaker assumptions in \cite{mukherjee2022estimation}. Joint $\sqrt{N}$-estimation of $(\beta,B)$ jointly has been studied in \cite{daskalakis2020logistic,mukherjee2024logistic} when $\A_N$ is the \emph{adjacency matrix of a bounded degree graph}. However, none of these proof techniques translate to explicit limit distribution theory for the proposed estimators of $\beta$ and $B$. Overall, we are not aware of any results in the literature that yield joint limit distribution theory for estimating $(\beta,B)$. The goal of this Section is to fill that void in the literature. A major strength of this paper is that our main distributional result \cref{theo:CLTmain} is relatively model agnostic, which helps us obtain inferential results under \eqref{eq:modelhighis} without imposing strong sparsity assumptions on the nature of the interaction (i.e., the matrix $\A_N$).  
	
	To state our main results, we introduce some preliminary notation. First given any matrix $\A_N$, define the symmetrized tensor 
	\begin{align}\label{eq:symatrix}
		\mathrm{Sym}[\A_N](i_1,\ldots ,i_v):=\frac{1}{v!}\sum_{\pi\in S_v}\prod_{(a,b)\in E(H)}\A_N(i_{\pi(a)},i_{\pi(b)}) 
	\end{align}
	for $(i_1,\ldots ,i_v)\in [N]^v$, where $S_v$ denotes the set of all permutations of $[v]$. In a similar vein, given a symmetric measurable function $W:[0,1]^2\to [0,1]$, define the symmetrized tensor 
	\begin{align}\label{eq:symfunc}
		\mathrm{Sym}[W](x_1,\ldots ,x_v):=\frac{1}{v!}\sum_{\pi\in S_v}\prod_{(a,b)\in E(H)} W(x_{\pi(a)},x_{\pi(b)})
	\end{align}
	for $(x_1,\ldots ,x_v)\in [0,1]^v$. 
	the \emph{local fields} (similar to \eqref{eq:avisdef}) as follows: 
	\begin{align}\label{eq:m_ihigh}
		m_i\equiv m_i(\ms):=\frac{1}{N^{v-1}}\sum_{(i_2,\ldots ,i_v)\in \cS(N,v,i)}\mathrm{Sym}[\A_N](i,i_2,\ldots ,i_v)\left(\prod_{a=2}^v \si_{i_a}\right), \quad \mbox{for} \,\, i\in [N],
	\end{align}
	where $\cS(n,v,i)$ denotes the set of all distinct tuples of $[N]^{v-1}$ such that none of the elements equal to $i$.  
	Direct computations reveal that 
	\begin{align}\label{eq:tensorsimpl}
		\E[\si_i|\si_j,j\neq i]=\Xi'(\beta m_i+B).
	\end{align}
	Therefore $\E[\si_i|\si_j,j\neq i]$ is a smooth transformation of the $m_i$s which are in turn product of monomials. Following the discussion in \cref{sec:howtover}, we can use \cref{lem:smoothcont} to establish \cref{as:cmean}. 
	Next we state an appropriate row-sum boundedness assumption that ensures \cref{as:cmean} holds. 
	
	\begin{Assumption}\label{as:highrowsum}
		The symmetrized tensor $\mathrm{Sym}[\A_N]$ satisfies 
		$$\limsup\limits_{N\to\infty}\max_{\ell\in [v]}\max_{i_{\ell}\in [N]} \sum_{(\{i_1,\ldots ,i_v\}\setminus \{i_{\ell}\})\in [N]^{v-1}} \mathrm{Sym}[\A_N](i_1,\ldots ,i_v)<\infty.$$
	\end{Assumption}
	The above assumption holds when $\A_N$ is the scaled adjacency matrix of a complete graph. It also holds when the complete graph is replaced by the Erd\H{o}s-R\'{e}nyi random graph $\mathcal{G}(N,p_N)$ with $p_N\equiv p\in (0,1)$ (fixed). It also holds for sparser Erd\H{o}s-R\'{e}nyi graphs depending on $H$. For example, if $H$ is a star graph then \cref{as:highrowsum} holds for $p_N\gg \log{N}/N$. On the other hand, if $H$ is the triangle graph, then \cref{as:highrowsum} holds if $p_N\gg \log{N}/\sqrt{N}$.  
	
	We now state a CLT for the conditionally centered statistic $N^{-1/2}\sum_{i=1}^N c_i(\si_i-\Xi'(\beta m_i+B))$. For ease of presentation, we have chosen $g(x)=x$ in \eqref{eq:pivotstat}.
	
	\begin{theorem}\label{thm:highordclt}
		Suppose Assumptions \ref{as:coeff} and \ref{as:highrowsum} hold. Recall the definitions of $U_N,V_N$ from \eqref{eq:randvar} with $g(x)=x$ and suppose \eqref{eq:bddzero} holds. Then given any sequence of positive reals $\{a_N\}_{N\ge 1}$ such that $a_N\to 0$, we have 
		$$\frac{1}{\sqrt{(U_N+V_N)\vee a_N}}\sum_{i=1}^N c_i(\si_i-\Xi'(\beta m_i+B))\overset{w}{\longrightarrow} N(0,1).$$
	\end{theorem}
	
	We can now leverage \cref{thm:highordclt} to provide asymptotic distribution of the pseudolikelihood estimator for $(\beta,B)$.  Following \cref{def:plestim}, the pseudolikelihood estimator is given by $(\hbem,\hbm)$ which satisfies the equations 
	$$\begin{pmatrix} \sum_{i=1}^{N} m_i(\si_i-\Xi'(\beta m_i+B)) \\ \sum_{i=1}^{N} (\si_i-\Xi'(\beta m_i+B))\end{pmatrix}=\begin{pmatrix} 0 \\ 0 \end{pmatrix},$$
	with $m_i$s defined in \eqref{eq:m_ihigh}. To obtain the limit distribution of $(\hbem,\hbm)$, we will adopt the same framework of cut norm convergence (see \cref{def:defirst}) as in \cref{sec:jointpseudo}. In particular, we assume that there exists a measurable $W:[0,1]^2\to [0,1]$ such that 
	\begin{equation}\label{eq:cutg}
		d_{\square}(W_{\A_N},W)\to 0.
	\end{equation}
	Under model \eqref{eq:modelhighis} and assumption \eqref{eq:cutg}, by \cite[Proposition 1.1]{bhattacharya2023gibbs}, it follows that: 
	\begin{align}\label{eq:highoptim}
		\frac{1}{N}Z_N(\beta,B)&\overset{N\to\infty}{\longrightarrow} \sup_{f:[0,1]\to [-1,1]} \bigg(\beta \int_{[0,1]^v} \mathrm{Sym}[W](x_1,\ldots ,x_v)\left(\prod_{a=1}^v f(x_a)\right)\,\prod_{a=1}^v \,d x_a  \nonumber \\ &\qquad \qquad \qquad + B \int_{[0,1]} f(x)\, dx - \int_{[0,1]} ((\Xi')^{-1}(f(x))f(x)-\Xi((\Xi')^{-1}(f(x))))\,dx\bigg).
	\end{align}
	As in \cref{thm:jointCLT}, our main result below shows that the limiting distribution of $(\hbem,\hbm)$ can be characterized in terms of the optimizers of \eqref{eq:optimprob}. In the same spirit as the irregularity assumption earlier (see \cref{as:irregraph}), we impose an irregularity assumption on an appropriately symmetrized tensor, which we state below.
	\begin{Assumption}[Irregular tensor]\label{as:irretensor}
		Consider a symmetric measurable $W:[0,1]^2\to [0,1]$. The symmetrized tensor $\mathrm{Sym}[W]$ (defined in \eqref{eq:symfunc}) is said to be an irregular tensor if 
		\begin{align}\label{eq:irreg}
			\int_{x_1\in[0,1]} \left(\int_{(x_2,\ldots ,x_v)\in [0,1]^{v-1}} \mathrm{Sym}[W](x_1,x_2,\ldots ,x_v)\,\prod_{a=2}^v \,dx_a - \int_{[0,1]^v} W(x_1,\ldots ,x_v)\,\prod_{a=1}^v \,dx_a\right)^2\,dx_1>0.
		\end{align}
		In other words, the row integrals of $\mathrm{Sym}[W]$ are non-constant.
	\end{Assumption}
	We are now in position to state the main result of this section. 
	\begin{theorem}\label{thm:jointCLThigh}
		Suppose $\A_N$ satisfies \cref{as:highrowsum} and \eqref{eq:cutg} for some $W$ satisfying the irregularity condition in \cref{as:irregraph}. Suppose that $\beta>0$, $B>0$ and the MPLE $(\hbem,\hbm)$ is consistent for $(\beta,B)$.  For any $f:[0,1]\to [-1,1]$, define $\mathcal{A}_f$ and $\mathcal{B}_f$ as in \eqref{eq:amat} and \eqref{eq:bmat} respectively. Assume now that the optimization problem \eqref{eq:highoptim} has an almost everywhere unique solution $\fs$. 
		Then $\mca_{\fs}$ is invertible and  
		$$\sqrt{N}\begin{pmatrix} \hbem-\beta \\ \hbm-B\end{pmatrix}\overset{w}{\longrightarrow} N\left(\begin{pmatrix} 0 \\ 0\end{pmatrix}, \mca_{\fs}^{-1}\mcb_{\fs}\mca_{\fs}^{-1}\right).$$
	\end{theorem} 
	
	\cref{thm:jointCLThigh} therefore provides a joint CLT for estimating $(\beta,B)$ using the maximum pseudolikelihood estimator $(\hbem,\hbm)$. As mentioned in \cref{sec:jointpseudo}, a sufficient condition for unique solutions to the optimization problem in \eqref{eq:highoptim} is to assume that $B$ is large enough. While we have focused on joint estimation of $(\beta,B)$ under the irregularity assumption \cref{as:irretensor}, our results can also be used to yield marginal CLTs for $\hbem$ (when $B$ is known) and $\hbm$ (when $\beta$ is known). The main ideas are similar to those in \cref{sec:marginalpseudo}. 
	
	\begin{remark}[Difference with \cref{thm:jointCLT}]
		We note that \cref{thm:jointCLThigh} has two extra assumptions compared to \cref{thm:jointCLT} --- namely the consistency of $(\hbem,\hbm)$ and the positivity of $B$. So the latter does not follow from the former. The consistency assumption can be removed by restricting $(\beta,B)$ to a compact parameter space. The positivity of $B>0$ will be used to ensure that $\mca_{\fs}$ is invertible. On the event that consistency of $(\hbem,\hbm)$ and $\mca_{\fs}$ are proved under weaker assumptions, \cref{thm:jointCLThigh} will immediately extend to such regimes.
	\end{remark}
	\subsection{Exponential random graph model}\label{sec:ergm}
	
	Exponential random graph models (ERGMs) are a family of Gibbs distributions on the set of graphs with $N$ vertices. They provide a natural extension to the Erd\H{o}s-R\'{e}nyi graph model by allowing for interactions between edges. They have become a staple in modern parametric network analysis with applications in sociology \cite{frank1986markov,park2005solution} and statistical physics \cite{Richard2017}. We refer the reader to \cite{Chatt2016bull} for a survey on random graph models. In this Section, we will focus on the following ERGM on undirected networks (following the celebrated works of \cite{Bhamidi2011,Chatterjee2013}) --- Consider 
	a finite list (not growing with $N$) of template graphs \(H_1,\dots,H_k\) without isolated vertices
	and a parameter vector \(\bbeta=(\beta_1,\dots,\beta_k)\in\mathbb{R}^k\).
	Let \(\mathcal{G}_N\) be the set of all simple graphs (undirected without self-loops or multiple edges) on vertex set \(\{1,\dots,N\}\).
	For \(G\in \mathcal{G}_N\), the ERGM puts probability
	\begin{equation}
		\label{eq:ergm}
		\PP_{\bbeta}(G) \ =\ \frac{1}{Z_N(\bbeta)}\,
		\exp\!\Big(\,N^2\sum_{m=1}^k \beta_m\,t(H_m,G)\Big),
	\end{equation}
	where
	\[
	t(H_m,G)\ :=\ \frac{|{\rm Hom}(H_m,G)|}{N^{\,|V(H_m)|}}\!,
	\]
	and \(|{\rm Hom}(H_m,G)|\) denotes the number of homomorphisms of $H_m$ into $G$ (i.e. the number of injective mappings from the vertex set of $H_m$ to the vertex set of $G$ such that edge in $H_m$ is mapped to an edge in $G$). Typically $t(H_m,G)$ is referred to as the homomorphism density. In particular if $H_m$ is an edge, then $t(H_m,G)=2N^{-2}\#\{\mbox{number of edges in }G\}$. On the other hand if $H_m$ is a triangle, then $t(H_m,G)=6N^{-3}\#\{\mbox{number of triangles in }G\}$. In this paper, we assume throughout that $H_1$ is an edge and $H_2,\ldots ,H_k$ have at least two edges each. Let $v_m$ and $e_m$ denote the number of vertices and edges in $H_m$. therefore $v_1=2$ and $e_1=1$. 
	
	Theoretical understanding of \eqref{eq:ergm} is hindered by the non-linear nature of the Hamiltonian. We first introduce the wonderful works of \cite{Bhamidi2011} and \cite{Chatterjee2013} (also see \cite{Chatt2010}) where the authors identified a parameter regime where \eqref{eq:ergm} ``behaves as" the Erd\H{o}s-R\'{e}nyi random graph model, thereby significantly advancing the understanding of \eqref{eq:ergm}.  
	
	\begin{definition}[Sub-critical regime]\label{def:subcrit}
		Define the functions 
		\begin{align}\label{eq:subcrit}
			\Phb(x):=\sum_{m=1}^k \beta_m e_m x^{e_m-1}, \qquad \quad \vphb(x):=\frac{\exp(2\Phb(x))}{\exp(2\Phb(x))+1}.
		\end{align}
		The \emph{sub-critical} regime contains all the parameters $\bbeta=(\beta_1,\ldots ,\beta_k)$, $\beta_1\in\R$ and $\beta_m>0$ for $m\ge 2$,  such that there is a unique solution $\ps\equiv \psb$ to the equation $\vphb(x)=x$ in $(0,1)$ and $\vphb'(\ps)<1$. In \cite[Theorem 7]{Bhamidi2011}, the authors show that in the sub-critical regime graphs drawn according to \eqref{eq:ergm} have asymptotically independent edges with edge-probability $\ps$. In \cite[Theorem 4.2]{Chatterjee2013}, the authors show that in the sub-critical regime, \eqref{eq:model} behaves like an Erd\H{o}s-R\'{e}nyi model with edge probability $\ps$ in terms of large deviations on the space of graphons. More recently, \cite{Reinert2019} provide a quantitative bound for the proximity between model \eqref{eq:subcrit} and the Erd\H{o}s-R\'{e}nyi model in the sub-critical regime. Note that the term sub-critical regime is not explicit in \cite{Bhamidi2011,Chatterjee2013}. We adopt this from more recent developments in the area; see \cite{Ganguly2024,fang2025normal}.
	\end{definition}
	
	\begin{remark}[Edge-triangle example]
		Let \(H_1\) be a single edge and \(H_2=K_3\) (a triangle), with parameters \((\beta_1,\beta_2)\).
		Then $v_1=2$, $e_1=1$, $v_2=3$, and $e_2=3$, so
		\[
		\Phb(x)\ =\beta_1 + 3\beta_2 x^2,
		\qquad
		\vphb(x)\ =\ \frac{\exp\!\big(2\beta_1+6\beta_2 x^2\big)}
		{1+\exp\!\big(2\beta_1+6\beta_2 x^2\big)}.
		\]
		The fixed point \(\ps\in(0,1)\) satisfies \(2\beta_1+6\beta_2 (\ps)^{2}=\log\!\big(\ps/(1-\ps)\big)\),
		and the sub-critical condition reads
		\[
		\varphi_\beta'(\ps)= 2\ps(1-\ps)\Phi'_\beta(\ps)
		\ =\ 2\,\ps(1-\ps)\cdot \big(6\beta_2 \ps\big)
		\ =\ 12\,\beta_2\,(\ps)^{2}(1-\ps)\ <\ 1.
		\]
	\end{remark}
	
	A standing question in the ERGM literature has been to obtain the asymptotic distribution of the total number of edges of a graph $G$ drawn according to \eqref{eq:model}. In \cite{mukherjee2013consistent}, the authors study CLTs for number of edges in the special case of two-star ERGMs (where $k=2$, $H_1$ is an edge, $H_2$ is a two-star). Their proof heavily exploits the relationship between the said model and the Curie-Weiss Ising model, and consequently doesn't extend to the general case of model \eqref{eq:model}. \cite{Ganguly2024} proved a CLT for the number of edges in $o(N^2)$ disconnected locations (which do not share a common vertex) in the sub-critical phase. In the same regime \cite{Sambale2020} shows that CLTs for general subgraph counts can be derived from the CLT of edges. More recently the authors of \cite{fang2025normal} prove a CLT for the total number of edges in the full sub-critical regime \cref{def:subcrit}.
	
	Therefore the existing edge CLTs are either specialized to specific choices of $H_i$s or focus entirely on the sub-critical regime. In the main result of this Section, We show that for conditionally centered number of edges, a studentized CLT holds without restricting to the sub-critical phase as long as variance positivity condition is satisfied. To state the result, we observe that the edge indicators under model \eqref{eq:ergm} have the probability mass function 
	\begin{align}\label{eq:edgepmf}
		\PPe(\by):=\frac{1}{Z_N(\bbeta)}\exp\left(\sum_{m=1}^k \frac{\beta_m}{N^{v_m-2}}\big|\mathrm{Hom}(H_m,G_y)\big|\right), \quad \by\in \{0,1\}^{{N \choose 2}}.
	\end{align}
	where $G_{\by}$ is the graph with edge indicators $\by$. Writing $L(x):=\exp(x)/(1+\exp(x))$ to denote the logistic function. Let $\bby\sim \PPe$. For $1\le i<j\le N$, let $Y_{-ij}$ denote the set of all edge indicators other than $Y_{ij}$. Then 
	\begin{equation}\label{eq:conergm}
		\EEB[Y_{ij}|Y_{-ij}]=L(\eta_{ij}), \quad \eta_{ij}:=\sum_{m=1}^k \frac{\beta_m}{N^{v_m-2}}\sum_{(a,b)\in E(H_m)}\sum_{\substack{(k_1,\ldots ,k_{v_m}) \textrm{ distinct, }\\ \{k_a,k_b\}=\{i,j\}}}\prod_{(p,q)\in E(H_m)\setminus (a,b)} Y_{k_p k_q}.
	\end{equation}
	Once again $\EEB[\si_i|\si_j,j\neq i]$ is a smooth transformation of a product of monomials. Following the discussion in \cref{sec:howtover}, we can use \cref{lem:smoothcont} to establish \cref{as:cmean}. This will allow use to invoke our main result \cref{theo:CLTmain} without restricting to the sub-critical regime in \cref{def:subcrit}.
	\begin{theorem}\label{thm:ergmclt}
		Consider the conditionally centered edge counts 
		\begin{align}\label{eq:tnedge}
			\Tne:=\frac{1}{\sqrt{{N \choose 2}}}\sum_{1\le i<j \le N} (Y_{ij}-L(\eta_{ij})).
		\end{align}
		Set $\mathcal{I}:=\{(i,j): 1\le i<j\le N\}$. We define $\Une$ and $\Vne$ as follows: 
		\begin{equation}\label{eq:randvarergm}
			\Une:=\frac{1}{{N \choose 2}}\sum_{(i,j)\in\mathcal{I}}(Y_{ij}-L^2(\eta_{ij})) \quad \mbox{and} \quad \Vne:=\frac{1}{{N\choose 2}}\sum_{\substack{(i_1,j_1)\neq (i_2,j_2)\\ \in \mathcal{I}}} (Y_{i_1j_1}-L(\eta_{i_1 j_1})(L(\eta_{i_2 j_2}^{(i_1,j_1)})-L(\eta_{i_2 j_2})).
		\end{equation}
		Suppose there exists $\eta>0$ such that 
		\begin{align}\label{eq:varlbdergm}
			\PPe(\Une+\Vne\geq \eta)\to 1.
		\end{align}
		Then given any sequence of positive reals $\{a_N\}$ we have 
		$$\frac{\Tne}{\sqrt{(\Une+\Vne)\vee a_N}}\overset{w}{\longrightarrow} N(0,1).$$
	\end{theorem}
	
	We note that \cref{thm:ergmclt} does not impose any sub-criticality restriction for the eventual limit. In the aforementioned regime, the variance can be simplified as stated in the following corollary.
	
	\begin{corollary}\label{cor:ergmclt}
		Consider $\Tne$ defined as in \eqref{eq:tnedge}. Suppose the parameter vector $\bbeta$ lies in the sub-critical regime from \cref{def:defirst}. Then 
		$$\Tne \overset{w}{\longrightarrow} N(0,\ps(1-\ps)(1-\vphb'(\ps))).$$
	\end{corollary}
	Note that the sub-criticality condition $\vphb'(\ps)<1$ ensures that the above limiting variance is strictly positive. 
	\begin{remark}[Extension to negative $\beta_m$s]
		The proof of \cref{cor:ergmclt} follows from combining \cref{thm:ergmclt} with the proximity between model \eqref{eq:ergm} and the appropriate Erd\H{o}s-R\'{e}nyi model as proved in \cite{Reinert2019}. We have stated the result for the sub-criticality regime as it seems to be the primary focus of the current literature. However the same conclusion also applies to the Dobrushin uniqueness regime 
		$$\sum_{m=2}^k |\beta_m|e_m(e_m-1)<2,$$
		which accommodates small negative values of $(\beta_2,\ldots ,\beta_k)$. 
		The proof strategy would exactly be the same as we would combine \cref{thm:ergmclt} (which puts no parameter restrictions), coupled with \cite[Theorem 1.7]{Reinert2019} which applies to the above uniqueness regime. 
	\end{remark}
	
	An immediate implication of \cref{thm:ergmclt} is a CLT for the pseudolikelihood estimator of $\beta_m$, $1\le m\le k$ when the rest are known. For simplicity, we will focus only on estimating $\beta_1$. To the best of our knowledge, limit theory for estimating the parameters of the ERGM \eqref{eq:ergm} has only been studied in the special case of the two-star model in \cite{mukherjee2013consistent}. Corollary 1.3 of \cite{mukherjee2013consistent} suggests that joint $O(N)$ estimation of $(\beta_1,\ldots ,\beta_k)$ may not be possible. Therefore, we only focus on the marginal estimation problem here. Under \eqref{eq:edgepmf}, the pseudolikelihood function is given by 
	\begin{align}\label{eq:plf}
		\mathrm{PL}(\beta_1):=\sum_{(i,j)\in\mathcal{I}} \left(Y_{ij}\eta_{ij}(\beta_1)-\log(1+\exp(\eta_{ij}(\beta_1))\right).
	\end{align}
	Note that $\eta_{ij}$ defined in \eqref{eq:conergm} depends on $\beta_1$. Therefore we have parametrized it as $\eta_{ij}\equiv \eta_{ij}(\beta_1)$. Fix some known compact set $K\in\R$ which contains the true parameter $\beta_1$. Following \eqref{eq:MPLE}, we take the derivative of the above pseudolikelihood function, and define the pseudolikelihood estimator for $\beta_1$ as $\hbpm\in K$ satisfying 
	\begin{align}\label{eq:ergmPL}
		\sum_{(i,j)\in\mathcal{I}} (Y_{ij}-L(\eta_{ij}(\hbpm)))=0,
	\end{align}
	when it exists. The following result provides the limit distribution of $\hbpm$.
	
	\begin{theorem}\label{thm:MPLergmCLT}
		Recall the definitions of $\Une$ and $\Vne$ from \eqref{eq:randvarergm}. Suppose that the true parameter $\beta_1\in K$, the known compact set. Then a unique pseudolikelihood estimator $\hbpm$ exists with probability converging to $1$. Suppose further that \eqref{eq:varlbdergm} holds. Then for any sequence of positive reals $\{a_N\}$ converging to $0$, we have:
		\begin{align}\label{eq:firstconc}
			\frac{1}{\sqrt{(\Une+\Vne)\vee a_N}}\left(\frac{2}{{N\choose 2}}\sum_{(i,j)\in\mathcal{I}} L(\eta_{ij}(\beta_1))(1-L(\eta_{ij}(\beta_1)))\right)\sqrt{{N\choose 2}}(\hbpm-\beta_1)\overset{w}{\longrightarrow} N(0,1),
		\end{align}
		provided $$\left(\frac{2}{{N\choose 2}}\sum_{(i,j)\in\mathcal{I}} L(\eta_{ij}(\beta_1))(1-L(\eta_{ij}(\beta_1)))\right)^{-1}=O_{\PPe}(1).$$
		In particular, in the sub-critical regime from \cref{def:defirst}, we have: 
		\begin{align}\label{eq:secondconc}
			\sqrt{{N\choose 2}}(\hbpm-\beta_1) \overset{w}{\longrightarrow} N\left(0,\frac{\ps(1-\ps)}{4(1-\vphb'(\ps))}\right).
		\end{align}
	\end{theorem}
	
	Note that \cref{thm:MPLergmCLT} applies without imposing the sub-criticality assumption. This is largely due to the fact that \cref{thm:ergmclt} applies without the same restrictions. Once again this shows the benefits of having our main result \cref{theo:CLTmain} without imposing any restrictive modeling assumptions. 
	
	\section{Discussion and proof overview}\label{sec:proof-overview}
	
	The main technical tool for proving our main results, namely Theorems \ref{theo:CLTmain} and \ref{theo:conmain}, is a method of moments argument. The lack of independence between the observations presents a significant challenge towards proving the above Theorems only under smoothness assumptions on the conditional mean (see \cref{as:cmean}). To contextualize, let us outline how the method of moments argument works when dealing with independent random variables. Suppose $\{X_i\}_{i=1}^{\infty}$ are bounded i.i.d. random variables. Then 
	$$\E\left(\frac{1}{\sqrt{N}}\sum_{i=1}^N X_i\right)^k=\frac{1}{N^{k/2}}\sum_{(i_1,\ldots ,i_k)\in [N]^k} \E [X_{i_1}\ldots ,X_{i_k}].$$
	By independence, $\E[X_{i_1}\cdots X_{i_k}]$ factorizes over distinct indices. Writing
	the multiplicities of $\{i_1,\dots,i_k\}$ as a composition $(\ell_1,\dots,\ell_r)$ with $\ell_1+\cdots+\ell_r=k$ and $\ell_j\ge1$, each configuration contributes on the order of
	\[
	N^{\,r-k/2}\cdot \prod_{j=1}^r \E\!\big[X_1^{\,\ell_j}\big].
	\]
	Since $\E X_1=0$ and the variables are bounded, any part with an odd $\ell_j$ or with some $\ell_j\ge3$ either vanishes or is $o(1)$ after the $N^{-k/2}$ normalization; the only contributions that can survive are those with $r=k/2$ and \emph{all} multiplicities equal to $2$, i.e.
	\[
	(\ell_1,\dots,\ell_r) \;=\; \underbrace{(2,2,\dots,2)}_{k/2\ \text{times}}.
	\]
	This immediately forces $k$ to be even. The conclusion then follows from a standard counting argument. 
	
	The argument for our random field setting is much more subtle. Let us write $Y_i=\si_i-\E[\si_i|\si_j,j\neq i]$. Of course, 
	$$\E\left(\frac{1}{\sqrt{N}}\sum_{i=1}^N Y_i\right)^k=\frac{1}{N^{k/2}}\sum_{(i_1,\ldots ,i_k)\in [N]^k} \E [Y_{i_1}\ldots ,Y_{i_k}].$$
	The expectation no longer factorizes over distinct indices. So we can only simplify it as 
	\begin{align}\label{eq:calat}
		N^{\,r-k/2}\cdot \E\prod_{j=1}^r \!\big[Y_{i_j}^{\,\ell_j}\big].
	\end{align}
	This time around, both the terms 
	$$(\ell_1,\ldots ,\ell_r)=\underbrace{(1,1,\ldots ,1)}_{k-\mbox{times}} \qquad \mbox{and} \qquad (\ell_1,\ldots ,\ell_r)=\underbrace{(2,2,\ldots ,2)}_{(k/2)-\mbox{times}}$$
	contribute to the limiting variance, unlike in the i.i.d. setting. In fact, the number of contributing  summands is of the order $k$, and each of their contributions need to be tracked and combined to arrive at the correct limiting variance. This makes the method of moments computation considerably more challenging in our setting. Let us lay out below the chain of auxiliary ingredients that enable the argument.
	
	\paragraph{Road map and main ideas.}
	\begin{enumerate}
		\item \textbf{From a structural limit to a pivot.} The studentized CLT in \cref{theo:CLTmain} is proved using the unstudentized CLT in \cref{theo:conmain}. This requires a careful tightness+diagonal subsequence argument. The variance positivity condition in \eqref{eq:bddzero} ensures that studentization step removes the mixture randomness and yields a pivotal Gaussian limit.
		
		\item \textbf{Truncating weights and exponential concentration.} \cref{theo:conmain} is proved using \cref{theo:conmainder}. The subject of \cref{theo:conmainder} is to claim the same unstudentized limit but with the additional assumption that the weight vector $\bc$ is uniformly bounded. By leveraging concentration inequalities established in \cref{lem:auxtail}, we show that this additional boundedness assumption can be made without loss of generality. 
		
		\item \textbf{Moment method with combinatorial pruning.} Next we establish \cref{theo:conmainder}. The key tool here is a method of moments argument. The primary technical device is a \emph{rank/matching} bookkeeping result (see \cref{lem:surviveterm}) that prunes all high-order contributions except certain ``weak  pairings". Concretely, if any component of \eqref{eq:calat} appears with power \(\ge 3\) or the total multiplicity is odd, the configuration's contribution vanishes in the limit. The only surviving terms are when the number of isolated components is even and all the others occur with multiplicity $2$. This is a crucial point of difference with the i.i.d. case where terms with isolated components do not contribute. \cref{lem:surviveterm} reduces high-order moments to a reasonably tractable counting problem.
		
		\item \textbf{A Decision tree approach.} The final ingredient is the proof of \cref{lem:surviveterm}. We take a decision tree approach where every term of the form \eqref{eq:calat} is split up sequentially into a group of ``smaller" terms, till they meet a termination criteria. The splitting is made explicit in Algorithms \ref{alg:construct_tree} and \ref{alg:construct_treep2}. In every step of the split, we throw away terms which have exactly mean $0$ (see \cref{prop:baseprop}). Using some technical bounds, we show in \cref{prop:degbound} that the split leads to asymptotically negligible terms if either the tree grows too large or if the tree terminates too early. This leads us to characterize the set of all branches of the tree  that have non-negligible contributions in the large $N$ limit, which is the subject of \cref{lem:negterm}.
		\item \textbf{Verifying \cref{as:cmean}.} An important component of this paper is to provide a clean method to verify \cref{as:cmean} which is the main technical condition. This is achieved in \cref{lem:smoothcont}, which can be viewed as a consequence of a discrete Fa\`{a} Di Bruno type formula which is established in \cref{cl:smoothclaim2}, and may be of independent interest. 
	\end{enumerate}
	
	\section*{Acknowledgement}
	The author would like to thank Prof. Sumit Mukherjee for proposing this problem, and for continued help and insightful suggestions throughout this project.
	
	\bibliographystyle{imsart-nameyear}
	\bibliography{template}
	
	\appendix
	
	\section{Proof of Main Result}\label{sec:pfmain}
	
	This Section is devoted to proving our main results, namely Theorems \ref{theo:CLTmain},  \ref{theo:conmain}, and \ref{lem:smoothcont}. In the sequel, we will use $\lesssim$ to hide constants free of $N$. We begin  with a preparatory lemma which will be used multiple times in the proof of~\cref{theo:conmain}. The Lemma basically provides concentrations for certain conditionally centered functions of $\ms$. 
	\begin{lemma}\label{lem:auxtail}
		Suppose $\ms\sim \P$. Recall the definition of $t_i$s from \eqref{eq:consig}. Then given any vector $\mathbf{d}^{(N)}:=(d_1,d_2,\ldots ,d_N)$, and scalar $t>0$, there exists a constant $C$ free of $t$ such that the following conclusions hold: 
		\begin{enumerate}
			\item[(a)] Under~\cref{as:cmean}, we have:
			$$\P\Big(\big|\sum_{i=1}^N d_i(g(\si_i)-t_i)\big|>t\Big)\leq 2\exp\left(-\frac{C t^2}{\lVert \mathbf{d}^{(N)}\rVert^2}\right).$$
			\item[(b)] Let $r_i\equiv r_i(\si_1,\si_2,\ldots ,\si_{i-1},b_0,\si_{i+1},\ldots ,\si_N)$ be a function of $N-1$ coordinates where the $i$-th coordinate is fixed at $b_0\in\mcb$, where $\sup_{\ms\in\mcb^N}\max_{1\le i\le N} |r_i|\le 1$. Also assume $|r_i-r_i^j|\le \Q_{N,2}(i,j)$ where $\Q_{N,2}$ is the matrix from \cref{as:cmean}. Then, under~\cref{as:cmean}, provided $\lVert\mathbf{d}^{(N)}\rVert=O(N)$, we have:
			$$\E\left(\sum_{i=1}^N d_i(g(\si_i)-t_i)r_i\right)^2\lesssim N.$$
		\end{enumerate}		
	\end{lemma}
	Here, parts (a) and (b) follows by making minor adjustments in the proofs  of~\cite[Lemma 1]{Mukherjee2018}, \cite[Lemma 3.2]{Cha2016}, and \cite[Lemma 3.1]{lee2025fluctuations}, respectively. We include a short proof in \cref{Sec:concproof} for completion.
	
	\subsection{Proof of Theorems~\ref{theo:CLTmain} and \ref{theo:conmain}}
	
	As the proof of \cref{theo:CLTmain} uses \cref{theo:conmain}, we will begin with the proof of the latter. In order to achieve this, it is convenient to work under the following condition:
	
	\begin{Assumption}\label{as:coeffalt}[Uniform boundedness of coefficient vector] 
		The vector $\mc=(c_1,\ldots ,c_N)$ satisfies the following condition:
		$$\limsup_{N\to\infty}\max_{i\in [N]}|c_i|<\infty.$$
	\end{Assumption}
	This Assumption is of course strictly stronger than \cref{as:coeff}. However, as it turns out, through a careful truncation and diagonal subsequence argument, proving our main result under \cref{as:coeffalt} is equivalent to proving it under \cref{as:coeff}. To formalize this, we present the following modified result.
	
	\begin{theorem}\label{theo:conmainder}
		For any $k\in\N$, under Assumptions~\ref{as:cmean}, \ref{as:empcon}, and \ref{as:coeffalt}, we have 
		$$\E T_N^k U_N^{k_1}V_N^{k_2} \to m_{k,k_1,k_2}$$
		as $n\to\infty$, where $m_{k,k_1,k_2}$ is defined in \eqref{eq:limmoment}, and $U_N$, $V_N$ are defined as in \eqref{eq:randvar}. 
	\end{theorem}
	
	Next, we will derive \cref{theo:conmain} from \cref{theo:conmainder}. 
	
	\begin{proof}[Proof of \cref{theo:conmain}]
		Suppose Assumptions~\ref{as:coeff},~\ref{as:cmean}, and~\ref{as:empcon} are satisfied. First let us assume that \eqref{eq:bigconmoment} holds and establish the existence of a unique $\rho$ with moment sequence $m_{k,0,0}$ and the non-negativity of $P_1+P_2$. We will then establish \eqref{eq:bigconmoment} independently. By \eqref{eq:boundZ1} and \eqref{eq:boundz2}, it follows that $P_1$ and $P_2$ are almost surely bounded, which implies that the sequence $m_{k,k_1,k_2}$ is well-defined. Further, given any $\lambda>0$, we have: 
		$$\sum_{k=0}^{\infty} \frac{\lambda^k m_{k,0,0}}{k!}<\infty.$$
		Therefore, provided we can show $\EE T_N^k\to m_{k,0,0}$ for all $k\ge 0$ (which is a consequence of \eqref{eq:bigconmoment}), $\{m_{k,0,0}\}_{k\ge 0}$ will correspond to the moment sequence of a unique probability measure $\rho$ (see e.g.~\cite[Corollary 2.12]{sodin2019classical}). Set $\mathrm{i}=\sqrt{-1}$. By \cref{lem:auxtail}, part (a), coupled with \eqref{eq:bigconmoment}, we have for any $t\in \R$, the following 
		\begin{align*}
			\varphi_n(t):=\E [\exp(\mathrm{i} t T_N)] &= \sum_{j=0}^{\infty}\frac{(\mathrm{i}t)^j}{j!}\E [T_N^j] \overset{N\to\infty}{\longrightarrow}\sum_{j=0}^{\infty} (-1)^j\frac{t^{2j}}{(2j)!}(2j)!!m_{j,0,0}=\mathbb{E} \exp\left(-\frac{t^2}{2}(P_1+P_2)\right)=:\varphi(t). 
		\end{align*}
		As $P_1$ and $P_2$ are bounded by \eqref{eq:boundZ1} and \eqref{eq:boundz2}, $\varphi(\cdot)$ is everywhere continuous. Therefore by \cite[Theorem 3.3.17]{durrett2019probability}, $\varphi(\cdot)$ is a characteristic function. Clearly it is always real-valued and non-negative. 
		
		With this in mind, suppose that $P_1+P_2$ is not non-negative almost surely. Then there exists $\eta_1>0$ and $0<\eta_2<1$, such that $\mathbb{P}(P_1+P_2<-\eta_1)\ge \eta_2$. As $\varphi(\cdot)$ is a characteristic function, by choosing $t> \sqrt{-2\log{\eta_2}}/\sqrt{\eta_1}$, we have 
		\begin{align*}
			1\ge \varphi(t) &\ge \mathbb{E}[\exp(-t^2 (P_1+P_2)/2)\ind(P_1+P_2\le -\eta_1)] \\ &\ge \exp\left(\frac{t^2\eta_1}{2}\right)\mathbb{P}(P_1+P_2\le -\eta_1)  > 1,
		\end{align*}
		which results in a contradiction. Therefore $P_1+P_2$ is almost surely non-negative. This implies $\varphi(\cdot)$ is the characteristic function of $\textnormal{Law}(\sqrt{P_1+P_2}Z)$ where $Z\sim N(0,1)$ is independent of $P_1,P_2$. This establishes the conclusions of \cref{theo:conmain} outside of \eqref{eq:bigconmoment}.
		
		\vspace{0.05in}
		
		In the rest of the argument, we will focus on proving the aforementioned moment convergence in \eqref{eq:bigconmoment}.
		\noindent It suffices to show that, given any subsequence $\{N_r\}_{r\geq 1}$, there exists a further subsequence $\{N_{r_\ell}\}_{\ell\geq 1}$ such that:
		\begin{equation}\label{eq:cboundedmain}
			\E T_{N_{r_{\ell}}}^k U_{N_{r_{\ell}}}^{k_1} V_{N_{r_{\ell}}}^{k_2}\to m_{k,k_1,k_2}.
		\end{equation}
		
		Towards this direction, define $c_{i,1,M}:=c_i\ind(|c_i|\leq M)$ and $c_{i,2,M}:=c_i-c_{i,M,1}=c_i\ind(|c_i|>M)$. We also define 
		$$T_{N_r,M}:=\frac{1}{\sqrt{N_r}}\sum_{i=1}^{N_r} c_{i,1,M}(g(\si_i)-t_i),$$
		for $M\in\N$. Therefore $T_{N_r,M}$ is the truncated version of the target statistic $T_{N_r}$ defined in \eqref{eq:pivotstat}. In a similar vein, we define 
		$$U_{N_r,M}:=\frac{1}{N_r}\sum_{i=1}^{N_r} c_{i,1,M}^2(g(\si_i)^2-t_i^2), \quad \mbox{and}\quad V_{N_r,M}:=\frac{1}{N_r}\sum_{i,j} c_{i,1,M}c_{j,1,M}(g(\si)-t_i)(t_j^i-t_j).$$
		Clearly $U_{N_r,M}$ and $V_{N_r,M}$ are truncated versions of $U_N$ and $V_N$ defined in \eqref{eq:randvar}. 
		
		We now note that by using \cref{as:coeff}, it follows that: 
		\begin{align}\label{eq:uniftruncbd1}
			\sup_{M\ge 1}\frac{1}{N_{r}}\sum_{i=1}^{N_{r}} c_{i,1,M}^2|g(\si_i)^2-t_i^2|\lesssim \frac{1}{N_r}\sum_{i=1}^{N_r} c_i^2\lesssim 1.
		\end{align}
		Also note that by~\cref{as:cmean}, we have $|t_j^{i}-t_j|\leq \Q_{N,2}(i,j)$. By using \eqref{eq:cmean2} and \cref{as:coeff}, we then get:
		\begin{align}\label{eq:uniftruncbd2}
			\sup_{M\ge 1}\frac{1}{N_{r}}\Bigg|\sum_{i,j=1}^{N_{r}} c_{i,1,M} c_{j,1,M}(g(\si_i)-t_i)(t_j^{i}-t_j)\bigg|\lesssim \frac{1}{N_{r}}\lambda_1(\Q_{N,2})\sum_{i=1}^{N_{r}} c_i^2\lesssim 1.
		\end{align}
		Then by Prokhorov's Theorem and a standard diagonal subsequence argument, there exists a bivariate random variable $\mz_M=(P_{1,M},P_{2,M})$ and a common subsequence $\{N_{r_{\ell}}\}_{\ell\geq 1}$ such that:
		\begin{equation}\label{eq:cbd1}\begin{bmatrix}
				(N_{r_{\ell}})^{-1}\sum_{i=1}^{N_{r_{\ell}}} c_{i,1,M}^2(g(\si_i)^2-t_i^2)\\ (N_{r_{\ell}})^{-1}\sum_{i,j=1}^{N_{r_{\ell}}} c_{i,1,M} c_{j,1,M}(g(\si_i)-t_i)(t_j^{i}-t_j)
			\end{bmatrix}\overset{w}{\longrightarrow} \mz_M
		\end{equation}
		for all $M\in \mathbb{N}$. Further, by using \cref{theo:conmainder}, we can without loss of generality ensure that 
		\begin{align}\label{eq:concfollows}
			\E T_{N_{r_{\ell}},M}^k U_{N_{r_{\ell}},M}^{k_1} V_{N_{r_{\ell}},M}^{k_2}\to m_{k,k_1,k_2,M} \end{align}
		for all $M\ge 1$, where 
		$$m_{k,k_1,k_2,M}:=\begin{cases} 0 & \mbox{if}\, k \mbox{ is odd} \\ (k)!!\EE[(P_{1,M}+P_{2,M})^{k/2} P_{1,M}^{k_1} P_{2,M}^{k_2}] & \mbox{if}\, k \mbox{ is even}.\end{cases}$$
		Now we will show that along this same subsequence $\E T_{N_{r_{\ell}}}^k U_N^{k_1} V_N^{k_2} \to m_{k,k_1,k_2}$ for all $k,k_1,k_2\ge 0$. Note that by using the triangle inequality, we can write 
		\begin{align*}
			&\;\;\;\;\big|\E T_{N_{r_{\ell}}}^k U_{N_{r_{\ell}},M}^{k_1} V_{N_{r_{\ell}},M}^{k_2}-m_{k,k_1,k_2}\big| \\ &\le \big|\E T_{N_{r_{\ell}}}^k U_{N_{r_{\ell}}}^{k_1} V_{N_{r_{\ell}}}^{k_2} - \E T_{N_{r_{\ell}},M}^k U_{N_{r_{\ell}},M}^{k_1} V_{N_{r_{\ell}},M}^{k_2}\big| + \big|\E T_{N_{r_{\ell}},M}^k U_{N_{r_{\ell}},M}^{k_1} V_{N_{r_{\ell}},M}^{k_2} - m_{k,k_1,k_2,M}\big| \\ &\qquad + \big|m_{k,k_1,k_2,M}-m_{k,k_1,k_2}\big|. 
		\end{align*}
		By using \eqref{eq:concfollows}, the middle term in the above display converges to $0$ for every fixed $M$ as $\ell\to\infty$. It therefore suffices to show that 
		\begin{align}\label{eq:cbdd4}
			\lim_{M\to\infty}\limsup_{\ell\to\infty} \big|\E T_{N_{r_{\ell}},M}^k U_{N_{r_{\ell}},M}^{k_1} V_{N_{r_{\ell}},M}^{k_2} - m_{k,k_1,k_2,M}\big| = 0, 
		\end{align}
		and 
		\begin{align}\label{eq:cbdd5}
			\lim_{M\to\infty} m_{k,k_1,k_2,M}=m_{k,k_1,k_2}. 
		\end{align}
		\emph{Proof of \eqref{eq:cbdd4}.} Note that by using \cref{lem:auxtail}, part (a), we have: 
		$$\sup_{M,\ell\ge 1}\E |T_{N_{r_{\ell}},M}|^k\lesssim \sup_{M,\ell\ge 1}\left(\frac{1}{N_{r_{\ell}}}\sum_{i=1}^{N_{r_{\ell}}} c_{i,1,M}^2\right)^{\frac{k}{2}} \le \sup_{\ell\ge 1}\left(\frac{1}{N_{r_{\ell}}} \sum_{i=1}^{N_{r_{\ell}}} c_i^2\right)<\infty,$$
		where the last bound uses \cref{as:coeff}. The same argument also shows that $\E |T_{N_{r_{\ell}}}|^k\lesssim 1$, for all $k\ge 1$. Moreover $U_{N_{r_{\ell}},M}$, $V_{N_{r_{\ell}},M}$, $U_{N_{r_{\ell}}}$, and $V_{N_{r_{\ell}}}$ are all bounded by \eqref{eq:uniftruncbd1}, \eqref{eq:uniftruncbd2}, \eqref{eq:boundZ1}, and \eqref{eq:boundz2} respectively. Therefore, to establish \eqref{eq:cbdd4}, it suffices to show that \begin{equation}\label{eq:cbd4}\lim_{M\to\infty}\limsup_{\ell\to\infty} \P\left(|T_{N_{r_{\ell}}}-T_{N_{r_{\ell}},M}|\geq\epsilon\right)=0, 
		\end{equation}
		
		\begin{equation}\label{eq:cbd6}\lim_{M\to\infty}\limsup_{\ell\to\infty}\P\left(\bigg|U_{N_{r_{\ell}},M}-U_{N_{r_{\ell}}}\bigg|\ge \epsilon\right)\overset{P}{\longrightarrow} 0,
		\end{equation}
		
		\begin{equation}\label{eq:cbd7}	\lim_{M\to\infty}\limsup_{\ell\to\infty}\P\left(\bigg|V_{N_{r_{\ell}},M}-V_{N_{r_{\ell}}}\bigg|\ge \epsilon\right)\overset{P}{\longrightarrow} 0,
		\end{equation}
		for all $\epsilon>0$. We will only prove~\eqref{eq:cbd4} and \eqref{eq:cbd7} as the proof of~\eqref{eq:cbd6} follows using similar computations as that of \eqref{eq:cbd6} (and is in fact much simpler). 
		
		Let us begin with the proof of \eqref{eq:cbd4}. To wit, note that by~\cref{lem:auxtail}(a), we have
		\begin{equation}\label{eq:cbd3}
			\E|T_{N_{r_{\ell}}}-T_{N_{r_{\ell}},M}| = (N_{r_{\ell}})^{-1/2}\E\bigg|\sum_{i=1}^{N_{r_l}} c_{i,2,M}(g(\si_i)-t_i)\bigg|\lesssim \sqrt{(N_{r_{\ell}})^{-1}\sum_{i=1}^{N_{r_{\ell}}} c_i^2\ind(|c_i|>M)}\to 0
		\end{equation}
		as $\ell\to\infty$ followed by $M\to\infty$ by~\cref{as:coeff}. This establishes~\eqref{eq:cbd4} by Markov's inequality.
		
		\noindent Next we prove \eqref{eq:cbd7}. Let us first write $c_i c_j-c_{i,1,M} c_{j,1,M}=c_{i,2,M}c_j+c_{i,1,M}c_{j,2,M}$ and recalling that $|t_j^{i}-t_j|\leq \Q_{N,2}(i,j)$ (see~\cref{as:cmean}), we get the following bound
		\begin{align*}
			&\;\;\;\;\big|V_{N_{r_{\ell}},M}-V_{N_{r_{\ell}}}\big| \\ &=\bigg|\frac{1}{N_{r_{\ell}}}\sum_{i,j=1}^{N_{r_{\ell}}} (c_{i,1,M} c_{j,1,M}-c_i c_j)(g(\si_i)-t_i)(t_j^{i}-t_j)\bigg| \\ &\le \frac{1}{N_{r_{\ell}}}\sum_{i,j=1}^{N_{r_{\ell}}} \left(|c_j c_{i,2,M}|+|c_{i,1,M}c_{j,2,M}|\right)\Q_{N,2}(i,j)\\ &\le 2\lambda_1(\Q_{N,2})\sqrt{\frac{1}{N_{r_{\ell}}}\sum_{i=1}^{N_{r_{\ell}}} c_i^2}\sqrt{\frac{1}{N_{r_{\ell}}} \sum_{i=1}^{N_{r_{\ell}}} c_{i}^2\ind(|c_i| > M)}
		\end{align*}
		which converges to $0$ as $\ell\to\infty$ followed by $M\to\infty$ (using~\cref{as:coeff}). This establishes~\eqref{eq:cbd7} by Markov's inequality and completes the proof.
		
		\emph{Proof of \eqref{eq:cbdd5}.} By \eqref{eq:uniftruncbd1} and \eqref{eq:uniftruncbd2}, it suffices to show that \eqref{eq:cbd6} and \eqref{eq:cbd7} hold. This was already proved above.
	\end{proof}
	
	\begin{proof}[Proof of \cref{theo:CLTmain}]
		Suppose that Assumptions \ref{as:coeff} and \ref{as:bddrowsum} hold. By \cref{lem:auxtail}, part (a), we observe that the sequence $\{T_N\}_{N\ge 1}$ is tight. Moreover by \eqref{eq:boundZ1} and \eqref{eq:boundz2}, we also have that the sequences $\{U_N\}_{N\ge 1}$ and $\{V_N\}_{N\ge 1}$ are tight. Therefore, by Prokhorov's Theorem, there exists a subsequence $N_r$ such that 
		$$(U_{N_r} , V_{N_r}) \overset{w}{\longrightarrow} P=(P_1,P_2).$$
		We note that $(P_1,P_2)$ here can depend on the choice of the subsequence. 
		Therefore, along the sequence $\{N_r\}_{r\ge 1}$, \cref{as:empcon} holds as well. We can now apply \cref{theo:conmain} along the sequence $\{N_r\}_{r\ge 1}$ to get that 
		$$(T_{N_r}, U_{N_r} , V_{N_r})\overset{w}{\longrightarrow} (\sqrt{P_1+P_2}Z,P_1,P_2),$$
		where $Z\sim N(0,1)$ is independent of $(P_1,P_2)$. 
		Now by \eqref{eq:bddzero}, we have that $\mathbb{P}(P_1+P_2\ge \eta/2)=1$. Therefore, by applying Slutsky's Theorem, given any sequence of positive reals $\{a_N\}_{N\ge 1}$ converging to $0$, we have 
		$$\frac{T_{N_r}}{\sqrt{(U_{N_r}+V_{N_r})\vee a_{N_r}}} \overset{w}{\longrightarrow} N(0,1).$$
		As this limit is free of the chosen subsequence $\{N_r\}_{r\ge 1}$, the conclusion follows.
	\end{proof}
	
	To summarize, in this Section, we have proved \cref{theo:conmain} using \cref{theo:conmainder}. Then we proved \cref{theo:CLTmain} using \cref{theo:conmain}. Therefore it is now sufficient to prove~\cref{theo:conmainder}, which is the focus of the following section.
	
	\subsection{Proof of~\cref{theo:conmainder}}\label{sec:pftheo}
	Before delving into the proof of~\cref{theo:conmainder}, let us introduce and recall some notation. Given any $n\geq 1$, recall that \begin{equation}\label{eq:pftheo1} (2n)!!:=(2n-1)\times (2n-3) \times \ldots \times 3 \times 1.\end{equation}  Further, given two real-valued sequences $\{a_n\}_{n\geq 1}$, $\{b_n\}_{n\geq 1}$, we say
	\begin{equation}\label{eq:pftheo2}
		a_n\leftrightarrow b_n \quad \mbox{if}\quad \lim_{n\to\infty} |a_n-b_n|=0.
	\end{equation}
	In the same spirit, given two real valued random sequences $\{A_n\}_{n\ge 1}$ and $\{B_n\}_{n\ge 1}$ defined on the same probability space $(\Omega,P)$, we say 
	\begin{equation}\label{eq:pftheo3}
		A_n\overset{P}{\leftrightarrow} B_n \quad \mbox{if}\quad |A_n-B_n| \overset{P}{\rightarrow}0.
	\end{equation}
	Recall the definition of $\ms_{\cS}$ from~\eqref{eq:newsig}. Further, for any function $u:\B^N\to\R$, $\cS\subseteq [N]$, and $\ms\in \B^N$, let $u^S:\B^N\to\R$ denote the function satisfying $u^S(\ms)=u(\ms_S)$. In particular, with $u(\cdot)\equiv t_i(\cdot)$ (see~\eqref{eq:consig2}), we have $u^{\cS}(\cdot)=t_i^{\cS}(\cdot)$ (see~\eqref{eq:consig}). Similarly, for any $\cS_1\subseteq \cS_2\subseteq [N]$, let $u^{S_1;S_2}:\B^N\to\R$ be such that $u^{S_1;S_2}(\ms)=u(\ms_{\cS_1})-u(\ms_{\cS_2})$. For example, $u^{\phi;S}(\ms)=u(\ms)-u(\ms_{\cS})$, for $\cS\subseteq [N]$.
	
	With the above notation in mind, we now define two important notions:
	\begin{definition}[Rank of a function]\label{def:rankfub}
		Given a function $u:\B^N\to\R$, we define the rank of $u(\cdot)$, denoted by $\ns(u)$, as the minimum element in the set $\mathbb{N}\cup \{\infty\}$ such that 
		$$\bigg|\E \bigg[u(\ms)\bigg]\bigg|\le N^{\ns(u)}$$
		where $\ms\sim\P$. For instance, suppose $\P$ is any probability measure  supported on $\{-1,1\}^N$ and let $$u(\ms):=\bigg(\sum\limits_{i=1}^N \sigma_i\bigg)^k,\qquad k\in\mathbb{N}.$$ Then $\ns(u)\leq k$.  
	\end{definition}
	\begin{definition}[Matching]\label{def:matching}
		Fix $n\geq 1$. Given a finite set $A=\{a_1,a_2,\ldots ,a_{2n}\}$, a \textit{matching} on $A$ is a set of unordered pairs $\{(\mathfrak{m}_{1,1},\mathfrak{m}_{1,2}),\ldots (\mathfrak{m}_{n,1},\mathfrak{m}_{n,2})\}$ with elements from $A$ so that  $\mathfrak{m}_{i,1}>\mathfrak{m}_{i,2}$, $\mathfrak{m}_{i,1}> \mathfrak{m}_{j,1}$ and $\mathfrak{m}_{k,\ell}\in A$ are distinct elements for $i<j$, $k\in [n]$, $\ell=\{1,2\}$. Let $\M(A)$ be the set of all \textit{matchings} of $A$. For example,
		$$\M([4])=\{\{(4,1),(3,2)\},\{(4,2),(3,1)\},\{(4,3),(2,1)\}\}.$$ By~\cite[Page 88, Example E8]{Engel1998}, it follows that $|\M(A)|=(2n)!!$ (see~\eqref{eq:pftheo1}).
	\end{definition}
	
	Finally, we define 
	\begin{equation}
		\begin{aligned}\label{eq:reducedex}
			U_{N,\ib,\jb}:=\frac{1}{N}\sum_{k\notin (i_1,\ldots ,i_p,j_1,\ldots ,j_q)} c_k^2(g(\si_k)^2-t_k^2), \\
			V_{N,\ib,\jb}:=\frac{1}{N}\sum_{m\neq k, (m,k)\notin (i_1,\ldots ,i_p,j_1,\ldots ,j_q)} c_m c_k (g(\si_k)-t_k)(t_m^{k}-t_m).
		\end{aligned}
	\end{equation}
	In particular $U_{N,\ib,\jb}$, $V_{N,\ib,\jb}$ are analogs of $U_N$, $V_N$ from \eqref{eq:randvar} after removing the indices $(i_1,\ldots ,i_p,j_1, \ldots ,j_q)$. 
	We now state a lemma which will be useful in proving~\cref{theo:conmainder}. Its proof is deferred to~\cref{sec:pfsurviveterm}.
	\begin{lemma}\label{lem:surviveterm}
		Fix $k,k_1,k_2,p,q\in  \mathbb{N}\cup \{0\}$, and define  
		\begin{align}\label{eq:cpqk}
			\mathcal{C}_{p,q,k}:= \{(\el_1,\el_2,\ldots ,\el_p,q):\el_i\geq 2\ \forall i\in [p],q+\sum_{i=1}^p \el_i=k\}.
		\end{align}
		Also, for $r,N\in \mathbb{N}$ with $r\le N$  set 
		\begin{align}\label{eq:thnr}
			\Theta_{N,r}:=\{(a_1,a_2,\ldots ,a_r)\in [N]^r: a_i\neq a_j\ \forall i\neq j\}
		\end{align}
		to be the set of all distinct $r$ tuples from $[N]^r$.  For any $(i_1,\ldots ,i_p,j_1,\ldots ,j_q)\in\Theta_{N,p+q}$ and $L:=(\el_1,\el_2,\ldots ,\el_p,q)\in \mathcal{C}_{p,q,k}$, set $$h_{i_r}(\ms):=(c_{i_r}(g(\sigma_{i_r})-t_{i_r}))^{\el_r},\,\, \mbox{for}\,\, r\in [p],\quad \mbox{and} \quad \tih_{j_r}(\ms):=c_{j_r}(g(\sigma_{j_r})-t_{j_r})\,\, \mbox{for}\,\, r\in [q].$$Recall the definitions of $U_N,V_N$ from \eqref{eq:randvar} and $U_{N,\ib,\jb}, V_{N,\ib,\jb}$ from \eqref{eq:reducedex}. Next consider the function $f_L(\cdot):\B^N\to\R$ defined as,
		\begin{equation}\label{eq:root}
			f_L(\ms):=\sum_{(\ib,\jb)\in\Theta_{N,p+q}}\left(\prod\limits_{r=1}^p h_{i_r}(\ms)\right)\left(\prod\limits_{r=1}^q \tih_{j_r}(\ms)\right)U_{N,\ib,\jb}^{k_1} V_{N,\ib,\jb}^{k_2}.
		\end{equation}
		Then the following conclusions hold under Assumptions~\ref{as:cmean}~and~\ref{as:coeffalt}: 
		\begin{enumerate}
			\item[(a)] There exists a universal constant $0<C<\infty$ (free of $N$, only depending on the upper bounds in Assumptions \ref{as:cmean}~and~\ref{as:coeffalt}) such that the rank of $C^{-1}f_L(\cdot)$, i.e., $\ns(C^{-1} f_L)\leq\floor{\frac{k-1}{2}}$,  if either $\exists\ i$ such that $ \el_i>2$, or if $q$ is an odd number.
			\item[(b)] Suppose $q$ is an even number and $\el_i=2$ for $1\leq i\leq p$. Consider the function $\tf_L(\cdot):\B^N\to\R$ given by,
			\begin{align*}
				\tf_L(\ms)&:=\sum_{\mathfrak{m}\in\M([q])} \sum\limits_{(\ib,\jb)\in \Theta_{N,p+q}}\left(\prod_{r=1}^p c_{i_r}^2\left(g(\sigma_{i_r})-t_{i_r}\right)^2\right)\Bigg(\prod_{r=1}^{q/2} \Bigg(c_{j_{\mathfrak{m}_{r,1}}} c_{j_{\mathfrak{m}_{r,2}}}\nonumber \\ & \Big(g(\sigma_{j_{\mathfrak{m}_{r,1}}})- t_{j_{\mathfrak{m}_{r,1}}}\Big) \Big(t_{j_{\mathfrak{m}_{r,2}}}^{j_{\mathfrak{m}_{r,1}}}-t_{j_{\mathfrak{m}_{r,2}}}\Big)\Bigg)\Bigg)U_N^{k_1} V_N^{k_2}.
			\end{align*}
			Then $\E [N^{-k/2}f_L(\ms)] \lra \E[N^{-k/2}\tf_L(\ms)]$ (according to~\eqref{eq:pftheo2}).
		\end{enumerate}
	\end{lemma}
	It is important to note that~\cref{lem:surviveterm} holds without the empirical convergence condition as in~\cref{as:empcon}. We now outline why \cref{lem:surviveterm} is useful for proving \cref{theo:conmainder}. For $k\in\mathbb{N}$, define 
	\begin{align}\label{eq:unag}
		\mathcal{C}_k:= \cup_{p,q\in\mathbb{N}\cup\{0\}}\mathcal{C}_{p,q,k}.
	\end{align}
	Next, we observe that by a standard multinomial expansion we have: 
	\begin{equation}\label{eq:conmain1}
		\E(T_N^k)=\frac{1}{N^{k/2}}\E\Bigg[\sum_{L=(\el_1,\ldots ,\el_p,q) \in \mathcal{C}_k}D(L)f_L(\ms)\Bigg]
	\end{equation}
	where $D(L)\equiv D(\el_1,\ldots ,\el_p,q)$ denotes the coefficient of the term $\left(\prod\limits_{r=1}^p (c_r(g(\sigma_r)-t_r))^{\el_r}\right)\left(\prod\limits_{r=1}^q (c_{p+r}(g(\sigma_{p+r})-t_{p+r}))\right)$ in~\eqref{eq:conmain1}. It is possible to write out $D(\ell_1,\ldots ,\ell_p,q)$ in terms of standard multinomial coefficients, but that is not necessary for the proof for general $(\ell_1,\ldots ,\ell_p,q)$, so we avoid including it here. Next, we make the following simple note.
	\begin{observation}\label{obs:conmain1}
		The outer sum in \eqref{eq:conmain1} is a sum over finitely many indices $|\mathcal{C}_k|$ (depending only on $k$) and $\sup_{L\in\mathcal{C}_k} D(L)$ is finite (depending only on $k$). 
	\end{observation}
	First suppose that $L=(\ell_1,\ldots ,\ell_p,q)$ where there exists some $\ell_i>2$. Then for all such summands, \cref{lem:surviveterm} yields that 
	$$\frac{1}{N^{k/2}}\E|f_L(\ms)|\lesssim \frac{N^{\lfloor \frac{k-1}{2}\rfloor}}{N^{k/2}} \to 0.$$
	The same comment also applies if $q$ is odd. By \cref{obs:conmain1}, the total contribution of all such summands is therefore asymptotically negligible. The only case left is to consider $L\in\mathcal{C}_k$ of the form 
	\begin{align}\label{eq:intuitsurvive}
		(\underbrace{2,2,\ldots ,2}_{p\,\mbox{times}},q),
	\end{align}
	where $q$ is even and $2p+q=k$. \cref{lem:surviveterm}, part (b), now implies that in all such summands, we can replace $f_L$ by $\tilde{f}_L$. The argument now boils down to calculating $D(L)$ and finding the limit of $N^{-k/2}\E \tilde{f}_L(\ms)$ for all $L\in\mathcal{C}_k$ of the same form as \eqref{eq:intuitsurvive}. Using a simple combinatorial argument, it is easy to check that the quantity $D(\underbrace{2,2,\ldots ,2}_{p\textrm{ times}},q)$ with $p=(k-q)/2$ is given by:
	\begin{equation}\label{eq:conmain10}
		\frac{1}{p!}\left[\binom{k}{2}\cdot \binom{k-2}{2}\cdot \binom{k-4}{2} \ldots \binom{k-2p+2}{2}\right].
	\end{equation} 
	The limit of $\tilde{f}_L$ is derived from \cref{as:empcon} and \cref{lem:auxtail}. The formal steps for the proof are provided below.
	
	\begin{proof}[Proof of \cref{theo:conmainder}]
		
		We break the proof into two cases.
		
		\noindent \emph{(a) $k$ is odd}: Let us define,
		$$\tilde{\mathcal{C}}_k:=\cup_{p,q\in\mathbb{N}\cup\{0\},\ q\textrm{ is even}}\left\{(\el_1,\el_2,\ldots ,\el_p,q): \el_i=2\ \forall i\in [p], \sum_{i=1}^p \el_i+q=k\right\}.$$
		Fix any $\tilde{L}=(\el_1,\el_2,\ldots ,\el_p,q)\in\mathcal{C}_k$ satisfying $\sum_{i=1}^p \el_i+q=k$. As $k$ is odd, either (i) $q$ is odd or (ii) $\sum_{i=1}^p \el_i$ is odd which in turn implies that there exists $i\in [p]$ such that $\el_i\geq 3$. Therefore any such $\tilde{L}$ belongs to $\mathcal{C}_k\setminus \tilde{\mathcal{C}}_k$. Using~\cref{lem:surviveterm}(a),~\cref{obs:conmain1}~and the expansion in~\eqref{eq:conmain1}, we consequently get:
		\begin{equation}\label{eq:conmain2}
			|\E(T_N^k)|\leq \frac{1}{N^{k/2}}\sum_{\substack{L=(\el_1,\ldots ,\el_p,q)\\ \in \mathcal{C}_k\setminus\tilde{\mathcal{C}}_k}}D(L)|\E f_L(\ms)|\lesssim \frac{N^{\floor{(k-1)/2}}}{N^{k/2}}.
		\end{equation}
		The right hand side above converges to $0$ as $N\to\infty$ which implies $\E (T_N^k)\overset{N\to\infty}{\longrightarrow}0$.
		
		\noindent \emph{(b) $k$ is even}: 
		
		Recall the notion of $\lra$ from~\eqref{eq:pftheo2} and $\overset{P}{\leftrightarrow}\,\equiv\, \overset{\P}{\leftrightarrow}$ from \eqref{eq:pftheo3}. Using~\eqref{eq:conmain1}, we have:
		$$\E(T_N^k)=\frac{1}{N^{k/2}}\E\Bigg[\sum_{L=(\el_1,\ldots ,\el_p,q) \in \tilde{\mathcal{C}}_k}D(L)f_L(\ms)\Bigg]+\frac{1}{N^{k/2}}\E\Bigg[\sum_{L=(\el_1,\ldots ,\el_p,q)\in \mathcal{C}_k\setminus\tilde{\mathcal{C}}_k}D(L)f_L(\ms)\Bigg].$$
		The second term in the above display converges to $0$ as $N\to\infty$ by~\eqref{eq:conmain2}. Then~\cref{lem:surviveterm}(b) implies that,
		\begin{align}\label{eq:conmain3}
			\E(T_N^k)&\lra\frac{1}{N^{k/2}}\E\Bigg[\sum_{L=(\el_1,\ldots ,\el_p,q) \in \tilde{\mathcal{C}}_k}D(L)f_L(\ms)\Bigg] \nonumber \\ &\lra \frac{1}{N^{k/2}}\E\Bigg[\sum_{L=(\el_1,\ldots ,\el_p,q) \in \tilde{\mathcal{C}}_k}D(L)\tf_L(\ms)\Bigg].
		\end{align}
		As $q$ is even for $(\el_1,\ldots ,\el_p,q) \in \tilde{\mathcal{C}}_k$, we have  $|\M([q])|=q!!$ (see~\cref{def:matching}). Recall the expression of $\tf_L(\ms)$ from~\cref{lem:surviveterm}(b). By symmetry, we have:
		\begin{align}\label{eq:conmain4}
			N^{-k/2}\E \tf_L(\ms)&=q!! N^{-k/2}\sum\limits_{(\ib,\jb)\in \Theta_{N,p+q}}\E\Bigg[\left(\prod_{r=1}^p c_{i_r}^2(g(\si_{i_r})-t_{i_r})^2\right)\Bigg(\prod_{r=1}^{q/2} \Bigg( c_{j_{2r-1}} c_{j_{2r}}\nonumber \\ &\qquad\qquad \Big(g(\sigma_{j_{2r-1}})-t_{j_{2r-1}}\Big) \Big(t_{j_{2r}}^{j_{2r-1}}-t_{j_{2r}}\Big)\Bigg)U_N^{k_1} V_N^{k_2}\Bigg].
		\end{align}
		
		Recall from \eqref{eq:boundZ1} and \eqref{eq:boundz2} that 
		\begin{equation}\label{eq:conmain6}
			\sum_{i_1} \frac{c_{i_1}^2(g(\si_{i_1})-t_{i_1})^2}{N}\lesssim 1, \quad \sum_{j_1,j_2} \frac{c_{j_1}c_{j_2} \big|g(\sigma_{j_1})-t_{j_1}\big| \big|t_{j_{2}}^{j_1}-t_{j_2}\big|}{N}\lesssim 1.
		\end{equation}
		Also, on the set $\tilde{\mathcal{C}}_k$, we clearly have $k/2=p+q/2$. As $U_N,V_N$ are uniformly bounded, therefore, \eqref{eq:conmain6} implies that 
		\begin{align}\label{eq:conmainunif}
			&\;\;\;\;\frac{1}{N^{k/2}}\sum\limits_{\substack{(i_1,\ldots ,i_p,\\ j_1,\ldots ,j_q)\in \Theta_{N,p+q}}}\left(\prod_{r=1}^p c_{i_r}^2(g(\si_{i_r})-t_{i_r})^2\right)\Bigg(\prod_{r=1}^{q/2} \Bigg(c_{j_{2r-1}}c_{j_{2r}} \big|g(\sigma_{j_{2r-1}})-t_{j_{2r-1}}\big| \big|t_{j_{2r}}^{j_{2r-1}}-t_{j_{2r}}\big|\Bigg)\Bigg)U_N^{k_1}V_N^{k_2} 
			\nonumber \\ &\le \left(\prod_{r=1}^p\left(\frac{1}{N}\sum_{i_r}c_{i_r}^2(g(\sigma_{i_r})-t_{i_r})^2\right)\right)\left(\prod_{r=1}^{q/2}\left(\frac{1}{N}\sum_{j_{2r-1},j_{2r}}c_{j_{2r-1}}c_{j_{2r}}\big|g(\sigma_{j_{2r-1}})-t_{j_{2r-1}}\big|\big|t_{j_{2r}}^{j_{2r-1}}-t_{j_{2r}}\big|\right)\right)U_N^{k_1}V_N^{k_2} \nonumber \\ &\lesssim 1.
		\end{align}
		Therefore the random variable on the left hand side of \eqref{eq:conmainunif} is uniformly bounded. So, to study the limit of its expectation as in \eqref{eq:conmain4}, by appealing to the dominated convergence theorem, it suffices to study its weak limit. In this spirit, we will now prove the following:    
		\begin{align}\label{eq:conmain7}
			&\;\;\;\;\sum\limits_{(\ib,\jb)\in \Theta_{N,p+q}}\left(\prod_{r=1}^p \frac{c_{i_r}^2(g(\si_{i_r})-t_{i_r})^2}{N}\right)\Bigg(\prod_{r=1}^{q/2} \Bigg(\frac{c_{j_{2r-1}}c_{j_{2r}} \Big(g(\sigma_{j_{2r-1}})-t_{j_{2r-1}}\Big) \Big(t_{j_{2r}}^{j_{2r-1}}-t_{j_{2r}}\Big)}{N}\Bigg)\Bigg)U_N^{k_1}V_N^{k_2} \nonumber \\ & \overset{w}{\rightarrow} P_1^{p+k_1} P_2^{q/2+k_2}, 
		\end{align}
		where $(P_1,P_2)$ are defined in \cref{as:empcon}. 
		
		To address the weak limit in ~\eqref{eq:conmain7}, we first note the following identity:
		$$|c_i^2(g(\si)-t_i)^2-c_i^2(g(\si_i)^2-t_i^2)|\lesssim |(g(\si_i)-t_i)t_i|$$
		by using \cref{as:coeffalt}. 
		Therefore, by~\cref{lem:auxtail}(b), we have:
		\begin{align}\label{eq:onestepequiv}
			\frac{1}{N}\E\Bigg|\sum_{i_1} c_{i_1}^2(g(\si_{i_1})-t_{i_1})^2-\sum_{i_1} c_{i_1}^2 (g(\si_{i_1})^2-t_{i_1}^2)\Bigg|\lesssim \frac{1}{N}\E\Bigg|\sum_{i_1} (g(\si_{i_1})-t_{i_1})t_{i_1}\Bigg|\longrightarrow 0.
		\end{align}
		Combining the above observation with~\eqref{eq:conmain6}, we observe that 
		\begin{align*}
			&\;\;\;\frac{1}{N^{k/2}}\sum\limits_{\substack{(i_1,\ldots ,i_p,\\ j_1,\ldots ,j_q)\in \Theta_{N,p+q}}}\left(\prod_{r=1}^p c_{i_r}^2(g(\si_{i_r})-t_{i_r})^2\right)\Bigg(\prod_{r=1}^{q/2} \Bigg(c_{j_{2r-1}}c_{j_{2r}} \big(g(\sigma_{j_{2r-1}})-t_{j_{2r-1}}\big) \big(t_{j_{2r}}^{j_{2r-1}}-t_{j_{2r}}\big)\Bigg)\Bigg)U_N^{k_1}V_N^{k_2} \\ &\leftrightarrow 
			\left(\prod_{r=1}^p\left(\frac{1}{N}\sum_{i_r}c_{i_r}^2(g(\sigma_{i_r})-t_{i_r})^2\right)\right)\left(\prod_{r=1}^{q/2}\left(\frac{1}{N}\sum_{j_{2r-1},j_{2r}}c_{j_{2r-1}}c_{j_{2r}}\big(g(\sigma_{j_{2r-1}})-t_{j_{2r-1}}\big)\big(t_{j_{2r}}^{j_{2r-1}}-t_{j_{2r}}\big)\right)\right)U_N^{k_1}V_N^{k_2} \\ &\overset{\P}{\leftrightarrow} \left(\prod_{r=1}^p\left(\frac{1}{N}\sum_{i_r}c_{i_r}^2(g(\sigma_{i_r})^2-t_{i_r}^2)\right)\right)\left(\prod_{r=1}^{q/2}\left(\frac{1}{N}\sum_{j_{2r-1},j_{2r}}c_{j_{2r-1}}c_{j_{2r}}\big(g(\sigma_{j_{2r-1}})-t_{j_{2r-1}}\big)\big(t_{j_{2r}}^{j_{2r-1}}-t_{j_{2r}}\big)\right)\right)U_N^{k_1}V_N^{k_2} \\ &\overset{w}{\longrightarrow} P_1^{p+k_1} P_2^{q/2+k_2}.
		\end{align*}
		Here the first equivalence follows from \eqref{eq:conmain6}, the second equivalence follows from \eqref{eq:onestepequiv}, and the final weak limit follows from a direct application of \cref{as:empcon}. This establishes \eqref{eq:conmain7}. 
		
		\vspace{0.1in}
		
		Let us now put the pieces together by studying the limit of the expectation in \eqref{eq:conmain3}. First we recall the identity involving $\tilde{f}_L(\ms)$ in \eqref{eq:conmain4}. By using the dominated convergence theorem along with \eqref{eq:conmain7}, we get:
		
		$$\frac{1}{N^{k/2}}\E [\tf_L(\ms)]\overset{N\to\infty}{\longrightarrow} q!! \EE\left[P_1^{p+k_1} P_2^{q/2+k_2}\right]$$
		for $L=(\el_1,\el_2,\ldots ,\el_p,q)\in \tilde{\mathcal{C}}_k$. Plugging the above observation in~\eqref{eq:conmain3}, we then get:
		\begin{align}\label{eq:conmain9}
			\E(T_N^k)\overset{N\to\infty}{\longrightarrow}\sum_{l=(l_1,l_2,\ldots ,l_p,q)\in \tilde{\mathcal{C}}_k} D(L)q!! \EE\left[P_1^p P_2^{q/2}\right].
		\end{align}
		
		Using~\eqref{eq:conmain10} and the identity  $2p+q=k$, we further get:
		\begin{equation}\label{eq:conmain11}
			q!!D(L)=\frac{1}{p!(q/2)!}\left[\binom{k}{2}\cdot \binom{k-2}{2}\ldots \binom{k-2p+2}{2}\cdot \binom{q}{2}\cdot \binom{q-2}{2}\ldots \binom{2}{2}\right]=k!!\cdot \binom{k/2}{q/2}.
		\end{equation}
		Finally, combining~\eqref{eq:conmain9},~\eqref{eq:conmain11}, and the identity $2p+q=k$, we get:
		\begin{align*}
			\E(T_N^{k})&\overset{N\to\infty}{\longrightarrow}k!!\cdot \EE\left[P_1^{k_1} P_2^{k_2} \sum_{q=0,\ q\textrm{ is even}}^k \cdot \binom{k/2}{q/2}P_1^{(k-q)/2} P_2^{q/2}\right]\\ &=k!!\EE\left[P_1^{k_1}P_2^{k_2}\sum_{r=0}^{k/2}\binom{k/2}{r}P_1^{(k/2-r)} P_2^{r}\right]=k!!\cdot \EE[(P_1 + P_2)^{k/2}P_1^{k_1}P_2^{k_2}].
		\end{align*}	
		This completes the proof.
	\end{proof}
	
	\subsection{Proof of~\cref{lem:smoothcont}}
	
	In order to prove \cref{lem:smoothcont}, we will use the following discrete Fa\`{a} Di Bruno type formula (see \cite{faa1855sullo}) whose proof is provided alongside the statement. 
	
	\begin{lemma}\label{cl:smoothclaim2}
		Set $\cS_k=\{j_1,j_2,\ldots ,j_k\}\subseteq [N]^k$, $k\ge 1$. Consider an arbitrary function $w:\B^N\to\R$. Suppose that the function $f:\R\to [-1,1]$ has $k$ continuous and uniformly bounded derivatives. Then we have:
		\begin{align}\label{eq:smoothcont1}
			&\;\;\;\;\De(f\circ w;\tcS;\cS_{\tk})\nonumber \\&=\sum_{D\subseteq \cS_{\tk}\setminus \{j_1\}}\int_0^1 \De(w;D\cup \tcS;\cS_{\tk}\setminus D)\De(f'(w^{j_1}+z(w-w^{j_1}));\tcS;D)\,dz,
		\end{align}
		for all $1\leq\tk\leq k$ and all $\tcS\subseteq [N]$ such that $\tcS\cap \cS_{k}=\phi$.
	\end{lemma}
	
	\begin{proof}
		This proof proceeds through induction. 
		
		\emph{$\tk=1$ case.} In this case, the LHS of~\eqref{eq:smoothcont1} is $\De(f\circ w;\tcS;\{j_1\})=f(w^{\tcS})-f(w^{\tcS\cup\{j_1\}})$. Now by the Fundamental Theorem of Calculus, it is easy to check that
		\begin{equation}\label{eq:smoothcont2}
			f(w^{\tcS})-f(w^{\tcS\cup\{j_1\}})=\int_0^1 (w^{\tcS}-w^{\tcS\cup\{j_1\}})\ f'(w^{\tcS\cup\{j_1\}}+z(w^{\tcS}-w^{\tcS\cup\{j_1\}}))\,dz.
		\end{equation}
		Next observe that in the RHS of~\eqref{eq:smoothcont1}, when $\tk=1$, the only permissible choice of $D$ in the summation is $D=\phi$. In this case, $$\De(w;\tcS;\cS_{\tk})=\De(w;\tcS;\{j_1\})=w^{\tcS}-w^{\tcS\cup\{j_1\}}$$ and $$\De(f'(w^{j_1}+z(w-w^{j_1}));\tcS;D)=\De(f'(w^{j_1}+z(w-w^{j_1}));\tcS;\phi)=f'(w^{\tcS\cup\{j_1\}}+z(w^{\tcS}-w^{\tcS\cup\{j_1\}})).$$ Plugging these observations into the RHS of~\eqref{eq:smoothcont1} immediately yields that~\eqref{eq:smoothcont1} holds for $\tk=1$.
		
		\emph{Induction hypothesis for $\tk\leq\sk$.} Next assume that~\eqref{eq:smoothcont1} holds for all $\tk\leq\sk$ for $\sk<k$ and all $\tcS$ such that $\tcS\cap \cS_k=\phi$. We will next prove that~\eqref{eq:smoothcont1} holds for $\tk=\sk+1$ to complete the induction.
		
		\emph{$\tk=\sk+1$ case.} By the induction hypothesis,~\eqref{eq:smoothcont1} holds for $\tk\le\sk$. We will also need the following crucial property of the $\De(\cdot;\cdot;\cdot)$ operator: Given any $\eta:\B^N\to\R$, $j\in[N]$, $D_1,D_2\subseteq [N]$, $j\notin D_1$, $j\notin D_2$, and $D_1\cap D_2=\phi$, we have:
		\begin{align}\label{eq:smoothcont4}
			\De(\eta;D_1;D_2)-\De(\eta;D_1\cup\{j\};D_2)=\De(\eta;D_1;D_2\cup\{j\}).
		\end{align}
		The proof of the above property is deferred to the end of the current proof.
		
		Next we observe that:
		\begin{align}\label{eq:smoothcont3}
			&\;\;\;\;\;\De(f\circ w;\tcS;\cS_{\sk+1})\nonumber \\&\overset{(i)}{=}\De(f\circ w;\tcS;\cS_{\sk})-\De(f\circ w;\tcS\cup \{j_{\sk+1}\};\cS_{\sk})\nonumber \\&\overset{(ii)}{=}\sum_{D\subseteq \cS_{\sk}\setminus \{j_1\}}\int_0^1 \De(w;D\cup \tcS;\cS_{\sk}\setminus D)\De(f'(w^{j_1}+z(w-w^{j_1}));\tcS;D)\,dz\nonumber \\ &-\sum_{D\subseteq \cS_{\sk}\setminus \{j_1\}}\int_0^1 \De(w;\tcS\cup D\cup \{j_{\sk+1}\};\cS_{\sk}\setminus D)\De(f'(w^{j_1}+z(w-w^{j_1}));\tcS\cup\{j_{\sk+1}\};D)\,dz.
		\end{align}
		Here (i) follows by using~\eqref{eq:smoothcont4} with $\eta=f\circ w$, $D_1=\tcS$, $j=j_{\sk+1}$, and $D_2=\cS_{\sk}$, while (ii) follows directly from the induction hypothesis. Next note that 
		\begin{align}\label{eq:smoothcont5}
			&\;\;\;\;\De(w;D\cup \tcS;\cS_{\sk}\setminus D)\nonumber \\ &=\De(w;\tcS\cup D\cup \{j_{\sk+1}\};\cS_{\sk}\setminus D)+\De(w;D\cup \tcS;\cS_{\sk+1}\setminus D),
		\end{align}
		and 
		\begin{align}\label{eq:smoothcont6}
			&\;\;\;\;\;\De(f'(w^{j_1}+z(w-w^{j_1}));\tcS;D)\nonumber \\ &=\De(f'(w^{j_1}+z(w-w^{j_1}));\tcS\cup \{j_{\sk+1}\};D)+\De(f'(w^{j_1}+z(w-w^{j_1}));\tcS;D\cup \{j_{\sk+1}\}),
		\end{align}
		by once again invoking~\eqref{eq:smoothcont4} with $\eta=w$ (for~\eqref{eq:smoothcont5}) or $f'(w^{j_1}+z(w-w^{j_1}))$ (for~\eqref{eq:smoothcont6}), $D_1=D\cup\tcS$ (for~\eqref{eq:smoothcont5}) or $\tcS$ (for~\eqref{eq:smoothcont6}), $j=j_{\sk+1}$ (for~\eqref{eq:smoothcont5}~and~\eqref{eq:smoothcont6}), and $D_2=\cS_{\sk+1}\setminus D$ (for~\eqref{eq:smoothcont5}) or $D$ (for~\eqref{eq:smoothcont6}). Plugging the above observation into~\eqref{eq:smoothcont3}, we further have:
		\begin{align*}
			&\;\;\;\;\;\De(f\circ w;\tcS;\cS_{\sk+1})\nonumber \\&=\sum_{D\subseteq \cS_{\sk}\setminus \{j_1\}}\int_0^1 \De(w;\tcS\cup D\cup \{j_{\sk+1}\};\cS_{\sk+1}\setminus (D\cup\{j_{\sk+1}\}))\De(f'(w^{j_1}+z(w-w^{j_1}));\tcS;D\cup \{j_{\sk+1}\})\,dz\nonumber \\ & \qquad \qquad +\sum_{D\subseteq \cS_{\sk}\setminus \{j_1\}}\int_0^1 \De(w;D\cup \tcS;\cS_{\sk+1}\setminus D)\De(f'(w^{j_1}+z(w-w^{j_1}));\tcS;D)\,dz\\ &=\sum_{D\subseteq \cS_{\sk+1}\setminus \{j_1\}}\int_0^1 \De(w;D\cup \tcS;\cS_{\sk+1}\setminus D)\De(f'(w^{j_1}+z(w-w^{j_1}));\tcS;D)\,dz.
		\end{align*}
		This establishes~\eqref{eq:smoothcont1} for $\tk=\sk+1$ and completes the proof of~\cref{cl:smoothclaim2} by induction. Therefore, it only remains to prove \eqref{eq:smoothcont4}.
		
		\emph{Proof of \eqref{eq:smoothcont4}.} Observe that, as $j\notin D_1\cup D_2$, we get: 
		\begin{align*}
			\De(\eta;D_1;D_2\cup\{j\}) &=\sum_{D\subseteq D_2\cup\{j\}} (-1)^{|D|}\eta(\ms_{D_1\cup D}) \\ &=\sum_{D\subseteq D_2} (-1)^{|D|}\eta(\ms_{D_1\cup D})+\sum_{D\subseteq D_2} (-1)^{|D\cup\{j\}|}\eta(\ms_{D_1\cup (D_2\cup\{j\})})\\ &=\De(\eta;D_1;D_2)-\De(\eta;D_1\cup\{j\};D_2).
		\end{align*}
		This completes the proof.
	\end{proof}
	Next we show how bounds on discrete differences for the function $w$ can be converted into bounds on discrete differences for $f\circ w$, provided the derivatives of $f(\cdot)$ are bounded. To wit, suppose that $\{\cT_{N,k}\}_{N,k\ge 1}$ is a collection of tensors of dimension $N\times \ldots \times N$ ($k$-fold product), with non-negative entries. We assume that 
	\begin{align}\label{eq:Tprop}
		\sup_{N\ge 1}\sum_{j_1,\ldots ,j_k} \cT_{N,k}(j_1,\ldots ,j_k) \le \alpha_k, 
	\end{align}
	for finite positive reals $\alpha_k$. Let us define 
	\begin{align}\label{eq:newTprop}
		&\;\;\;\;\;\tcT_{N,k}(j_1,j_2,\ldots ,j_k)\nonumber \\ &:=\cT_{N,k}(j_1,j_2,\ldots ,j_k)+\sum_{\substack{D\subseteq \{j_1,j_2,\ldots ,j_k\},\\ |D|\leq k-1,\ D\neq \phi}} \tcT_{N,|D|}(D)\cT_{N,k-|D|}(\{j_1,\ldots ,j_k\}\setminus D),
	\end{align}
	where, by convention, $\tcT_{N,1}(j_1)=\cT_{N,1}(j_1)$ for $j_1\in [N]$. 
	\begin{lemma}\label{lem:genmatrixbd}
		(1). For all functions $w:\B^N\to [-1,1]$ satisfying 
		\begin{equation}\label{eq:smoothcont7}
			|\De(w;\tcS;\scS)|\leq C \cT_{N,\tk}(\scS),\quad \sup_{\ms\in\B^N} |w(\ms)|\leq 1,
		\end{equation} 
		for any set $\scS\subseteq \cS_k=\{j_1,\ldots ,j_k\}$, $|\scS|=\tk$, $1\leq \tk\leq k$,  $\tcS\cap\scS=\phi$, and $C>1$, the following holds
		\begin{equation}\label{eq:smoothcont8}
			|\De(f\circ w;\tcS;\scS)|\leq C^{\tk} \tcT_{N,\tk}(\scS), 
		\end{equation}
		for any $f:[-1,1]\to\R$, $\sup_{|x|\leq 1}|f^{\ell}(x)|\leq 1$, $0\le \ell\le \tk$. 
		
		\noindent (2). Suppose \eqref{eq:Tprop} holds. Then there exists finite positive reals $\tal_k$ such that 
		$$\sup_{N\ge 1} \sum_{j_1,\ldots ,j_k} \tcT_{N,k}(j_1,\ldots ,j_k)\le \tal_k.$$
	\end{lemma}
	
	\begin{proof}
		\emph{Part (1).} Using~\cref{cl:smoothclaim2}, the proof will proceed  via induction on $\tk$, $1\leq \tk\leq k$.
		
		\emph{$\tk=1$ case.} In this case, say $\scS=\{j_{\ell}\}$ for some  $\ell\geq 1$. Suppose that~\eqref{eq:smoothcont7} holds. Observe that $$|\De(f\circ w;\tcS;\scS)|=|f(w^{\tcS})-f(w^{\tcS\cup\{j_{\ell}\}})|\leq |w^{\tcS}-w^{\tcS\cup\{j_{\ell}\}}|=|\De(w;\tcS;\scS)|\leq C \cT_{N,1}(j_{\ell}).$$
		Recall that $\tcT_{N,1}=\cT_{N,1}$. Therefore~\eqref{eq:smoothcont8} holds for $\tk=1$ provided~\eqref{eq:smoothcont7} holds.
		
		\emph{Induction hypothesis for $\tk\leq\sk$.} Next assume that~\eqref{eq:smoothcont8} holds for all $\tk\leq\sk (<k)$  provided~\eqref{eq:smoothcont7} holds. We will next prove~\eqref{eq:smoothcont8} under~\eqref{eq:smoothcont7} for $\tk=\sk+1$ to complete the induction.
		
		\emph{$\tk=\sk+1$ case.} Suppose $\scS\subseteq \cS_k$, $|\scS|=\sk+1$, $2\leq \sk+1\leq k$, $\tcS\cap\scS=\phi$. Without loss of generality,  assume that $\scS=\{j_1,j_2,\ldots ,j_{\sk+1}\}$. By~\eqref{eq:smoothcont1}, observe that:
		\begin{align}\label{eq:smoothcont9}
			&\;\;\;\;|\De(f\circ w;\tcS;\scS)|\\ &\leq |\De(w;\tcS;\scS)| + \sum_{\substack{D\subseteq \scS\setminus \{j_1\},\\ D\neq \phi}}\int_0^1 |\De(w;D\cup \tcS;\scS\setminus D)||\De(f'(w^{j_1}+z(w-w^{j_1}));\tcS;D)|\,dz\nonumber \\&\leq \cT_{N,1+\sk}(j_1,j_2,\ldots ,j_{\sk+1})+C\sum_{D\subseteq \scS\setminus \{j_1\},\ D\neq \phi} \tQ_{N,\sk+1-|D|}(\scS\setminus D)\nonumber \\&\qquad \qquad \int_0^1 |\De(f'(w^{j_1}+z(w-w^{j_1}));\tcS;D)|\,dz,
		\end{align}
		where the last line follows by invoking~\eqref{eq:smoothcont7} for $\tk=\sk+1$.
		
		Next observe that the $\De(\cdot;\cdot;\cdot)$ operator is linear in its first argument, i.e., $\De(\eta_1+\eta_2;\cdot;\cdot)=\De(\eta_1;\cdot;\cdot)+\De(\eta_2;\cdot;\cdot)$ where $\eta_1,\eta_2:\B^N\to\R$. Therefore, for any $z\in [0,1]$ and $D\subseteq \scS\setminus \{j_1\}$, we have:
		$$|\De(w^{j_1}+z(w-w^{j_1});\tcS;D)|\leq (1-z)|\De(w^{j_1};\tcS;D)|+z|\De(w;\tcS;D)|\leq C\cT_{N,|D|}(D),$$
		where the last line once again uses~\eqref{eq:smoothcont7} for $\tk=\sk+1$. Similarly $\sup_{\ms\in\B^N} |w^{j_1}+z(w-w^{j_1})|\leq 1$. Also note that $|D|\leq\sk$ for all $D\subseteq \scS\setminus \{j_1,j_2\}$. The above sequence of observations allows us to invoke the induction hypothesis with $\scS$ replaced with $D$, $f(\cdot)$ replaced by $f'(\cdot)$, and $w$ replaced by $w^{j_1}+z(w-w^{j_1})$. This implies 
		\begin{align*}
			\int_0^1 |\De(f'(w^{j_1}+z(w-w^{j_1}));\tcS;D)|\,dz \le C^{|D|}\tcT_{N,|D|}(D)\le C^{\sk}\tcT_{N,|D|}(D).
		\end{align*}
		The above display coupled with~\eqref{eq:smoothcont9} yields that 
		\begin{align*}
			&\;\;\;\;|\De(f\circ w;\tcS;\scS)|\\ &\leq \cT_{N,1+\sk}(j_1,j_2,\ldots ,j_{\sk+1})+C^{\sk+1}\sum_{D\subseteq \scS\setminus \{j_1\},\ D\neq \phi} \cT_{N,\sk+1-|D|}(\scS\setminus D)\tcT_{N,|D|}(D) \\ & \leq C^{\sk+1}\tcT_{N,1+\sk}(j_1,j_2,\ldots ,j_{\sk+1}).
		\end{align*}
		This completes the proof of part 1 by induction.
		
		\vspace{0.05in}
		
		\emph{Proof of part 2.} Recall the $\alpha_k$s from \eqref{eq:Tprop}. Define $\tal_1:=\alpha_1$ and for $k\geq 2$, set
		\begin{equation}\label{eq:smoothcont10}
			\tal_k:=\alpha_k+\sum_{0<j\leq k-1} {k\choose j} \tal_{j}\alpha_{k-j}.
		\end{equation}
		The proof proceeds via induction on $k$ with $\tal_k$ as defined in~\eqref{eq:smoothcont10}.
		
		\emph{$k=1$ case.} By \eqref{eq:Tprop}, $ \sum_{j_1} \tcT_{N,1}(j_1)= \sum_{j_1} \cT_{N,1}(j_1)\leq \alpha_1=\tal_1$. This establishes the base case.
		
		\emph{Induction hypothesis for $k\leq \sk$.} Suppose the conclusion in~\cref{lem:genmatrixbd}, part (2), holds for all $k\leq \sk$. We now prove the same $k=\sk+1$.
		
		\emph{$k=\sk+1$ case.} By using the definition of $\tcT$ from \eqref{eq:newTprop}, we have:
		\begin{align*}
			&\;\;\;\;\sup_{N\ge 1} \sum_{\{j_1,j_2,\ldots ,j_{\sk+1}\}} \tcT_{N,\sk+1}(j_1,j_2,\ldots ,j_{\sk+1})\\ &\le \sup_{N\ge 1} \sum_{\{j_1,j_2,\ldots ,j_{\sk+1}\}} \cT_{N,\sk+1}(j_1,j_2,\ldots ,j_{\sk+1})+\sup_{N\ge 1}\sum_{\substack{D\subseteq \{j_1,\ldots ,j_{\sk+1}\},\\ |D|\leq \sk,\ D\neq \phi}} \left(\sum_{\{j_t\in D\}}\tcT_{N,|D|}(D)\right)\\ &\qquad \qquad \left(\sum_{j_t\notin D}\cT_{N,\sk+1-|D|}(\{j_1,\ldots ,j_{\sk+1}\}\setminus D)\right).
		\end{align*}
		As $D$ is non-empty and $|D|\le \sk$, we have: 
		$$\sum_{\{j_t\in D\}}\tcT_{N,|D|}(D) \le \tal_{|D|}$$
		by the induction hypothesis and 
		$$\sum_{j_t\notin D}\cT_{N,\sk+1-|D|}(\{j_1,\ldots ,j_{\sk+1}\}\setminus D\cup \{j_1\}) \le \alpha_{\sk+1-|D|}$$
		as $\cT$ satisfies \eqref{eq:Tprop}. Combining the above observations, we get: 
		\begin{align*}
			&\;\;\;\sup_{N\ge 1} \sum_{\{j_1,\ldots ,j_{\sk+1}\}} \tcT_{N,\sk+1}(j_1,j_2,\ldots ,j_{\sk+1}) \nonumber \\ &\le \alpha_{\sk+1}+\sum_{D\subseteq \{j_1,\ldots ,j_{\sk+1}\},\ |D|\le \sk,\ D\ne \phi}\tal_{|D|}\alpha_{\sk+1-|D|} \nonumber \\ &\le \alpha_{\sk+1}+\sum_{j=1}^{\sk+1} {\sk+1 \choose j}\tal_{|D|}\alpha_{\sk+1-|D|} = \tal_{\sk+1}.
		\end{align*}
		This completes the proof by induction.
	\end{proof}
	
	\begin{proof}[Proof of \cref{lem:smoothcont}, parts 1 and 2] Recall that $\mathcal{R}[\cdot]$ is defined in \eqref{eq:newtensor}. Its symmetry follows from definition. The result follows by invoking parts 1 and 2 of \cref{lem:genmatrixbd} with $w\equiv b_{j_1}$, $\scS=\{j_2,\ldots ,j_k\}$, $\cT_{N,k-1}(\scS)=\tQ_{N,k}(j_1,\scS)$.
	\end{proof}	
	
	\section{Preliminaries and auxiliary results for proving~\cref{lem:surviveterm}}\label{sec:pfsurvive}
	
	This section is devoted to establishing the main ingredients for proving~\cref{lem:surviveterm}. The proof is based on a decision tree approach. In particular, we will begin with $f_L(\ms)$ from~\eqref{eq:root} as the \emph{root node} of the tree. Then we decompose the root into a number of \emph{child nodes} to form the \emph{first generation}. Next we will decompose each of the child nodes that do not satisfy a certain \emph{termination condition} into their own child nodes to form the \emph{second generation}, and so on. This process will continue till all the leaf nodes (with no children) satisfy the termination condition. 
	
	\subsection{Constructing the decision tree}\label{sec:construct}
	We begin the process of constructing the tree with a simple observation. First recall the definition of $\Theta_{N,p+q}$ from~\cref{lem:surviveterm} and consider the following proposition.
	\begin{prop}\label{prop:baseprop}
		Suppose $\tp,\tq\in \mathbb{N}$, $(i_1,\ldots ,i_{\tp},j_1,\ldots ,j_{\tq})\in\Theta_{N,{\tp+\tq}}$. We use $\ib=(i_1,\ldots ,i_p)$ and $\jb=(j_1,\ldots ,j_q)$ as shorthand. Let $\{h_{i_r}(\cdot)\}_{1\leq r\leq \tp}$, $\{h_{j_r}\}_{1\leq r\leq \tq}$, $U_{\ib,\jb}(\ms)$, $V_{\ib,\jb}(\ms)$ are functions from $\B^N\to\R$ such that  $h_{{\io}}(\ms):=c_{{\io}}(g(\sigma_{{\io}})-t_{{\io}})$ for some $\io\in \{j_1,j_2,\ldots ,j_{\tq}\}$. Then the following identity holds: 
		\begin{align}\label{eq:baseprop2}
			&\;\;\;\;\E\left[\left(\prod\limits_{r=1}^{\tp} h_{i_r}(\ms)\right)\left(\prod\limits_{r=1}^{\tq} h_{j_r}(\ms)\right)U_{\ib,\jb}(\ms) V_{\ib,\jb}(\ms)\right] \nonumber \\ &=\sum_{(\cD,\cE,\cU,\cV)\in\FG} \E\left[\left(\prod\limits_{r=1}^{\tp} h_{i_r}(\ms;\cD)\right)\left(\prod\limits_{r=1}^{\tq} h_{j_r}(\ms;\cE)\right)U_{\ib,\jb}(\ms;\cU) V_{\ib,\jb}(\ms;\cV)\right]
		\end{align}
		where
		\begin{equation}	
			\begin{aligned}
				&\FG:=\{(\cD,\cE,\cU,\cV):\ \cD\subseteq (i_1,\ldots ,i_{\tp}),\ \cE\subseteq (j_1,\ldots ,j_{\tq})\setminus\io,\ \cU\subseteq \{\iota\},\ \cV\subseteq \{\iota\}, \\ &\qquad \qquad (\cD,\cE,\cU,\cV)\neq ((i_1,\ldots ,i_{\tp}),(j_1,\ldots, j_{\tq})\setminus \io,\{\io\},\{\io\})\}, \label{eq:basepropa}
			\end{aligned}
		\end{equation}
		\begin{align}
			h_{i_r}(\ms;\cD):=h_{i_r}^{\io}(\ms)\mathbbm{1}(i_r\in\cD)+h_{i_r}^{\phi;\io}(\ms)\mathbbm{1}(i_r\in\bar{\cD}), \label{eq:basepropb}	
		\end{align}	
		\begin{align}
			h_{j_r}(\ms;\cE):=h_{j_r}^{\io}(\ms)\mathbbm{1}(j_r\in\cE)+h_{j_r}^{\phi;\io}(\ms)\mathbbm{1}(j_r\in \bar{\cE}), \; j_r\neq\io , \label{eq:basepropc}
		\end{align}
		\begin{align} h_{{\io}}(\ms;\cE):=c_{{\io}}(g(\sigma_{{\io}})-t_{{\io}}),\label{eq:basepropd}
		\end{align}
		\begin{align} U_{\ib,\jb}(\ms;\cU):=U_{\ib,\jb}^{\iota}(\ms)\ind(\iota\in\cU)+U_{\ib,\jb}^{\phi;\iota}(\ms)\ind(\iota\notin \cU),\qquad \label{eq:baseprope}
		\end{align}
		\begin{align} V_{\ib,\jb}(\ms;\cV):=V_{\ib,\jb}^{\io}(\ms)\ind(\io\in\cV)+V_{\ib,\jb}^{\phi;\io}(\ms)\ind(\io\notin\cV),\label{eq:basepropf}
		\end{align}
		and $\bar{\cE}:=((j_1,\ldots ,j_{\tq})\setminus {\io})\setminus \cE$ and $\bar{\cD}:=(i_1,\ldots ,i_{\tp})\setminus \cD$. Further for any fixed $D=(\cD,\cE,\cU,\cV)\in \FG$, we have:
		\begin{align}\label{eq:basepropnew1}	&\;\;\;\;\E\left[\left(\prod\limits_{r=1}^{\tp} h_{i_r}(\ms;\cD)\right)\left(\prod\limits_{r=1}^{\tq} h_{j_r}(\ms;\cE)\right)U_{\ib,\jb}(\ms;\cU)V_{\ib,\jb}(\ms;\cV)\right]\nonumber \\ &=\sum_{\tcE\subseteq \cE} \E\left[\left(\prod\limits_{r=1}^{\tp} h_{i_r}(\ms;\cD,\cD)\right)\left(\prod\limits_{r=1}^{\tq} h_{j_r}(\ms;\cE,\tcE)\right)U_{\ib,\jb}(\ms;\cU,\cU)V_{\ib,\jb}(\ms;\cV,\cV)\right],
		\end{align}
		where
		\begin{align}\label{eq:basepropnew2}
			h_{j_r}(\ms;\cE,\tcE):=h_{j_r}(\ms)\mathbbm{1}(j_r\in\tcE)-h_{j_r}^{\phi;\io}(\ms)\mathbbm{1}(j_r\in\cE\setminus\tcE)+h_{j_r}^{\phi;\io}(\ms)\mathbbm{1}(j_r\in\bar{\cE}), \; j_r\neq\io ,
		\end{align}
		\begin{align}\label{eq:basepropnew3}
			h_{i_r}(\ms;\cD,\cD):=h_{i_r}(\ms;\cD),\quad h_{{\io}}(\ms;\cE,\tcE):=h_{{\io}}(\ms;\cE)=c_{\iota}(g(\si_{\iota})-t_{\iota}),
		\end{align}
		\begin{align}\label{eq:basepropnew4}
			U_{\ib,\jb}(\ms;\cU,\cU):=U_{\ib,\jb}(\ms;\cU),\quad V_{\ib,\jb}(\ms;\cV,\cV):=V_{\ib,\jb}(\ms;\cV).
		\end{align}
	\end{prop} 
	\begin{proof}
		Observe that $h_{i_r}(\ms)=h_{i_r}^{\io}(\ms)+h_{i_r}^{\phi;\io}(\ms)$, $h_{j_r}(\ms)=h_{j_r}^{\io}(\ms)+h_{j_r}^{\phi;\io}(\ms)$, $U_{\ib,\jb}(\ms)=U_{\ib,\jb}^{\io}(\ms)+U_{\ib,\jb}^{\phi;\io}(\ms)$, and $V_{\ib,\jb}(\ms)=V_{\ib,\jb}^{\io}(\ms)+V_{\ib,\jb}^{\phi;\io}(\ms)$. Set $\fL:=(i_1,\ldots ,i_{\tp})$ and $\fM:=(j_1,\ldots ,j_{\tq})$. Therefore,
		\begin{align}\label{eq:baseprop1}
			&\;\;\;\;\E\left[\left(\prod\limits_{r=1}^{\tp} h_{i_r}(\ms)\right)\left(\prod\limits_{r=1}^{\tq} h_{j_r}(\ms)\right)U_{\ib,\jb}(\ms)V_{\ib,\jb}(\ms)\right]\nonumber \\ &=\E\bigg[\left(\prod\limits_{r=1}^{\tp} (h_{i_r}^{\io}(\ms)+h_{i_r}^{\phi;\io}(\ms))\right)\left(\prod\limits_{r=1}^{\tq} (h_{j_r}^{\io}(\ms)+h_{j_r}^{\phi;\io}(\ms))\right)(U_{\ib,\jb}^{\io}(\ms)+U_{\ib,\jb}^{\phi;\io}(\ms)) \nonumber \\ &\qquad \qquad (V_{\ib,\jb}^{\io}(\ms)+V_{\ib,\jb}^{\phi;\io}(\ms))\bigg]\nonumber \\ &=\E\Bigg[c_{\io}(\sigma_{\io}-t_{\io})\sum_{\cD\subseteq \fL,\ \cE\subseteq \fM\setminus \io,\cU\subseteq\{\io\},\cV\subseteq\{\io\}}\left(\prod_{i_r\in\cD} h_{i_r}^{\io}(\ms)\right)\left(\prod_{i_r\in\bar{\cD}} h_{i_r}^{\phi;\io}(\ms)\right)\left(\prod_{j_r\in\cE} h_{j_r}^{\io}(\ms)\right)\nonumber \\ &\qquad\qquad \left(\prod_{j_r\in\bar{\cE}} h_{j_r}^{\phi;\io}(\ms)\right)\left(\prod_{\cU\ni\io}U_{\ib,\jb}^{\io}(\ms)\right)\left(\prod_{\cU\not\ni\io}U_{\ib,\jb}^{\phi;\io}(\ms)\right)\left(\prod_{\cV\ni\io}V_{\ib,\jb}^{\io}(\ms)\right)\left(\prod_{\cV\not\ni\io}V_{\ib,\jb}^{\phi;\io}(\ms)\right)\Bigg]\nonumber \\ &=\E\Bigg[c_{\io}(\sigma_{\io}-t_{\io})\sum_{\cD\subseteq \fL,\ \cE\subseteq \fM\setminus \io, \cU\subseteq \{\io\},\cV\subseteq\{\io\}}\left(\prod_{r=1}^{\tp} (h_{i_r}^{\io}(\ms)\mathbbm{1}(i_r\in\cD)+h_{i_r}^{\phi;\io}(\ms)\mathbbm{1}(i_r\in\bar{\cD}))\right)\nonumber \\ & \qquad \qquad \left(\prod_{r=1}^{\tq} (h_{j_r}^{\io}(\ms)\mathbbm{1}(j_r\in\cE)+h_{j_r}^{\phi;\io}(\ms)\mathbbm{1}(j_r\in\bar{\cE}))\right)\left(U_{\ib,\jb}^{\iota}(\ms)\ind(\iota\in\cU)+U_{\ib,\jb}^{\phi;\iota}(\ms)\ind(\iota\notin \cU)\right)\nonumber \\ &\qquad \qquad \left(V_{\ib,\jb}^{\io}(\ms)\ind(\io\in\cV)+V_{\ib,\jb}^{\phi;\io}(\ms)\ind(\io\notin\cV)\right)\Bigg].
		\end{align}
		Next note that in the above summation, the term corresponding to $(\cD,\cE,\cU,\cV)=(\fL,\fM\setminus \io,\{\io\},\{\io\})$ can be dropped. This is because, $h_{i_r}^{\io}(\ms)$, $h_{j_r}^{\io}(\ms)$, $U_{\ib,\jb}^{\io}(\ms)$, and $V_{\ib,\jb}^{\io}(\ms)$ are measurable with respect to the sigma field generated by $(\sigma_1,\ldots, \sigma_{\io-1},\sigma_{\io+1},\ldots ,\sigma_N)$ and consequently, by the tower property, we have:
		\begin{equation*}
			\E\left[c_{{\io}}(\sigma_{{\io}}-t_{{\io}})\left(\prod_{r\in [\tp]} h_{i_r}^{\io}(\ms)\right)\left(\prod_{r\in [\tq]\setminus \io} h_{j_r}^{\io}(\ms)\right)U_{\ib,\jb}^{\io}(\ms)V_{\ib,\jb}^{\io}(\ms)\right]=0.
		\end{equation*}
		The conclusion in~\eqref{eq:baseprop2} then follows by combining the above observation with~\eqref{eq:baseprop1}. The conclusion in~\eqref{eq:basepropnew1} follows by using $h_{j_r}^{\io}(\ms)=h_{j_r}(\ms)-h_{j_r}^{\phi;\io}(\ms)$ for $j_r\in \cE$  and repeating a similar computation as above.
	\end{proof}
	Observe that in~\cref{prop:baseprop} (see~\eqref{eq:basepropnew1}), for every fixed $(\cD,\cE,\tcE,\cU,\cV)$, the left and right hand sides have the same form with the functions $h_{i_r}(\cdot)$, $h_{j_r}(\cdot)$, $U_{\ib,\jb}(\cdot)$, $V_{\ib,\jb}(\cdot)$ on the LHS being replaced with $h_{i_r}(;\cD,\cD)$, $h_{j_r}(;\cE,\tcE)$, $U_{\ib,\jb}(\cdot;\cU,\cU)$, and $V_{\ib,\jb}(\cdot;\cV,\cV)$ on the RHS. This suggests a recursive approach for further splitting $h_{n_r}(;\cD,\cD)$ and $h_{m_r}(;\cE,\tcE)$.
	
	Let us briefly see how \cref{prop:baseprop} ties into our goal of studying the limit of $\E T_N^{k}U_N^{k_1}V_N^{k_2}$ (where $T_N$ is defined in \eqref{eq:pivotstat} and $U_N$, $V_N$ are defined in \eqref{eq:randvar}). Recall the definition of $\mathcal{C}_k$ from \eqref{eq:unag}.  Through some elementary computations (see \cref{lem:initremove}), one can show that 
	\begin{align}\label{eq:mainrep}
		&\;\;\;\;\E T_N^{k} U_N^{k_1} V_N^{k_2} \nonumber \\ &=\E\left[\left(\frac{1}{N^{k/2}}\sum_{\substack{(\ell_1,\ldots ,\ell_p,\\ q)\in \mathcal{C}_k}}\sum_{\substack{(i_1,\ldots ,i_p,\\ j_1,\ldots ,j_q)\in \Theta_{N,p+q}}}\prod_{r=1}^p (c_{i_r}(\si_{i_r}-t_{i_r}))^{\ell_r} \prod_{r=1}^q (c_{j_r}(\si_{j_r}-t_{j_r}))\right)U_N^{k_1} V_N^{k_2}\right] \nonumber \\ &=\E\left[\frac{1}{N^{k/2}}\sum_{\substack{(\ell_1,\ldots ,\ell_p,\\ q)\in \mathcal{C}_k}}\sum_{\substack{(i_1,\ldots ,i_p,\\ j_1,\ldots ,j_q)\in \Theta_{N,p+q}}}\prod_{r=1}^p (c_{i_r}(\si_{i_r}-t_{i_r}))^{\ell_r} \prod_{r=1}^q (c_{j_r}(\si_{j_r}-t_{j_r})) U_{N,\ib,\jb}^{k_1} V_{N,\ib,\jb}^{k_2}\right] + o(1), 
	\end{align}
	where $U_{N,\ib,\jb}$ and $V_{N,\ib,\jb}$ are defined in \eqref{eq:reducedex}. 
	We then apply \cref{prop:baseprop}-\eqref{eq:baseprop2} for every fixed $(i_1,\ldots ,i_p,j_1, \ldots ,j_q)\in\Theta_{N,p+q}$ in \eqref{eq:mainrep} with 
	$$h_{i_r}(\ms)=(c_{i_r}(g(\si_{i_r})-t_{i_r}))^{\ell_r}, \, \, \, h_{j_r}=c_{j_r}(g(\si_{j_r})-t_{j_r}), \, \, \, U_{\ib,\jb}=U_{N,\ib,\jb}, \, \, \, V_{\ib,\jb}=V_{N,\ib,\jb}, \,\,\, \io=j_q.$$
	This implies that the following term in \eqref{eq:mainrep}, which we call the \emph{root}, can be split into nodes indexed by sets of the form $(\cD,\cE,\cU,\cV)$ in $\mathfrak{G}$ (see \eqref{eq:basepropa}). This will form the first level of our tree. Now we take each of the nodes in \emph{level one}, and further split them according to \eqref{eq:basepropnew1}, to get \emph{level two} of the tree. Now note that every node in level two is characterized by sets $(\cD,\cE,\tcE,\cU,\cV)$ where $(\cD,\cE,\cU,\cV)\in\mathfrak{G}$, $\tcE\subseteq \cE$. Also by construction, either $\tcE$ is empty or $\tcE\subseteq \{j_1,\ldots ,j_{q-1}\}$. Also for each $j_r\in\tcE$, $h_{j_r}(\ms;\cE,\tcE)=c_{j_r}(g(\si_{j_r})-t_{j_r})$. If $\tcE$ is empty, we don't split that node further. If not, then we split that node, again by using \cref{prop:baseprop}-\eqref{eq:basepropa} and \eqref{eq:basepropnew1}, and choosing a new $\io\in \tcE$. This will lead to \emph{levels three and four}. We continue this process at every even level of the tree. Our choice of $\io$ is always distinct at every even level and always belongs to $\{j_1,\ldots ,j_q\}$. Therefore, by construction, our tree terminates after at most $2q$ levels. The core of our argument is to characterize all the (finitely many) nodes of the tree that have non-vanishing contribution when summed up over $(i_1,\ldots ,i_p,j_1,\ldots ,j_q)\in\Theta_{N,p+q}$ (after appropriate scaling). 
	
	We now refer the reader to Algorithm~\ref{alg:construct_tree}-\ref{alg:construct_treep2}, where we present a formal description of the above recursive approach to construct the required decision tree.
	
	Observe that~\eqref{eq:baseprop2}~and~\eqref{eq:basepropnew1} have a very similar form. The major difference is that $h_{m_r}(\ms;\cE)$ (see~\eqref{eq:baseprop2}~and~\eqref{eq:basepropc}) equals $h_{m_r}^{\io}(\ms)$ for $m_r\in\cE$, whereas $h_{m_r}(\ms;\cE,\tcE)$ (see~\eqref{eq:basepropnew1}~and~\eqref{eq:basepropnew3}) equals $h_{m_r}(\ms)$ for $m_r\in\tcE$. Also note that $\E h_{m_r}^{\io}(\ms)$ may not equal $0$ whereas $\E h_{m_r}(\ms)=0$. This observation is crucial for the construction of the tree. It ensures that we can drop the $(\cD,\cE,\cU,\cV)=((i_1\ldots ,i_p),(j_1,\ldots ,j_q)\setminus \io, \{\io\}, \{\io\})$ term in $\mathfrak{G}$ (see \eqref{eq:basepropa}). We therefore differentiate between these two cases by referring to them as \textbf{centering} and \textbf{re-centering} steps respectively; see steps 7,8, 21, and 22, in Algorithm~\ref{alg:construct_tree}-\ref{alg:construct_treep2}.
	
	\hbadness=10000
	
	\begin{algorithm}
		\caption{Decision tree --- first and second generations}
		\label{alg:construct_tree}
		
		\begin{algorithmic}[1]
			\State {DECISION TREE}{$(l_1,\ldots ,l_p,q)\in \mathcal{C}_{p,q,k},\, (i_1,\ldots ,i_p,j_1,\ldots ,j_q)\in \Theta_{N,p+q},\,$} {(see \eqref{eq:cpqk} and \eqref{eq:thnr} for relevant definitions).} {Recall the definitions of $U_{N,\ib,\jb}\equiv U_{N,\ib,\jb}(\ms)$ and $V_{N,\ib,\jb}\equiv V_{N,\ib,\jb}(\ms)$ from \eqref{eq:reducedex}.}
			\State Label the \textbf{root node} as $R_0$ and assign \begin{equation}\label{eq:rootnode}R_0\equiv R_0(i_1,\ldots ,i_p,j_1,\ldots ,j_q)\gets \E\left[\left(\prod_{r=1}^p (c_{i_r}(g(\sigma_{i_r})-t_{i_r}))^{l_r}\right)\left(\prod_{r=1}^q c_{j_r}(g(\sigma_{j_r})-t_{j_r})\right)U_{N,\ib,\jb}^{k_1}V_{N,\ib,\jb}^{k_2}\right].\end{equation}
			\State Also assign
			\begin{equation*}
				\cD_0\gets (i_1,\ldots ,i_p),\quad \cE_0\gets (j_1,\ldots ,j_q),\quad \mbox{and}\quad M_0\gets\{j_b\in\cE_0:j_{b'}\notin \cE_0\ \mbox{for}\ b'>b\}=j_q, \; M_0\gets -\infty \ \mbox{if}\ \cE_0=\phi,
			\end{equation*}
			\begin{equation*}
				\cU_0=\cV_0=\phi.
			\end{equation*}
			\If{$q=0$}
			\State \textbf{terminate}.
			\Else 
			\State \textbf{First generation (Centering step)}: Set 
			$\tp\gets p$, $\tq\gets q$, and $\io\gets M_0$ and construct $\FG_1$ as in~\eqref{eq:basepropa}. Enumerate $\FG_1$ as 
			\begin{equation}\label{eq:gen2pt1a}
				\FG_1\gets\{G_{1,1},G_{1,2},\ldots ,G_{1,|\FG_1|}\},
			\end{equation} 
			where each $G_{1,z_1}$ is of the form $(\cD_{1,z_1},\cE_{1,z_1},\cU_{1,z_1},\cV_{1,z_1})$ as in~\eqref{eq:basepropa}. Then apply~\cref{prop:baseprop} with functions $h_{i_r}(\ms)=(c_{i_r}(g(\sigma_{i_r})-t_{i_r}))^{l_r}$ for $r\in [\tp]$, $h_{j_r}(\ms)=c_{j_r}(g(\sigma_{j_r})-t_{j_r})$  for $r\in [\tq]$, $U_{\ib,\jb}(\ms)=U_{N,\ib,\jb}^{k_1}$, and $V_{\ib,\jb}(\ms)=V_{N,\ib,\jb}^{k_2}$, to get the nodes of the \textbf{first generation} (which we label as $R_{1,z_1}\equiv R_{z_1}(i_1,\ldots ,i_p,j_1,\ldots ,j_q)$):
			\begin{equation}\label{eq:gen2pt1c}
				\begin{aligned}
					R_0&=\sum_{k_1:\ (\cD_{1,z_1},\cE_{1,z_1}, \cU_{1,z_1},\cV_{1,z_1})\in \FG_1} R_{z_1},\\  R_{z_1}\gets \E\bigg[\left(\prod\limits_{r=1}^{\tp} h_{i_r}(\ms;\cD_{1,z_1})\right)&\left(\prod\limits_{r=1}^{\tq} h_{j_r}(\ms;\cE_{1,z_1})\right)U_{\ib,\jb}(\ms;\cU_{1,z_1}) V_{\ib,\jb}(\ms;\cV_{1,z_1})\bigg].
				\end{aligned}
			\end{equation}
			Here $h_{i_r}(\ms;\cD_{1,z_1})$ for $r\in [\tp]$, $h_{j_r}(\ms;\cE_{1,z_1})$ for $j_r\in (j_1,\ldots ,j_q)\setminus \io$, $h_{\io}(\ms;\cE_{1,z_1})$, $U_{\ib,\jb}(\ms;\cU_{1,z_1})$, and $V_{\ib,\jb}(\ms;\cV_{1,z_1})$ are defined as in~\eqref{eq:basepropb},~\eqref{eq:basepropc}~\eqref{eq:basepropd},~\eqref{eq:baseprope},~and~\eqref{eq:basepropf} respectively. In addition, we also assign 
			$$M_{1,z_1}\gets M_0, \quad \bar{\cD}_{1,z_1}\gets (i_1,\ldots ,i_{\tp})\setminus \cD_{1,z_1}, \quad \bar{\cE}_{1,z_1}\gets ((j_1,\ldots ,j_{\tq})\setminus \{\io\})\setminus \cE_{1,z_1}.$$
			
			\State \textbf{Second generation (Re-centering step)}: With  
			$\tp$, $\tq$, and $\io$ as in the \textbf{first generation}, by using~\cref{prop:baseprop}-\eqref{eq:basepropnew1}, we get:
			\begin{equation}\label{eq:gen2pt2a}
				R_0=\sum_{\substack{(\cD_{1,z_1},\cE_{1,z_1}, \cU_{1,z_1}, \cV_{1,z_1})\in \FG_1,\ \cE_{2,z_2}\subseteq \cE_{1,z_1},\\  \cD_{2,z_2}=\cD_{1,z_1}, \cU_{2,z_2}=\cU_{1,z_1}, \cV_{2,z_2}=\cV_{1,z_1}}} R_{z_1,z_2}(i_1,\ldots ,i_p,j_1,\ldots ,j_q).
			\end{equation}
			where
			\begin{equation}\label{eq:gen2pt2aa}
				\begin{aligned}
					R_{z_1,z_2}(i_1,\ldots ,i_p,j_1,\ldots ,j_q)&\gets \E\Bigg[\left(\prod\limits_{r=1}^{\tp} h_{i_r}(\ms;\cD_{1,z_1},\cD_{2,z_2})\right)\left(\prod\limits_{r=1}^{\tq} h_{j_r}(\ms;\cE_{1,z_1},\cE_{2,z_2})\right)\\ &\;\;\;\;U_{\ib,\jb}(\ms;\cU_{1,z_1},\cU_{2,z_2})V_{\ib,\jb}(\ms;\cV_{1,z_1},\cV_{2,z_2}\Bigg]
				\end{aligned}
			\end{equation}
			For the definitions of all relevant terms in~\eqref{eq:gen2pt2a} see~\eqref{eq:basepropnew2},~\eqref{eq:basepropnew3},~and~\eqref{eq:basepropnew4}. Further we assign
			\begin{equation}\label{eq:gen2pt2b}
				M_{2,z_2}\gets\{j_b\in\cE_{2,z_2}:j_{b'}\notin \cE_{2,z_2}\ \mbox{for}\ b'>b\},\; M_{2,z_2}\gets -\infty\ \mbox{if}\ \cE_{2,z_2}=\phi, \quad \mbox{and}\quad \bar{\cE}_{2,z_2}\gets \cE_{1,z_1}\setminus \cE_{2,z_2}.
			\end{equation}
			
			\EndIf
			
			\algstore{myalg1}
		\end{algorithmic}
	\end{algorithm}
	\begin{algorithm}
		\caption{Iterative construction of $2T+1$ and $2T+2$-th generation of the decision tree}
		\label{alg:construct_treep2}
		\begin{algorithmic}[1]
			\algrestore{myalg1}
			\State Assign $\boldsymbol{\mathrm{flag}\gets \mathrm{TRUE}}$; $T\gets 1$.
			\While{$\mathrm{flag}=\mathrm{TRUE}$}
			
			\State Set $\mathrm{flag}=\mathrm{FALSE}$.
			\Repeat
			\State over all $(z_1,\ldots ,z_{2T})$ such that $R_{z_1,\ldots ,z_{2T}}\equiv R_{z_1,\ldots ,z_{2T}}(i_1,\ldots ,i_p,j_1,\ldots ,j_q)$ is a node of the ${2T}$-th \textbf{generation}.
			\State Associated with every node of the ${2T}$-th \textbf{generation}, there is a sequence of nodes $R_0\rightarrow R_{z_1}\rightarrow R_{z_1,z_2} \rightarrow \ldots \rightarrow R_{z_1,\ldots ,z_{2T-1}}\rightarrow R_{z_1,\ldots ,z_{2T}}$ where each is a child of its predecessor, sequences of sets $(\cD_0,\cD_{1,z_1},\ldots ,\cD_{2T,z_{2T}})$, $(\cE_0,\cE_{1,z_1},\ldots ,\cE_{2T,z_{2T}})$, $(\cU_{1,z_1},\cU_{2,z_2},\ldots ,\cU_{2T,z_{2T}})$, $(\cV_{1,z_1},\cV_{2,z_2},\ldots ,\cV_{2T,z_{2T}})$, a sequence of integers $(M_0,M_{1,z_1},\ldots ,M_{2T,z_{2T}})$ and functions $\{h_{i_r}(\ms;\cD_{1,z_1},\ldots ,\cD_{2T,z_{2T}})\}_{r\in [p]}$, $\{h_{j_r}(\ms;\cE_{1,z_1},\ldots ,\cE_{2T,z_{2T}})\}_{r\in [q]}$, $U_{\ib,\jb}(\ms;\cU_{1,z_1},\ldots ,\cU_{2T,z_{2T}})$, $V_{\ib,\jb}(\ms;\cV_{1,z_1},\ldots ,\cV_{2T,z_{2T}})$. For $T=1$, these notations were already introduced while describing the \textbf{first generation} (see~\eqref{eq:gen2pt1a},~\eqref{eq:gen2pt1c},~\eqref{eq:gen2pt2a},~\eqref{eq:gen2pt2aa},~and~\eqref{eq:gen2pt2b}).
			\If{$M_{2T,z_{2T}}=-\infty$ or equivalently $\cE_{2T,z_{2T}}=\phi$}
			\State \textbf{terminate}.
			\Else
			\State Set $\mathrm{flag}=\mathrm{TRUE}$.
			
			\State \textbf{$(2T+1)$-th generation (Centering step):} With $\tq=q$, $\tp=p$, we define $\io=M_{2T,z_{2T}}$. Apply~\cref{prop:baseprop}-\eqref{eq:baseprop2} with the functions $(\{h_{i_r}(\ms;\cD_{1,z_1},\ldots ,\cD_{2T,z_{2T}})\}_{r\in [\tp]})$, $(\{h_{j_r}(\ms;\cE_{1,z_1},\ldots ,\cE_{2T,z_{2T}})\}_{r\in \tq})$, $U_{\ib,\jb}(\ms;\cU_{1,z_1},\ldots ,\cU_{2T,z_{2T}})$, and $V_{\ib,\jb}(\ms;\cV_{1,z_1},\ldots ,\cV_{2T,z_{2T}})$. This yields a collection $\FG_{2T+1}$ (depending on $(z_1,\ldots ,z_{2T})$) of sets $\{G_{2T+1,z_{2T+1}}\}$ each of the form $(\cD_{2T+1,z_{2T+1}},\cE_{T,z_{2T+1}},\cU_{2T+1,z_{2T+1}},\cV_{2T+1,z_{2T+1}})$ (see~\eqref{eq:basepropa}), such that, with $R_{z_1,\ldots ,z_{2T+1}}\equiv R_{z_1,\ldots ,z_{2T+1}}(i_1,\ldots ,i_p,j_1,\ldots ,j_q)$,
			\begin{align*}
				&\qquad \qquad \qquad \qquad R_{z_1,\ldots ,z_{2T}}=\sum_{(\cD_{2T+1,z_{2T+1}},\cE_{2T+1,z_{2T+1}},\cU_{2T+1,z_{2T+1}},\cV_{2T+1,z_{2T+1}})\in \FG_{2T+1}} R_{z_1,\ldots ,z_{2T+1}},\\ &\quad R_{z_1,\ldots ,z_{2T+1}}\gets \E\Bigg[\left(\prod\limits_{r=1}^{\tp} h_{i_r}(\ms;\cD_{1,z_1},\ldots ,\cD_{2T,z_{2T}},\cD_{2T+1,z_{2T+1}})\right)\left(\prod\limits_{r=1}^{\tq} h_{j_r}(\ms;\cE_{1,z_1},\ldots ,\cE_{2T,z_{2T}},\cE_{2T+1,z_{2T+1}})\right)\\ &\qquad\qquad\qquad\qquad\qquad U_{\ib,\jb}(\ms;\cU_{1,z_1},\ldots ,\cU_{2T+1,z_{2T+1}}) V_{\ib,\jb}(\ms;\cV_{1,z_1},\ldots ,\cV_{2T+1,z_{2T+1}})\Bigg].
			\end{align*}
			We also set $M_{2T+1,z_{2T+1}}\gets M_{2T,z_{2T}}$, $\bar{\cD}_{2T+1,z_{2T+1}}=(i_1,\ldots, i_{\tp})\setminus \cD_{2T+1,z_{2T+1}}$, and $\bar{\cE}_{2T+1,z_{2T+1}}=((j_1,\ldots ,j_{\tq})\setminus \{M_{2T,z_{2T}}\})\setminus \cE_{2T+1,z_{2T+1}}$.
			\State \textbf{$(2T+2)$-th generation (Re-centering step):} With $\tq=q$, $\tp=p$, $\io$ defined as in the previous generation, by using~\cref{prop:baseprop}-\eqref{eq:basepropnew1}, we get with $R_{z_1,z_2,\ldots ,z_{2T+2}}\equiv R_{z_1,z_2,\ldots ,z_{2T+2}}(i_1,\ldots ,i_p,j_1,\ldots ,j_q)$, 
			\begin{equation}\label{eq:gen2ptreca}
				R_{z_1,\ldots ,z_{2T}}=\sum_{\substack{(\cD_{2T+1,z_{2T+1}},\cE_{2T+1,z_{2T+1}},\cU_{2T+1,z_{2T+1}},\cV_{2T+1,z_{2T+1}})\in \FG_{2T+1},\\ \cE_{2T+2,z_{2T+2}}\subseteq {\color{red} \boldsymbol{\cE_{2T,z_{2T}}\cap \cE_{2T+1,z_{2T+1}}}}, \cD_{2T+2,z_{2T+2}}=\cD_{2T+1,z_{2T+1}}, \\ \cU_{2T+2,z_{2T+2}}=\cU_{2T+1,z_{2T+1}}, \cV_{2T+2,z_{2T+2}}=\cV_{2T+1,z_{2T+1}}}} R_{z_1,\ldots ,z_{2T+2}}
			\end{equation}
			\begin{equation}\label{eq:gen2ptrecaa}
				\begin{aligned}
					R_{z_1,\ldots ,z_{2T+2}}&\gets \E\Bigg[\left(\prod\limits_{r=1}^{\tp} h_{i_r}(\ms;\cD_{1,z_1},\ldots ,\cD_{2T+2,z_{2T+2}})\right)\left(\prod\limits_{r=1}^{\tq} h_{j_r}(\ms;\cE_{1,z_1},\ldots,\cE_{2T+1,z_{2T+2}})\right)\\  &\qquad\qquad\qquad\qquad U_{\ib,\jb}(\ms;\cU_{1,z_1},\ldots ,\cU_{2T+2,z_{2T+2}}) V_{\ib,\jb}(\ms;\cV_{1,z_1},\ldots ,\cV_{2T+2,z_{2T+2}})\Bigg]
				\end{aligned}
			\end{equation}
			For the definitions of all relevant terms in~\eqref{eq:gen2pt2a} see~\eqref{eq:basepropnew2}~and~\eqref{eq:basepropnew3}. Further we assign
			\begin{equation*}
				M_{2T+2,z_{2T+2}}\gets\{j_b\in\cE_{2T+2,z_{2T+2}}:j_{b'}\notin \cE_{2T+2,z_{2T+2}}\ \mbox{for}\ b'>b\},\; M_{2T+2,z_{2T+2}}\gets -\infty\ \mbox{if}\ \cE_{2T+2,z_{2T+2}}=\phi.
			\end{equation*}
			\begin{equation*}
				\bar{\cE}_{2T+2,z_{2T+2}}\gets (\cE_{2T,z_{2T}}\cup \cE_{2T+1,z_{2T+1}})\setminus \cE_{2T+2,z_{2T+2}}.
			\end{equation*}
			
			\EndIf
			\Until no more nodes remain in the $2T$-th \textbf{generation}.
			\State $T\gets T+1$.
			
			\EndWhile
			%\EndProcedure
		\end{algorithmic}
	\end{algorithm}
	
	\subsection{An example of a decision tree}\label{sec:exmpledtree}
	
	In this section, we provide an example of a decision tree (see Algorithm~\ref{alg:construct_tree}-\ref{alg:construct_treep2} for details) for better understanding of our techniques, and in the process, we define some relevant terms which will be useful throughout the paper. As the intent here is to build intuition for the proof, we will assume that $U_{\ib,\jb}$ and $V_{\ib,\jb}$ are both constant functions. 
	
	\begin{definition}[Leaf node]\label{def:leafnode}
		A node in the decision tree is called a \emph{leaf node} if it does not have any child nodes. Based on~\cref{alg:construct_tree}, a node is equivalently a leaf node if it satisfies the termination condition, as given in step 17 of~\cref{alg:construct_tree}.
	\end{definition} 
	\begin{observation}[Invariance of sum]\label{obs:insum}
		Note that at every step, whenever a node is split into child nodes, by virtue of~\cref{prop:baseprop}, the sum of the child nodes equals the parent node. Consequently, we have:
		$$R_0(i_1,\ldots ,i_p,j_1,\ldots ,j_q)=\sum_{R_{z_1,\ldots ,z_t}\ \mbox{is a leaf node}} R_{z_1,\ldots ,z_t}(i_1,\ldots ,i_p,j_1,\ldots ,j_q).$$
	\end{observation}
	\begin{definition}[Path]\label{def:path}
		A \emph{path} is a sequence of nodes in the tree such that each node in the sequence is a child of its predecessor. For example, $R_0\rightarrow R_{z_1}\rightarrow R_{z_1,z_2} \rightarrow \ldots \rightarrow R_{z_1,\ldots ,z_t}$ is a path if $R_{z_1}$ is a child of $R_0$, $R_{z_1,z_2}$ is a child of $R_{z_1}$ and so on.
	\end{definition}
	\begin{definition}[Branch and length]\label{def:length}
		A branch is a path which begins with the root (see~\eqref{eq:rootnode}) and ends with a leaf node (see~\cref{def:leafnode}). The length of a branch is one less than the number of nodes in that branch (to account for the root node). The tree has \emph{length} $\cT\in \mathbb{N}\cup\{\infty\}$ if no node of the $\cT^{\mathrm{th}}$ generation has any child nodes, i.e., all nodes of the $\cT^{\mathrm{th}}$ generation satisfy the termination condition (see step 17 of~\cref{alg:construct_tree}).
	\end{definition}
	
	In~\cref{fig:egtree}, we present an example of a decision tree when the root (see~\eqref{eq:rootnode}) is $$R_0\equiv (c_{i_1}(\sigma_{i_1}-t_{i_1}))^2(c_{j_1}(\sigma_{j_1}-t_{j_1}))(c_{j_2}(\sigma_{j_2}-t_{j_2}))$$ with $p=1,q=2,$ and $j_1<j_2$. It will also provide some insight into the proof of~\cref{lem:surviveterm} (which is the subject of the next section, i.e.,~\cref{sec:pfsurvive}). 
	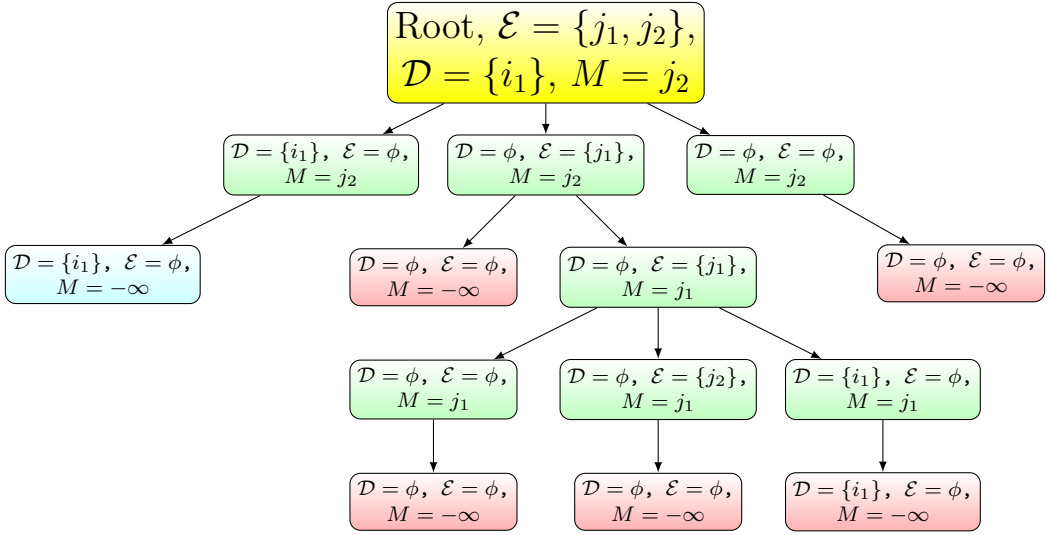
\begin{figure}[h]
		\centering
		\begin{tikzpicture}
			[
			grow                    = down,
			sibling distance        = 10em,
			level distance          = 5em,
			edge from parent/.style = {draw, -latex},
			every node/.style       = {font=\footnotesize},
			sloped
			]
			\node [root] {Root, $\cE=\{j_1,j_2\}$,\\ $\cD=\{i_1\}$, $M=j_2$}
			child { node(gen11) [con] {$\cD=\{i_1\}$, $\cE=\phi$,\\ $M=j_2$}
				child { node [env, below left = 2.2em and 1em of gen11] {$\cD=\{i_1\}$, $\cE=\phi$,\\ $M=-\infty$}
				}
			}
			child { node [con] {$\cD=\phi$, $\cE=\{j_1\}$,\\ $M=j_2$}
				child { node [nocon] {$\cD=\phi$, $\cE=\phi$,\\ $M=-\infty$}
				}
				child {node [con] {$\cD=\phi$, $\cE=\{j_1\}$,\\ $M=j_1$}
					child {node [con] {$\cD=\phi$, $\cE=\phi$,\\ $M=j_1$}
						child {node[nocon] {$\cD=\phi$, $\cE=\phi$,\\ $M=-\infty$}
						}
					}
					child {node [con] {$\cD=\phi$, $\cE=\{j_2\}$,\\ $M=j_1$}
						child {node [nocon] {$\cD=\phi$, $\cE=\phi$,\\ $M=-\infty$}
						}
					}
					child {node [con] {$\cD=\{i_1\}$, $\cE=\phi$,\\ $M=j_1$}
						child{node [nocon] {$\cD=\{i_1\}$, $\cE=\phi$,\\ $M=-\infty$}
						}
					}			
				}
			}
			child { node(gen13) [con] {$\cD=\phi$, $\cE=\phi$,\\ $M=j_2$}
				child { node [nocon, below right = 2.2em and 1em of gen13] {$\cD=\phi$, $\cE=\phi$,\\ $M=-\infty$}
				}
			};
		\end{tikzpicture}
		\caption{In the above diagram, we plot the complete decision tree according to~\cref{alg:construct_tree} when $p=1$, $q=2$. The root node is in yellow, the leaf nodes are in red (except in one case where it is in cyan, the reasons for which are explained in the main text) and the non leaf nodes are in green. The values of $\cD,\cE$, and $M$ are specified along with each node (we drop the subscripts used in Algorithm~\ref{alg:construct_tree}-\ref{alg:construct_treep2} to avoid notational clutter). At the root, $M=j_2$, $\cD=\{i_1\}$, $\cE=\{j_1,j_2\}$ (see step 3 of~\cref{alg:construct_tree}). Therefore, in the first generation, $\cD\subseteq \{i_1\}$ and $\cE\subseteq \{j_1\}$. By~\cref{prop:baseprop}, the case $(\cD,\cE)=(\{i_1\},\{j_1\})$ does not contribute. This leads to $3$ choices for $(\cD,\cE)$ which form the $3$ nodes of the \textbf{first generation}. In $2$ of these nodes $\cE=\phi$, and consequently their child nodes will have $\cE=\phi$ and $M=-\infty$ (see~\eqref{eq:gen2pt2b}) which satisfy the \textbf{termination condition} from step 17 in~\cref{alg:construct_treep2}. For the other node in the first generation, the only remaining option is $\cD=\phi$, $\cE=\{j_1\}$. For its child nodes, by step 8 of~\cref{alg:construct_tree} (see~\eqref{eq:gen2pt2a}), the only options of $\cE$ are $\phi$ and $\{j_1\}$. The case $\cE=\phi$ once again satisfies the termination condition from step 17 of~\cref{alg:construct_treep2} and is thus a leaf node. Therefore the only node in the \textbf{second generation} which has child nodes is the case where $\cD=\phi$, $\cE=\{j_1\}$. This in turn implies $M=\{j_1\}$ (see~\eqref{eq:gen2pt2b}). The \textbf{third and fourth generations} are formed similarly using the recursive approach described in~\cref{alg:construct_treep2}.}
		\label{fig:egtree}
	\end{figure}
	Note that by~\eqref{eq:root}, $f_L(\cdot)$ can be written as:
	\begin{equation}\label{eq:egexplain1}
		\sum_{(i_1,j_1,j_2)\in \Theta_{N,3}} (c_{i_1}(\sigma_{i_1}-t_{i_1}))^2 (c_{j_1}(\sigma_{j_1}-t_{j_1}))(c_{j_2}(\sigma_{j_2}-t_{j_2}))
	\end{equation}
	when $p=1$, $q=2$. By~\cref{fig:egtree},~\eqref{eq:egexplain1} can be split into the sum of $6$ terms corresponding to each leaf node (see~\cref{obs:insum}). Let us focus on the first leaf node (from the right) in the second generation, where $(\cD,\cE)=(\phi,\phi)$. By~\cref{prop:baseprop}, we have:
	\begin{align*}
		&|h_{i_1}^{\phi;j_2}(\ms)|=|(c_{i_1}(g(\sigma_{i_1})-t_{i_1}))^2-(c_{i_1}(g(\sigma_{i_1})-t_{i_1}^{j_2}))^2|\lesssim |t_{i_1}-t_{i_1}^{j_2}|\lesssim \Q_{N,2}(i_1,j_2)\\ & |h_{j_1}(\ms;G)|=|c_{j_1}(g(\sigma_{j_1})-t_{j_1})-c_{j_1}(g(\sigma_{j_1})-t_{j_1}^{j_2})|\lesssim |t_{j_1}-t_{j_1}^{j_2}|\lesssim \Q_{N,2}(j_1,j_2).
	\end{align*}
	Therefore, the contribution of this leaf node can be bounded by
	$$\big|\sum_{(i_1,j_1,j_2)\in\Theta_{N,3}} h_{i_1}^{\phi;j_2}(\ms)h_{j_1}^{\phi;j_2}(\ms)c_{j_2}(g(\si_{j_2})-t_{j_2})\big|\lesssim \sum_{(i_1,j_1,j_2)\in\Theta_{N,3}} \Q_{N,2}(i_1,j_2)\Q_{N,2}(j_1,j_2)\lesssim N$$
	where the last line uses~\cref{as:cmean}. In this case, $k=4$ and so~\eqref{eq:conmain1} implies that the contribution of this leaf node to~\eqref{eq:egexplain1} is negligible asymptotically. A similar argument shows that the contribution of the middle leaf node in the second generation can also be bounded by
	$$\sum_{(i_1,j_1,j_2)\in\Theta_{N,3}} \Q_{N,2}(i_1,j_2)\Q_{N,2}(j_1,j_2)\lesssim N$$ 
	which shows that its contribution too, is negligible by~\eqref{eq:conmain1}. In a similar vein, the contribution of the leftmost leaf node in the fourth generation can be bounded by:
	\begin{equation*}
		\sum_{(i_1,j_1,j_2)\in\Theta_{N,3}} \tQ_{N,3}(i_1,j_1,j_2)\Q_{N,2}(j_1,j_2)\lesssim \sum_{(j_1,j_2)\in [N]^2} \Q_{N,2}(j_1,j_2)\lesssim N
	\end{equation*}
	where $\tQ_{N,3}$ is defined as in~\eqref{eq:newtensor}; also see~\cref{lem:supbound}, part (d). Further, the middle and the rightmost leaf nodes in the fourth generation have contributions bounded by:
	$$\sum_{(i_1,j_1,j_2)\in\Theta_{N,3}} \tQ_{N,3}(i_1,j_1,j_2)\lesssim N,$$
	and
	$$\sum_{(i_1,j_1,j_2)\in \Theta_{N,3}} \Q_{N,2}(i_1,j_2)\Q_{N,2}(j_1,j_2)\lesssim N.$$ 
	This shows that all leaf nodes other than the leftmost leaf node in the second generation of~\cref{fig:egtree} (which is highlighted in {\color{cyan}{\emph{cyan}}}), have asymptotically negligible contribution. Our argument for proving~\cref{lem:surviveterm} is an extension of the above observations for general $p$ and $q$. In the sequel, we will characterize all the leaf nodes which have asymptotically negligible contribution. 
	
	\section{Properties of the decision tree}\label{sec:propdtree}
	In this section, we list some crucial properties of the decision tree that will be important in the sequel for proving \cref{lem:surviveterm}. We first show that the tree constructed in Algorithm~\ref{alg:construct_tree}-\ref{alg:construct_treep2} cannot grow indefinitely.
	\begin{prop}\label{prop:boundlen}
		\item $\cT\leq 2q$, i.e., the length of the tree (see~\cref{def:length})  is finite and bounded by $2q$.
	\end{prop}
	
	\begin{proof}
		Observe that the cardinality of the sets $\{\cE_{2t,z_{2t}}\}_{t\geq 0}$ form a strictly decreasing sequence. As $|\cE_0|=q$ and $\cE_{2t,z_{2t}}\subseteq \{j_1,\ldots ,j_q\}$, we must have $\cE_{2t,z_{2t}}=\phi$ for some $t\leq q$. Therefore, by step 17 of \cref{alg:construct_treep2}, all branches of the tree (see~\cref{def:length}) have length $\leq 2q$, and consequently $\cT\leq 2q$, which completes the proof.
	\end{proof}
	
	The next set of properties we are interested in, revolves around bounding the contribution of the nodes along an arbitrary branch, say $R_0\rightarrow R_{z_1} \rightarrow \ldots \rightarrow R_{z_1,\ldots ,z_t}$.  The proofs of these results  are deferred to later sections. As a preparation, we begin with the following observation:

	\begin{lemma}\label{lem:supbound}
		Consider a path $R_0\rightarrow R_{z_1}\rightarrow\ldots\rightarrow R_{z_1,\ldots ,z_{2t}}$ of the decision tree constructed in Algorithm~\ref{alg:construct_tree}-\ref{alg:construct_treep2}. Recall the construction of $(\cD_{a,z_a},\cE_{a,z_a},M_{a,z_a},\cU_{a,z_a},\cV_{a,z_a})_{a\in [2t]}$. Then, under Assumptions~\ref{as:cmean}~and~\ref{as:coeffalt}, the following holds:
		\begin{enumerate}
			\item [(a).] The following uniform bound holds:
			\begin{equation}\label{eq:supbound1}
				\max_{t_0\leq 2t}\max\left\{\max_{r=1}^p \max_{\ms\in\B^N}|h_{i_r}(\ms;\cD_{0,z_0},\ldots ,\cD_{t_0,z_{t_0}})|,\max_{r=1}^q \max_{\ms\in \B^N}|h_{j_r}(\ms;\cE_{0,z_0},\ldots ,\cE_{t_0,z_{t_0}})|\right\}\lesssim 1.
			\end{equation}
			\item [(b).] Further, fix $r\in [p]$ and set $\cI_{2t,i_r}:=\{a\in [2t]:\ a\ \mbox{is odd},\ i_r\in\bar{\cD}_{a,z_a}\}$, and $\cI^*_{2t,i_r}:=\{M_{a-1,z_{a-1}}:a\in \cI_{2t,i_r}\}$. Then the following holds:
			\begin{equation}\label{eq:supbound2}
				\max_{\ms\in \B^N}|h_{i_r}(\ms;\cD_{0,z_0},\ldots ,\cD_{2t,z_{2t}})|\lesssim \mathcal{R}[\bQ]_{N,1+|\cI^*_{2t,i_r}|}(i_r, \cI^*_{2t,i_r}). 
			\end{equation}
			\item [(c).] In a similar vein, for $r\in [q]$, set $\cJ_{2t,j_r}:=\{a\in [2t]:\ a\ \mbox{is odd},\ j_r\in\bar{\cE}_{a,z_a}\}\cup \{a\in [2t];\ a\ \mbox{is odd},\ j_r\in \cE_a,z_a)\setminus \cE_{a+1,z_{a+1}}\}$ and $\cJ^*_{2t,j_r}:=\{M_{a-1,z_{a-1}}:a\in \cJ_{2t,j_r}\}$. Then the following holds: 
			\begin{equation}\label{eq:supbound3}
				\max_{\ms\in \B^N}|h_{j_r}(\ms;\cE_{1,z_1},\ldots ,\cE_{2t,z_{2t}})|\lesssim \mathcal{R}[\bQ]_{N,1+|\cJ^*_{2t,j_r}|}(j_r, \cJ^*_{2t,i_r}). 
			\end{equation}
			\item [(d).] Set $\mU_{2t}:=\{a\in [2t]:\ a\ \mbox{is odd},\ \mathcal{U}_{a,z_a}=\phi\}$ and $\mU_{2t}^\star:=\{M_{a-1,z_{a-1}} :\ a\in \mU_{2t}\}$. Then, provided $k_1\ge 1$, we have: 
			\begin{equation}\label{eq:supbound4}
				\max_{\ms\in \B^N}|U_{\ib,\jb}(\ms;\cU_{1,z_1},\ldots ,\cU_{2t,z_{2t}})|\lesssim \bQ^U_{N,|\mU_{2t}^\star|}(\mU_{2t}^\star),  
			\end{equation}
			where $\{\bQ^U_{N,k}\}_{N,k\ge 1}$ is a collection of tensors with non-negative entries satisfying $\sup_{N\ge 1}\sum_{\ell_1,\ldots ,\ell_k} \bQ^U_{N,k}(\ell_1,\ldots ,\ell_k)<\infty$ for all fixed $k\ge 1$.
			\item [(e).] Set $\mV_{2t}:=\{a\in [2t]:\ a\ \mbox{is odd},\ \mathcal{V}_{a,z_a}=\phi\}$ and $\mV_{2t}^\star:=\{M_{a-1,z_{a-1}} :\ a\in \mV_{2,2t}\}$. Then, provided $k_2\ge 1$, we have: 
			\begin{equation}\label{eq:supbound5}
				\max_{\ms\in \B^N}|V_{\ib,\jb}(\ms;\cV_{1,z_1},\ldots ,\cV_{2t,k_{zt}})|\lesssim \bQ^V_{N,|\mU_{2t}^\star|}(\mU_{2t}^\star),  
			\end{equation}
			where $\{\bQ^V_{N,k}\}_{N,k\ge 1}$ is a collection of tensors with non-negative entries satisfying $\sup_{N\ge 1}\sum_{\ell_1,\ldots ,\ell_k} \bQ^V_{N,k}(\ell_1,\ldots ,\ell_k)<\infty$ for all fixed $k\ge 1$.
		\end{enumerate}
	\end{lemma}
	To understand the implications of~\cref{lem:supbound}, we introduce the notion of rank for every node of the tree, in the same spirit as rank of a function, as defined in~\cref{def:rankfub}.
	\begin{definition}[Rank of a node]\label{def:ranknode}
		Consider any node $R_{z_1,\ldots ,z_t}\equiv R_{z_1,\ldots ,z_t}(i_1,\ldots ,i_p,j_1,\ldots ,j_q)$ of the tree constructed in Algorithm~\ref{alg:construct_tree}-\ref{alg:construct_treep2}. Note that it is indexed by $(i_1,\ldots ,i_p,j_1,\ldots ,j_q)\in\Theta_{N,p+q}$ (see step 1 of~\cref{alg:construct_tree}). Then the rank of a node  $R_{z_1,\ldots ,z_t}$ is given by:
		$$\mathrm{rank}\left(\sum_{(i_1,\ldots ,i_p.j_1,\ldots ,j_q)\in\Theta_{N,p+q}} R_{z_1,\ldots ,z_t}(i_1,\ldots ,i_p,j_1,\ldots ,j_q)\right)$$
		in the sense of~\cref{def:rankfub}.
	\end{definition}
	At an intuitive level,~\cref{lem:supbound} implies that as we go lower and lower down the decision tree constructed in~\cref{alg:construct_tree}, the ranks of successive nodes decreases. This is formalized in the subsequent results.
	\begin{prop}\label{prop:degbound}
		Suppose $R_0\rightarrow R_{z_1}\rightarrow\ldots \rightarrow R_{z_1,\ldots ,z_{2T}}$ is a branch of the tree constructed in Algorithm~\ref{alg:construct_tree}-\ref{alg:construct_treep2} where $R_{z_1,\ldots ,z_{2T}}$ is a leaf node (see~\cref{def:leafnode}). Recall the construction of $(\cD_{a,z_a},\cE_{a,z_a},M_{a,z_a},\cU_{a,z_a},\cV_{a,z_a})_{a\in [2T]}$ from Algorithms~\ref{alg:construct_tree}-\ref{alg:construct_treep2}. Then the following conclusion holds under Assumptions~\ref{as:cmean}~and~\ref{as:coeffalt}:
		$$\mathrm{rank}\bigg(\sum_{(\ib,\jb)} R_{z_1,\ldots ,z_{2T}}(\ib,\jb)\bigg)\leq p+q-\max{\left(T,\Bigg|\cup_{a=1}^{T} (\cE_{2a-2,z_{2a-2}}\setminus M_{2a-2,z_{2a-2}})\setminus \cE_{2a,z_{2a}}\Bigg|\right)}.$$
	\end{prop}
	The following lemma complements~\cref{prop:degbound} in characterizing all leaf nodes whose contribution is asymptotically negligible.
	\begin{lemma}\label{lem:negterm}
		Consider the same setting and assumptions as in~\cref{prop:degbound}, with $R_{z_1,\ldots ,z_{2T}}$ being the leaf node. Then $\mathrm{rank}(\sum_{(\ib,\jb)}R_{z_1,\ldots ,z_{2T}}(\ib,\jb))<k/2$ if any of the following conclusions hold:
		\begin{enumerate}
			\item[(i)] $T\neq q/2$.
			\item[(ii)] there exists $p_0\in [p]$ such that $l_{p_0}>2$.
			\item[(iii)] there exists $1\leq a_0\leq 2T$ such that $\bar{\cD}_{a_0,z_{a_0}}\neq \phi$.
			\item[(iv)] there exists $1\leq a_0\leq T$ such that $M_{2a_0-1,z_{2a_0-1}}\in \cup_{a=1}^T (\bar{\cD}_{2a-1,z_{2a-1}}\cup \bar{\cE}_{2a-1,z_{2a-1}})$. 
			\item[(v)] there exists $1\le a_0\le T$ such that $\cU_{2a_0-1,z_{2a_0-1}}=\phi$.
			\item[(vi)] there exists $1\le a_0\le T$ such that $\cV_{2a_0-1,z_{2a_0-1}}=\phi$.
			\item[(vii)] there exists $1\leq a_0\leq T$ such that $$\bar{\cE}_{2a_0-1,z_{2a_0-1}}\cap \left(\{j_1,\ldots ,j_q\}\setminus \cE_{2a_0-2,z_{2a_0-2}}\right)\neq \phi.$$
			\item[(viii)] there exists $1\leq a_0\leq T$ such that $\big|(\cE_{2a_0-2,z_{2a_0-2}}\setminus M_{2a_0-2,z_{2a_0-2}})\setminus \cE_{2a_0,z_{2a_0}}\big|\neq 1$ or $|\bar{\cE}_{2a_0-1,z_{2a_0-1}}|\neq 1$ or $|\bar{\cE}_{2a_0-1,z_{2a_0-1}}\cap \cE_{2a_0-2,z_{2a_0-2}}|\neq 1$.
			\item[(ix)] there exists $1\leq a_0\leq T$ such that $((\cE_{2a_0-2,z_{2a_0-2}}\cap \cE_{2a_0-1,z_{2a_0-1}})\setminus\cE_{2a_0,z_{2a_0}})\neq \phi$.	
		\end{enumerate}
	\end{lemma}
	
	\begin{lemma}\label{lem:initremove}
		Suppose Assumptions \ref{as:coeffalt} and \ref{as:bddrowsum} hold. Then 
		\begin{align*}
			&\;\;\;\;\E T_N^{k} U_N^{k_1} V_N^{k_2} \\ &=\E\left[\frac{1}{N^{k/2}}\sum_{\substack{(\ell_1,\ldots ,\ell_p,\\ q)\in \mathcal{C}_k}}\sum_{\substack{(i_1,\ldots ,i_p,\\ j_1,\ldots ,j_q)\in \Theta_{N,p+q}}}\prod_{r=1}^p (c_{i_r}(\si_{i_r}-t_{i_r}))^{\ell_r} \prod_{r=1}^q (c_{j_r}(\si_{j_r}-t_{j_r})) U_{N,\ib,\jb}^{k_1} V_{N,\ib,\jb}^{k_2}\right] + o(1), 
		\end{align*}
		for all $k,k_1,k_2\in \N\cup\{0\}$, where $U_{N,\ib,\jb}$ and $V_{N,\ib,\jb}$ are defined in \eqref{eq:reducedex}, $\Theta_{N,p+q}$ is defined in \eqref{eq:thnr}, and $\mathcal{C}_k$ is defined in \eqref{eq:unag}.
	\end{lemma}
	\section{Proof of~\cref{lem:surviveterm}}\label{sec:pfsurviveterm}
	This section is devoted to proving our main technical lemma, i.e.,~\cref{lem:surviveterm} using~\cref{prop:degbound}, Lemmas~\ref{lem:negterm}~and~\ref{lem:initremove}.
	\begin{proof}
		\emph{Part (a).} Recall the construction of the decision tree in Algorithm~\ref{alg:construct_tree}-\ref{alg:construct_treep2} for fixed $(i_1,\ldots ,i_p,j_1,\ldots ,j_q)\in \Theta_{N,p+q}$. The nodes are indexed by $R_{z_1,\ldots ,z_{2T}}\equiv R_{z_1,\ldots ,z_{2T}}(i_1,\ldots ,i_p,j_1,\ldots ,j_q)$. Note that by~\cref{prop:setinc}, part (a), we get:
		\begin{equation*}
			f_L(\ms)=\sum_{R_{z_1,\ldots ,z_{2T}}\ \mbox{is a leaf node}}\; \sum_{(i_1,\ldots ,i_p,j_1,\ldots ,j_q)\in \Theta_{N,p+q}} R_{z_1,\ldots ,z_{2T}}(i_1,\ldots ,i_p,j_1,\ldots ,j_q).
		\end{equation*}
		If $\exists\ l_i>2$, then by~\cref{lem:negterm} (part (ii)), $\mbox{rank}(\sum_{(\ib,\jb)} R_{z_1,\ldots ,z_{2T}}(\ib,\jb))<k/2$ (see~\cref{def:ranknode} to recall the definition of $\mbox{rank}(R_{z_1,\ldots ,z_{2T}})$). Also if $q$ is odd, then $T\neq q/2$ and by~\cref{lem:negterm} (part (i)), $\mbox{rank}(\sum_{(\ib,\jb)} R_{z_1,\ldots ,z_{2T}}(\ib,\jb))<k/2$. As the number of leaf nodes is bounded in $N$ (by~\cref{prop:boundlen}), therefore $\mbox{rank}(f_L)\le k/2$. This completes the proof of part (a).
		
		\emph{Proof of (b).} In this part, $l_r=2$ for $r\in [p]$ and $q$ is even. Set $\FR:=\{R_{z_1,\ldots ,z_{2T}}:\ R_{z_1,\ldots ,z_{2T}}\ \mbox{is a leaf node}\}$ and note that $\FR=(\FR\cap\FB)\cup (\FR\cap\FB^c)$ where,
		\begin{align*}\label{eq:defset}\FB&:=\{R_{z_1,\ldots ,z_{2T}}:\ T=q/2,\ |\bar{\cE}_{2a-1,z_{2a-1}}|= 1,\ \bar{\cD}_{a,k_a}=\phi,\ ((\cE_{2a-2,z_{2a-2}}\cap \cE_{2a-1,z_{2a-1}})\setminus\cE_{2a,z_{2a}})=\phi,\nonumber \\ &\bar{\cE}_{2a-1,z_{2a-1}}\cap \left(\{j_1,\ldots ,j_q\}\setminus \cE_{2a-2,z_{2a-2}}\right)= \phi,\ M_{2a-1,z_{2a-1}}\notin \cup_{t=1}^T (\bar{\cD}_{2t-1,z_{2t-1}}\cup \bar{\cE}_{2t-1,z_{2t-1}}), \\ &\big|(\cE_{2a-2,z_{2a-2}}\setminus M_{2a-2,z_{2a-2}})\setminus \cE_{2a,z_{2a}}\big|=1,\ |\bar{\cE}_{2a-1,z_{2a-1}}\cap \cE_{2a-2,z_{2a-2}}|= 1\ \forall a\in [T],\\ &\ \cU_{2a-1,z_{2a-1}}\neq\phi \ \forall a\in [T],\ \cV_{2a-1,z_{2a-1}}\neq \phi\ \forall a\in [T]\}.
		\end{align*}
		In particular, we have intersected all the events in~\cref{lem:negterm} to form the set $\FB$. Consequently, by~\cref{lem:negterm}, $\mbox{rank}(R_{z_1,\ldots ,z_{2T}})<k/2$ for all $R_{z_1,\ldots ,z_{2T}}\in \FR\cap\FB^c$. Therefore, it follows that:
		\begin{equation}\label{eq:surviveterm1}
			\E[N^{-k/2}f_L(\ms)]\leftrightarrow N^{-k/2}\E\left[\sum_{R_{z_1,\ldots ,z_{2T}}\in \FR\cap\FB} \; \sum_{(\ib,\jb)\in \Theta_{N,p+q}} R_{z_1,\ldots ,z_{2T}}(\ib,\jb)\right].
		\end{equation} 
		Next, for $1\leq b\leq T$, let us define:
		\begin{equation*}
			S_b:=\cup_{r=b}^{T} M_{2r-1,z_{2r-1}}.
		\end{equation*}
		Note that by~\eqref{eq:surviveterm1}, we will now restrict to the set of leaf nodes in $\FR\cap\FB$. For any $r\in [p]$, $i_r\in \cD_{2a-1,z_{2a-1}}\cup \bar{\cD}_{2a-1,z_{2a-1}}$ for all $a\in [T]$. As $\bar{\cD}_{2a-1,z_{2a-1}}=\phi$ for all $a\in [T]$, $i_r\in \cD_{2a-1,z_{2a-1}}=\cD_{2a,z_{2a}}$ (see~\cref{prop:setinc}, part (b)) for all $a\in [T]$. Consequently, we have:
		\begin{equation}\label{eq:surviveterm2}
			h_{i_r}(\ms;\cD_{1,k_1},\ldots ,\cD_{2T,z_{2t}})=(c_{i_r}(g(\sigma_{i_r})-t_{i_r}^{S_1}))^2.
		\end{equation}
		Next we will focus on $U_{N,\ib,\jb}$. As $\cU_{2a-1,z_{2a-1}}\neq \phi$ for all $a\in [T]$ for all leaf nodes in $\FR\cap\FB$. Therefore $\cU_{2a-1,z_{2a-1}}= \{M_{2a-1,z_{2a-1}}\}$ for all $a\in [T]$. As a result, 
		\begin{equation}\label{eq:survivorterm1}
			U_{N,\ib,\jb}^{k_1}(\ms;\cU_{1,z_1},\ldots ,\cU_{2T,z_{2t}})=\big(U_{N,\ib,\jb}^{(S_1)}\big)^{k_1}.
		\end{equation}
		In a similar vein, we also get that for leaf nodes in $\FR\cap\FB$, we also have 
		\begin{equation}\label{eq:survivorterm1}
			V_{N,\ib,\jb}^{k_2}(\ms;\cV_{1,z_1},\ldots ,\cV_{2T,z_{2t}})=\big(V_{N,\ib,\jb}^{(S_1)}\big)^{k_2}.
		\end{equation}
		
		Next we will focus on $j_r$ for $r\in [q]$. As $R_{z_1,\ldots ,z_{2T}}$ is a leaf node, we must have $\cE_{2T,z_{2T}}=\phi$ (see step 17 of~\cref{alg:construct_treep2}). Further as we have restricted to $\FR\cap\FB$, by using~\cref{prop:setinc}, part (e), we get:
		\begin{equation}\label{eq:svtm1}
			\left(\cup_{a=1}^T \big(\bar{\cE}_{2a-1,z_{2a-1}} \cap \cE_{2a-2,z_{2a-2}}\big)\right)\cup \{M_{1,k_1},\ldots ,M_{2T-1,z_{2T-1}}\}=\{j_1,\ldots ,j_q\}.
		\end{equation}
		
		We now consider two disjoint cases: (a) $j_r\in \{M_{1,k_1},\ldots ,M_{2T-1,z_{2t-1}}\}$ or (b) $j_r\in \cup_{a=1}^T \big(\bar{\cE}_{2a-1,z_{2a-1}} \cap \cE_{2a-2,z_{2a-2}}\big)$. 
		
		\emph{Case (a).} If $j_r\in \{M_{1,k_1},\ldots ,M_{2T-1,z_{2t-1}}\}$, then there exists $a^*(r)$ such that $j_r=M_{2a^*(r)-1,z_{2a^*(r)-1}}$. Therefore, $h_{j_r}(\ms;\cE_{1,k_1},\ldots ,\cE_{2a^*(r),z_{2a^*(r)}})=c_{j_r}(g(\sigma_{j_r})-t_{j_r})$. Also, as we have restricted to the case $M_{2a-1,z_{2a-1}}\notin \cup_{t=1}^T (\cD_{2t-1,z_{2t-1}}^c\cup \cE_{2t-1,z_{2t-1}}^c)$ for any $a\in [T]$,  therefore, we have:
		\begin{equation}\label{eq:surviveterm3}
			h_{j_r}(\ms;\cE_{1,k_1},\ldots ,\cE_{2T,z_{2T}})=c_{j_r}(g(\sigma_{j_r})-t_{j_r}^{S_{a^*(r)+1}}).
		\end{equation}
		
		\emph{Case (b).} Next consider the case when $j_r\in \cup_{a=1}^T \big(\bar{\cE}_{2a-1,z_{2a-1}} \cap \cE_{2a-2,z_{2a-2}}\big)$. Then, by~\cref{prop:setinc}, part (g), there exists a unique $\tilde{a}(r)$ such that $j_r\in \cE_{2\tilde{a}(r)-1,z_{2\tilde{a}(r)-1}}^c \cap \cE_{2\tilde{a}(r)-2,z_{2\tilde{a}(r)-2}}$. Consequently, we have:
		\begin{align}\label{eq:surviveterm7}
			h_{j_r}(\ms;\cE_{1,k_1},\ldots ,\cE_{2\tilde{a}(r),z_{2\tilde{a}(r)}})&=c_{j_r}(\si_{j_r}-t_{j_r})-c_{j_r}(\si_{j_r}-t_{j_r}^{M_{2\tilde{a}(r)-1,z_{2\tilde{a}(r)-1}}}) \nonumber \\&=c_{j_r}\left(t_{j_r}^{M_{2\tilde{a}(r)-1,z_{2\tilde{a}(r)-1}}}-t_{j_r}\right)
		\end{align}
		Recall that, under $\FR\cap\FB$, we have $\cE_{2a-1,z_{2a-1}}^c\cap \left(\{j_1,\ldots ,j_q\}\setminus \cE_{2a-2,z_{2a-2}}\right)= \phi$ for all $a\in [T]$. This directly implies that $j_r\notin\cE_{2\bar{a}-1,z_{2\bar{a}-1}}$ for all $\bar{a}>\tilde{a}(r)$. Therefore, using~\eqref{eq:surviveterm7}, we get:
		\begin{align}\label{eq:surviveterm4}
			&\;\;\;\;h_{j_r}(\ms;\cE_{1,k_1},\ldots ,\cE_{2T,z_{2t}})=c_{j_r}\left(t_{j_r}^{S_{\tilde{a}(r)}}-t_{j_r}^{S_{\tilde{a}(r)+1}}\right)
		\end{align}
		
		Having obtained the form of $h_{j_r}(\cdot;\cE_{1,k_1},\ldots ,\cE_{2T,z_{2T}})$ for each $j_r$, we now move on to the rest of the proof.
		
		With $h_{i_r}(\ms;\cD_{1,z_1},\ldots ,\cD_{2T,z_{2T}})$ and  $h_{j_r}(\ms;\cE_{1,z_1},\ldots ,\cE_{2T,z_{2T}})$ as obtained in~\eqref{eq:surviveterm2},~and~\eqref{eq:surviveterm4}, the following holds by definition (see~\eqref{eq:gen2ptreca}):
		\begin{align}\label{eq:surviveterm5}
			&\;\;\;\sum_{(\ib,\jb)\in\Theta_{N,p+q}} R_{z_1,\ldots ,z_{2T}}(\ib,\jb)\nonumber \\ &=\sum_{(\ib,\jb)\in\Theta_{N,p+q}} \left(\prod_{r=1}^p h_{i_r}(\ms;\cD_{1,z_1},\ldots ,\cD_{2T,z_{2T}})\right)\left(\prod_{r=1}^q h_{j_r}(\ms;\cE_{1,z_1},\ldots ,\cE_{2T,z_{2T}})\right)\nonumber \\ &\qquad\qquad U_{N,\ib,\jb}^{k_1}(\ms;\cU_{1,z_1},\ldots ,\cU_{2T,z_{2T}}) V_{N,\ib,\jb}^{k_2}(\ms;\cV_{1,z_1},\ldots ,\cV_{2T,z_{2T}}).
		\end{align}
		As $\mbox{rank}(sum_{(\ib,\jb)\in\Theta_{N,p+q}} R_{z_1,\ldots,z_{2T}}(\ib,\jb))\le k/2$ for all leaf nodes in $\FR\cap\FB$, by the same calculation as in the proof of~\eqref{lem:negterm} (part (iii)), it is easy to show that:
		\begin{align}\label{eq:surviveterm6}
			&\mbox{rank}\bigg( \sum_{(\ib,\jb)\in\Theta_{N,p+q}} R_{z_1,\ldots ,z_{2T}}(\ib,\jb)-\sum_{(\ib,\jb)\in\Theta_{N,p+q}}\left(\prod_{r=1}^p h_{i_r}(\ms)\right)\left(\prod_{j_r\in \{M_{1,k_1},\ldots ,M_{2T-1,z_{2t-1}}\}} h_{j_r}(\ms)\right) \nonumber \\ &\quad \Bigg(\prod_{\substack{j_r\in \cup_{a=1}^T \big(\bar{\cE}_{2a-1,z_{2a-1}}\\ \cap \cE_{2a-2,z_{2a-2}}\big)}} h_{j_r}(\ms;\cE_{1,k_1},\ldots ,\cE_{2\tilde{a}(r),z_{2\tilde{a}(r)}})\Bigg)\big(U_N^{k_1}\big)\big(V_N^{k_2}\big)\bigg)<k/2,
		\end{align}
		where $h_{j_r}(\ms;\cE_{1,k_1},\ldots ,\cE_{2\tilde{a}(r),z_{2\tilde{a}(r)}})$ is defined as in~\eqref{eq:surviveterm7}. By combining~\eqref{eq:surviveterm1},~\eqref{eq:surviveterm5}, and~\eqref{eq:surviveterm6}, the following equivalence holds:
		\begin{align}\label{eq:surviveterm8}
			&\;\;\;\;\E[N^{-k/2}f_L(\ms)]\nonumber \\ &\leftrightarrow N^{-k/2}\E\Bigg[\sum_{R_{z_1,\ldots ,z_{2T}}\in \FR\cap\FB} \sum_{(\ib,\jb)\in\Theta_{N,p+q}} \left(\prod_{r=1}^p h_{i_r}(\ms)\right) \left(\prod_{j_r\in \{M_{1,k_1},\ldots ,M_{2T-1,z_{2t-1}}\}} h_{j_r}(\ms)\right) \nonumber \\ &\qquad\qquad \Bigg(\prod_{\substack{j_r\in \cup_{a=1}^T \big(\bar{\cE}_{2a-1,z_{2a-1}}\\ \cap \cE_{2a-2,z_{2a-2}}\big)}} h_{j_r}(\ms;\cE_{1,k_1},\ldots ,\cE_{2\tilde{a}(r),z_{2\tilde{a}(r)}})\Bigg) U_N^{k_1} V_N^{k_2} \Bigg].
		\end{align}
		Let us now write the right hand side of~\eqref{eq:surviveterm8} in terms of matchings (see~\cref{def:matching} for details) as in the statement of~\cref{lem:surviveterm} (part (b)). For $R_{z_1,\ldots ,z_{2T}}\in \FR\cap\FB$, $\big(\bar{\cE}_{2a-1,z_{2a-1}} \cap \cE_{2a-2,z_{2a-2}}\big)$ are all singletons for $a\in [T]$. Let the set \begin{equation}\label{eq:surviveterm9}
			\mfm:=\{(\mfm_{1,1},\mfm_{1,2}),(\mfm_{2,1},\mfm_{2,2}),(\mfm_{3,1},\mfm_{3,2}),\ldots ,(\mfm_{q/2,1},\mfm_{q/2,2})\}
		\end{equation}
		be defined such that $j_{\mfm_{a,1}}=M_{2a-1,z_{2a-1}}$ and $\{j_{\mfm_{a,2}}\}=\big(\bar{\cE}_{2a-1,z_{2a-1}} \cap \cE_{2a-2,z_{2a-2}}\big)$ for $a\in [q/2]$. By~\eqref{eq:svtm1}, $(j_{\mfm_{a,1}},j_{\mfm_{a,2}})_{a=1}^{[q/2]}$ induces a partition on $\{j_1,\ldots ,j_q\}$. By the definition of $M_{2a-1,z_{2a-1}}$ (see steps 21 and 22 in~\cref{alg:construct_treep2}), $\mfm_{a,1}>\mfm_{a,2}$ and $\mfm_{a,1}<\mfm_{a',1}$ for $a>a'$. Therefore, the set $\mfm$ in~\eqref{eq:surviveterm9} yields a matching on the set $[q]$ (in the sense of~\cref{def:matching}). With this observation, note that:
		\begin{align}\label{eq:surviveterm10}
			&\;\;\;\; h_{j_{\mfm_{a,2}}}(\ms;\cE_{1,k_1},\ldots ,cE_{2a,z_{2a}})h_{j_{\mfm_{a,1}}}(\ms)=c_{j_{\mfm_{a,1}}}c_{j_{\mfm_{a,2}}}(g(\sigma_{j_{\mfm_{a,1}}})-t_{j_{\mfm_{a,1}}})(t_{j_{\mfm_{a,2}}}^{j_{\mfm_{a,1}}}-t_{j_{\mfm_{a,2}}}).
		\end{align}
		Finally, by using~\eqref{eq:surviveterm8},~\eqref{eq:surviveterm10}, and~\eqref{eq:surviveterm2}, we have:
		\begin{align*}
			&\;\;\;\;\;\E[N^{-k/2}f_L(\ms)]\\ &\leftrightarrow N^{-k/2}\E\Bigg[\sum_{\substack{R_{k_1,\ldots ,k_T}\\ \in \FR\cap\FB}} \sum_{(\ib,\jb)\in\Theta_{N,p+q}}\left(\prod_{r=1}^p h_{i_r}(\ms)\right)\left(\prod_{a=1}^{q/2} h_{j_{\mu_{a,2}}}(\ms;\cE_{1,k_1},\ldots ,\cE_{2a,z_{2a}})h_{j_{\mu_{a,1}}}(\ms)\right)\nonumber U_N^{k_1} V_N^{k_2}\Bigg]\\ &\leftrightarrow\E\Bigg[\sum_{\mfm\in\M([q])} \sum\limits_{(\ib,\jb)\in \Theta_{N,p+q}}\left(\prod_{r=1}^p c_{i_r}^2\left(g(\sigma_{i_r})-t_{i_r}\right)^2\right)\Bigg(\prod_{a=1}^{q/2} c_{j_{\mfm_{a,1}}}c_{j_{\mfm_{a,2}}}(g(\sigma_{j_{\mfm_{a,1}}})-t_{j_{\mfm_{a,1}}})(t_{j_{\mfm_{a,2}}}^{j_{\mfm_{a,1}}}-t_{j_{\mfm_{a,2}}})\Bigg)\nonumber \\ & \qquad \qquad U_N^{k_1}V_N^{k_2}\Bigg].
		\end{align*}
		This completes the proof.	
	\end{proof}
	\section{Proofs from~\cref{sec:propdtree}}\label{sec:pfpropdtree}
	This Section is devoted to proving Lemmas  \ref{lem:supbound}, \ref{lem:negterm}, \ref{lem:initremove}, and \cref{prop:degbound}. To establish these results,  
	we begin this section by presenting a collection of set-theoretic results which follow immediately from our construction of the decision tree (as in Algorithm~\ref{alg:construct_tree}-\ref{alg:construct_treep2}). We leave the verification of these results to the reader. These properties will be leveraged in the proofs of the results in \cref{sec:propdtree}.
	\begin{prop}\label{prop:setinc}
		Consider a path $R_0\rightarrow R_{z_1}\rightarrow\ldots\rightarrow R_{z_1,\ldots ,z_{2t}}$ of the decision tree constructed in Algorithm~\ref{alg:construct_tree}-\ref{alg:construct_treep2}. Recall the construction of $(\cD_{a,z_a},\cE_{a,z_a},M_{a,z_a},\cU_{a,z_a},\cV_{a,z_a})_{a\in [t]}$ from Algorithm \ref{alg:construct_tree}-\ref{alg:construct_treep2}. Then, 
		\begin{enumerate}
			\item[(a).] Leaf nodes only occur in even-numbered generations of the tree.
			\item[(b).] For any positive integer $a$, $\cD_{2a-1,z_{2a-1}}=\cD_{2a,z_{2a}}$, $\cU_{2a-1,z_{2a-1}}=\cU_{2a,z_{2a}}$, $\cV_{2a-1,z_{2a-1}}=\cV_{2a,z_{2a}}$,  $M_{2a-1,z_{2a-1}}=M_{2a-2,z_{2a-2}}$, and $\{M_{2a-1},z_{2a-1}\}_{a=1}^t$ are $t$ distinct elements.
			
			\item[(c).] $\cE_{2a,z_{2a}}\subseteq (\cE_{2a-2,z_{2a-2}}\setminus M_{2a-2,z_{2a-2}})$ for any $a\in [t]$.
			\item[(d).] $M_{2a-1,z_{2a-1}}\notin \cup_{\ell=1}^{a} (\bar{\cD}_{2\ell-1,z_{2\ell-1}}\cup \bar{\cE}_{2\ell-1,z_{2\ell-1}})$ for $a\in [t]$.
			\item[(e).]  Recall $\cE_{0}=\{j_1,\ldots ,j_q\}$. For any $a\ge 1$, we have:
			\begin{align*}
				\cE_{2a,z_{2a}}\cup \big(\cup_{\ell=1}^a (\bar{\cE}_{2\ell-1,z_{2\ell-1}}&\cap \cE_{2\ell-2,z_{2\ell-2}})\big)\cup \left(\cup_{\ell=1}^{a} ((\cE_{2\ell-1,z_{2\ell-1}}\cap \cE_{2\ell-2,z_{2\ell-2}})\setminus \cE_{2\ell,z_{2\ell}})\right)\\ &\cup \{M_{1,z_1},\ldots ,M_{2a-1,z_{2a-1}}\}=\cE_{0}.
			\end{align*}
			\item[(f).] For any $a\geq 1$,
			$$(\cE_{2a-2,z_{2a-2}}\setminus M_{2a-2,z_{2a-2}})\setminus \cE_{2a}\supseteq (\bar{\cE}_{2a-1,z_{2a-1}}\cap \cE_{2a-2,z_{2a-2}})\cup ((\cE_{2a-1,z_{2a-1}}\cap \cE_{2a-2,z_{2a-2}})\setminus \cE_{2a,z_{2a}}),$$
			where the two sets on the right hand side above are disjoint.
			\item[(g).] The sets $\{\bar{\cE}_{2a-1,z_{2a-1}}\cap \cE_{2a-2,z_{2a-2}}\}_{a=1}^t$ are disjoint. Further, the two sets $\big(\cup_{a=1}^t (\bar{\cE}_{2a-1,z_{2a-1}}\cap \cE_{2a-2,z_{2a-2}})\big)$ and $\{M_{1,z_1},\ldots ,M_{2t-1,z_{2t-1}}\}$ are also disjoint.
			\item[(h)] For any $a\in [t]$, the following cannot hold simultaneously: $\bar{\cD}_{2a-1,z_{2a-1}}=\phi$, $\bar{\cE}_{2a-1,z_{2a-1}}=\phi$, $\cU_{2a-1,z_{2a-1}}\neq \phi$, and $\cV_{2a-1,z_{2a-1}}\neq \phi$.
		\end{enumerate}
	\end{prop}
	
	\subsection{Proof of~\cref{lem:supbound}}
	To begin with, recall the notation $\De(\cdot;\cdot;\cdot)$ from~\cref{sec:howtover}.
	\begin{proof}
		Parts (a), (b), and (c) of~\cref{lem:supbound} are similar. We will only prove part (b) here among these three. We will also prove parts (d) and (e). 
		
		\emph{Part (b).} Define  $\tI_{2t,i_r}:=\{M_{a-1,z_{a-1}}:\ a\in [2t]\}\setminus \mathcal{I}^*_{2t,i_r}$. By a simple induction, it follows that 
		$$h_{i_r}(\ms;\cD_0,\cD_{1,z_1},\ldots , \cD_{2t,z_{2t}})=\De(h_{i_r};\tI_{2t.i_r};\mathcal{I}^\star_{2t,i_r}).$$
		As $h_{i_r}(\ms)=c_{i_r}(\si_{i_r}-t_{i_r})^{\ell_r}$, note that $$h_{i_r}(\ms)=(c_{i_r}(g(\si_{i_r}-t_{i_r})^{\ell_r}))=c_{i_r}^{\ell_r}\sum_{s=0}^{\ell_r} {\ell_r \choose s} (-1)^{\ell_r-s}(g(\si_{i_r}))^{s}(t_{i_r})^{\ell_r-s}.$$
		By combining the above displays, we get:
		\begin{align*}
			|h_{i_r}(\ms;\cD_0,\cD_{1,z_1},\ldots ,\cD_{2t,z_{2t}})|&=\Bigg|\De\left(c_{i_r}^{\ell_r}\sum_{s=0}^{\ell_r} {\ell_r \choose s} (-1)^{\ell_r-s}(g(\si_{i_r}))^{s}(t_{i_r})^{\ell_r-s};\tI_{2t,i_r};\mathcal{I}^*_{2t,i_r}\right)\Bigg|\\ &\leq |c_{i_r}^{\ell_r}|\sum_{s=0}^{\ell_r} {\ell_r \choose s} |g(\si_{i_r})|^{s}\Bigg|\De\left((t_{i_r})^{\ell_r-s};\tI_{2t,i_r};\mathcal{I}^*_{2t,i_r}\right)\Bigg|.
		\end{align*}
		By an application of \cref{lem:smoothcont}, part 1, we have:
		$$\Bigg|\De\left((t_{i_r})^{\ell_r-s};\tI_{2t,i_r};\mathcal{I}^*_{2t,i_r}\right)\Bigg|\lesssim \mathcal{R}[\bQ]_{N,1+|\mathcal{I}^*_{2t,i_r}|}(i_r, \mathcal{I}^*_{2t,i_r}).$$
		Therefore, 
		\begin{align*}
			|h_{i_r}(\ms;\cD_0,\cD_{1,z_1},\ldots ,\cD_{2t,z_{2t}})| \lesssim \mathcal{R}[\bQ]_{N,1+|\mathcal{I}^*_{2t,i_r}|}(i_r, \mathcal{I}^*_{2t,i_r})|c_{i_r}^{\ell_r}|\sum_{s=0}^{\ell_r} {\ell_r \choose s} \lesssim \mathcal{R}[\bQ]_{N,1+|\mathcal{I}^*_{2t,i_r}|}(i_r, \mathcal{I}^*_{2t,i_r}).
		\end{align*}
		This completes the proof.
		
		\emph{Part (d).} Define $\bar{\mU}_{2t}^\star:=\{M_{0,},M_{2,z_2},\ldots ,M_{2t,z_{2t}}\}\setminus \mU_{2t}^\star$. As 
		\begin{align*}
			U_{\ib,\jb}(\ms;\cU_{1,z_1},\ldots ,\cU_{2t,z_{2t}})=\De(f\circ U_{N,\ib,\jb};\bar{\mU}_{2t}^\star;\mU_{2t}^\star),
		\end{align*}
		where $f(x)=x^{k_1}$ and $U_{N,\ib,\jb}$ is defined in \eqref{eq:reducedex}. Our strategy is to first bound $\De(U_{N,\ib,\jb};\bar{\mU}_{2t}^\star;\mU_{2t}^\star)$, and then invoke \cref{lem:genmatrixbd}, parts 1 and 2. As $\{M_{0,},M_{2,z_2},\ldots ,M_{2t,z_{2t}}\}\subseteq \{j_1,\ldots ,j_q\}$, we have 
		\begin{align*}
			\big|\De(U_{N,\ib,\jb};\bar{\mU}_{2t}^\star;\mU_{2t}^\star)\big|&=\bigg|\De\left(\frac{1}{N}\sum_{a\neq (\ib,\jb)} (g(\si_a)^2-t_a^2);\bar{\mU}_{2t}^\star;\mU_{2t}^\star\right)\bigg| \\ &\lesssim N^{-1}\sum_{a\notin (\ib,\jb)}\big|\De(t_a^2;\bar{\mU}_{2t}^\star;\mU_{2t}^\star)\big| \lesssim N^{-1}\sum_{a\notin (\ib,\jb)}\mathcal{R}[\bQ]_{N,1+|\mU_{2t}^\star|}(a,\mU_{2t}^\star).
		\end{align*}
		Without loss of generality, suppose that $\mU_{2t}^\star=\{j_1,\ldots ,j_r\}$. Then the above inequality can be written as 
		$$\big|\De(U_{N,\ib,\jb};\bar{\mU}_{2t}^\star;\mU_{2t}^\star)\big|\lesssim N^{-1}\sum_{a\notin (\ib,\jb)}\mathcal{R}[\bQ]_{N,1+r}(a,\{j_1,\ldots ,j_r\})=:\cT_{N,r}(j_1,\ldots ,j_r).$$
		By \cref{lem:smoothcont}, part (2), $\sup_{N\ge 1}\sum_{j_1,\ldots ,j_r} \cT_{N,r}(j_1,\ldots ,j_r)<\infty$. Now with $\cT$ as defined above, construct $\tcT$ as in \eqref{eq:newTprop}. The conclusion now follows with $Q^U_{N,|\mU_{2t}^\star|}=\tcT_{N,|\mU_{2t}^\star|}$, by \cref{lem:genmatrixbd}, parts 1 and 2.
		
		\emph{Part (e).} Define $\bar{\mV}_{2t}^\star:=\{M_{0,z_0}, M_{2,z_2}, \ldots , M_{2t,z_{2t}}\}\setminus \mV_{2t}^\star$. As in the proof of part (d), our strategy would be to bound $\De(V_{N,\ib,\jb};\bar{\mV}_{2t}^\star;\mV_{2t}^\star)$ and then apply \cref{lem:genmatrixbd}. 
		
		We note that 
		\begin{align}\label{eq:biv1}
			&\;\;\;\;\bigg|\De\left(\frac{1}{N}\sum_{k\neq \ell, (k,\ell)\notin (\ib,\jb)} (g(\si_k)-t_k)(t_{\ell}^k-t_{\ell});\bar{\mV}_{2t}^\star;\mV_{2t}^\star\right)\bigg|\nonumber \\&\le \frac{1}{N}\sum_{k\neq \ell,(k,\ell)\notin (\ib,\jb)}\bigg|g(\si_k)\De(t_{\ell}^k-t_{\ell};\bar{\mV}_{2t}^\star;\mV_{2t}^\star)\bigg|+\frac{1}{N}\sum_{k\neq \ell, (k,\ell)\notin (\ib,\jb)} \bigg|\De(t_k(t_{\ell}^k-t_{\ell});\bar{\mV}_{2t}^\star;\mV_{2t}^\star)\bigg| \nonumber\\ &\le \frac{1}{N}\sum_{k\neq \ell,(k,\ell)\notin (\ib,\jb)}\bQ_{N,2+|\mV_{2t}^\star|}(\ell,k,\mV_{2t}^\star) +\frac{1}{N}\sum_{k\neq \ell, (k,\ell)\notin (\ib,\jb)} \bigg|\De(t_k(t_{\ell}^k-t_{\ell});\bar{\mV}_{2t}^\star;\mV_{2t}^\star)\bigg|,
		\end{align}
		where the last inequality follows from the boundedness of $g(\cdot)$ and \eqref{eq:cmean1}. 
		To bound the second term in \eqref{eq:biv1}, we will show the following claim: 
		\begin{align}\label{eq:bivclaim1}
			\De(t_k(t_{\ell}-t_{\ell}^k);\bar{\mV}_{2t}^\star;\mV_{2t}^\star)=\sum_{D\subseteq \mV_{2t}^\star} \De(t_k;\bar{\mV}_{2t}^\star\cup (\mV_{2t}^\star\setminus D);D)\De(t_{\ell};\bar{\mV}_{2t}^\star;(k,\mV_{2t}^\star\setminus D))
		\end{align}
		for $k\neq \ell, (k,l)\notin (\ib,\jb)$. Let us first complete the proof assuming the above claim. without loss of generality, assume $\mV_{2t}^\star=\{j_1,\ldots ,j_r\}$. By combining \eqref{eq:biv1} and \eqref{eq:bivclaim1}, we get: 
		\begin{align*}
			&\;\;\;\;\bigg|\De\left(\frac{1}{N}\sum_{k\neq \ell, (k,\ell)\notin (\ib,\jb)} (g(\si_k)-t_k)(t_{\ell}^k-t_{\ell});\bar{\mV}_{2t}^\star;\mV_{2t}^\star\right)\bigg|\nonumber \\&\le \frac{1}{N}\sum_{k\neq \ell,(k,\ell)\notin (\ib,\jb)}\left(\bQ_{N,2+r}(\ell,k,\{j_1,\ldots ,j_r\})+\sum_{D\subseteq \{j_1,\ldots ,j_r\}} \bQ_{N,1+|D|}(k,D)\bQ_{N,2+r-|D|}(\ell,k,\{j_1,\ldots ,j_r\}\setminus D)\right)  \\ &=:\cT_{N,r}(j_1,\ldots ,j_r).
		\end{align*}
		By \eqref{eq:cmean2}, it is immediate that $\sum_{j_1,\ldots ,j_r} \cT_{N,r}(j_1,\ldots ,j_r)<\infty$. Now with $\cT$ as defined above, construct $\tcT$ as in \eqref{eq:newTprop}. The conclusion now follows with $Q^V_{N,|\mV_{2t}^\star|}=\tcT_{N,|\mV_{2t}^\star|}$, by using \cref{lem:genmatrixbd}, parts 1 and 2.
		
		\emph{Proof of \eqref{eq:bivclaim1}.} We will prove \eqref{eq:bivclaim1} by induction on $t$. 
		\emph{$t=1$ case.} Note from \cref{alg:construct_tree}, $M_{0,z_0}=j_q$. If $\cV_{1,z_1}=\phi$, then $\mV_2^\star=\{j_q\}$ and $\bar{\mV}_{2}^\star=\phi$. Then 
		\begin{align*}
			\De(t_k(t_{\ell}-t_{\ell}^k);\bar{\mV}_{2}^\star;\mV_{2}^\star)&=t_k(t_{\ell}-t_{\ell}^k)-t_k^{j_q}(t_{\ell}^{j_q}-t_{\ell}^{\{k,j_q\}}) \\ &=(t_k-t_k^{j_q})(t_{\ell}-t_{\ell}^k)+t_k^{j_q}(t_{\ell}-t_{\ell}^k-t_{\ell}^{j_q}+t_{\ell}^{\{k,j_q\}}).
		\end{align*}
		Also in this case, the subset $D$ in \eqref{eq:bivclaim1} can be either $\phi$ or $\{j_q\}$. Therefore, 
		\begin{align*}
			&\;\;\;\;\sum_{D\subseteq \mV_{2}^\star}\De(t_k;\bar{\mV}_{2}^\star\cup (\mV_{2}^\star\setminus D);D)\De(t_{\ell};\bar{\mV}_{2}^\star;(k,\mV_{2}^\star\setminus D)) \\ &=\De(t_k;\{j_q\};\phi)\De(t_{\ell};\phi;\{k,j_q\})+\De(t_k;\phi;\{j_q\})\De(t_{\ell};\phi;\{k\}) = t_k^{j_q}(t_{\ell}-t_{\ell}^k-t_{\ell}^{j_q}+t_{\ell}^{\{k,j_q\}})+(t_k-t_k^{j_q})(t_{\ell}-t_{\ell}^k).
		\end{align*}
		Therefore \eqref{eq:bivclaim1} holds. The other case is $\cV_{1,z_1}=\{j_q\}$, which yields $\mV_{2}^\star=\phi$ and $\bar{\mV}_2^\star=\{j_q\}$. Then 
		\begin{align*}
			\De(t_k(t_{\ell}-t_{\ell}^k);\bar{\mV}_{2}^\star;\mV_{2}^\star)&=t_k^{j_q}(t_{\ell}^{j_q}-t_{\ell}^{\{k,j_q\}}).
		\end{align*}
		Also in this case, the subset $D$ in \eqref{eq:bivclaim1} must be $\phi$. Therefore, 
		\begin{align*}
			&\;\;\;\;\sum_{D\subseteq \mV_{2}^\star} (-1)^{|D|+1}\De(t_k;\bar{\mV}_{2}^\star\cup (\mV_{2}^\star\setminus D);D)\De(t_{\ell};\bar{\mV}_{2}^\star;(k,\mV_{2}^\star\setminus D)) \\ &=\De(t_k;\{j_q\};\phi)\De(t_{\ell};\{j_q\};\{k\})=t_k^{j_q}(t_{\ell}^{j_q}-t_{\ell}^{\{k,j_q\}}).
		\end{align*}
		This completes the proof for the base case. 
		
		\emph{Induction hypothesis.} We suppose \eqref{eq:bivclaim1} holds for $t\le \st$. 
		
		\emph{$t=\st+1$ case.} Suppose $M_{2\st+1,z_{2\st+1}}=j_r$ for some $1\le r\le q$, where $j_r\notin \mV_{2t}^\star \cup \bar{\mV}_{2t}^\star$. If $\cV_{2\st+1,z_{2\st+1}}=\phi$, then $\mV_{2\st+2}^\star=\mV_{2\st}^\star\cup \{j_r\}$ and $\bar{\mV}_{2\st+2}^\star=\bar{\mV}_{2\st}$. Then, by the induction hypothesis, we have: 
		\begin{align*}
			&\;\;\;\;\De(t_k(t_{\ell}-t_{\ell}^k);\bar{\mV}_{2\st+2}^\star;\mV_{2\st+2}^\star) \\ &=\sum_{D\subseteq \mV_{2\st}^\star} \De\left(\De(t_k;\bar{\mV}_{2\st}^\star\cup (\mV_{2\st}^\star\setminus D);D)\De(t_{\ell};\bar{\mV}_{2\st}^\star;(k,\mV_{2\st}^\star\setminus D));\phi;\{j_r\}\right) \\ &=\sum_{D\subseteq \mV_{2\st}^\star} \bigg(\De(t_k;\bar{\mV}_{2\st+2}^\star\cup (\mV_{2\st}^\star\setminus D);D\cup \{j_r\})\De(t_{\ell};\bar{\mV}_{2\st+2}^\star;(k,\mV_{2\st}^\star\setminus D)) \\ &\qquad\qquad +\De(t_k;\bar{\mV}_{2\st+2}^\star\cup (\mV_{2\st}^\star\setminus D) \cup \{j_r\};D)\De(t_{\ell};\bar{\mV}_{2\st+2}^\star;(k,j_r,\mV_{2\st}^\star\setminus D))\bigg) \\ &=\sum_{D\subseteq \mV_{2\st}^\star} \bigg(\De(t_k;\bar{\mV}_{2\st+2}^\star\cup (\mV_{2\st+2}^\star\setminus (D\cup \{j_r\}));(D\cup \{j_r\}))\De(t_{\ell};\bar{\mV}_{2\st+2}^\star;(k,(\mV_{2\st+2}^\star\setminus (D\cup \{j_r\}))) \\ &\qquad\qquad +\De(t_k;\bar{\mV}_{2\st+2}^\star\cup (\mV_{2\st+2}^\star\setminus D) ;D)\De(t_{\ell};\bar{\mV}_{2\st+2}^\star;(k,\mV_{2\st+2}^\star\setminus D))\bigg) \\ &=\sum_{D\subseteq \mV_{2\st+2}^\star} \De(t_k;\bar{\mV}_{2\st+2}^\star\cup (\mV_{2\st+2}^\star\setminus D);D)\De(t_{\ell};\bar{\mV}_{2\st+2}^\star;(k,\mV_{2\st+2}^\star\setminus D)).
		\end{align*}
		Therefore \eqref{eq:bivclaim1} holds. The other case where $\cV_{2\st+1,z_{2\st+1}}=j_r$, the required equality is immediate. This completes the proof of \eqref{eq:bivclaim1}.
	\end{proof}

	\subsection{Proof of~\cref{prop:degbound}}
	Given any subset $D\subseteq \{1, 2,\ldots ,q\}$, $r\notin D$, and $\tilde{D}\subseteq D$, with $|D|,|\tilde{D}|\geq 1$, define
	\begin{equation}\label{eq:obsdef}\mathcal{R}[\bQ]_{N,1+|D\setminus \tilde{D}|}(r,(-j_i,\ i\in \tilde{D})):=\sum_{j_i,\ i\in \tilde{D}} \mathcal{R}[\bQ]_{N,1+|D\setminus \tilde{D}|}(r,D).\end{equation}
	By~\cref{lem:smoothcont}, part (2), we easily observe that:
	\begin{equation}\label{eq:obsdegbound}
		\sup_{N\ge 1} \max_{r}\max_{j_i,\ i\in D\setminus \tilde{D}} \mathcal{R}[\bQ]_{N,1+|D\setminus \tilde{D}|}(r,(-j_i,\ i\in\tilde{D}))< \infty.
	\end{equation}
	Similarly, we define 
	\begin{equation}\label{eq:obsdeg1}
		\bQ^U_{N,|D\setminus \tilde{D}|}(-j_i, i\in\tilde{D}) := \sum_{j_i, i\in\tilde{D}} \bQ^U_{N,1+|D\setminus \tilde{D}|}(D),
	\end{equation}
	and 
	\begin{equation}\label{eq:obsdeg2}
		\bQ^V_{N,|D\setminus \tilde{D}|}(-j_i, i\in\tilde{D}) := \sum_{j_i, i\in\tilde{D}} \bQ^V_{N,1+|D\setminus \tilde{D}|}(D).
	\end{equation}
	By \cref{lem:supbound}, parts (d) and (e), we get: 
	\begin{equation}\label{eq:obsdegbound1}
		\sup_{N\ge 1} \max_{j_i,\ i\in D\setminus \tilde{D}} \bQ^U_{N,|D\setminus \tilde{D}|}(-j_i,\ i\in\tilde{D})< \infty, \quad \mbox{and}\quad \sup_{N\ge 1} \max_{j_i,\ i\in D\setminus \tilde{D}} \bQ^V_{N,|D\setminus \tilde{D}|}(-j_i,\ i\in\tilde{D})< \infty.
	\end{equation}
	We will use \eqref{eq:obsdegbound} and \eqref{eq:obsdegbound1} multiple times in the proof.
	
	By construction, the collection of sets $\{(\cE_{2a-2,z_{2a-2}}\setminus M_{2a-2,z_{2a-2}})\setminus \cE_{2a,z_{2a}}\}_{a=1}^T$ are disjoint. Therefore, $\big|\cup_{a=1}^T ((\cE_{2a-2,z_{2a-2}}\setminus M_{2a-2,z_{2a-2}})\setminus \cE_{2a,z_{2a}})\big|=\sum_{a=1}^T |(\cE_{2a-2,z_{2a-2}}\setminus M_{2a-2,z_{2a-2}})\setminus \cE_{2a,z_{2a}}|$. We will therefore separately show the following: 
	
	(a) $\mbox{rank}(R_{z_1,\ldots ,z_{2T}})\leq p+q-T$, and 
	
	(b) $\mbox{rank}(R_{z_1,\ldots ,z_{2T}})\leq p+q-\sum_{a=1}^T |(\cE_{2a-2,z_{2a-2}}\setminus M_{2a-2,z_{2a-2}})\setminus \cE_{2a,z_{2a}}|$.
	
	\emph{For part (a)}. Let us enumerate $\mH_T:=\cup_{a=1}^T (\bar{\cD}_{2a-1,z_{2a-1}}\cup\, \bar{\cE}_{2a-1,z_{2a-1}})$ arbitrarily as $(\beta_1,\ldots ,\beta_{|\mH_T|})$. Note that $\cup_{a=1}^T M_{2a-1,z_{2a-1}}$ is the union of $T$ distinct singletons by~\cref{prop:setinc}, part (b). For $\beta_r\in \mH_T$, define 
	\begin{equation}\label{eq:kstar}\msk_{2t,\beta_r}:=\begin{cases} \cI^*_{2t,\beta_r} & \mbox{if } \beta_r\in \{i_1,\ldots ,i_p\}\\ 
			\cJ^*_{2T,\beta_r} & \mbox{if } \beta_r\in \{j_1,\ldots ,j_q\}\end{cases}\end{equation}
	for $t\in [T]$ (see the statement of~\cref{lem:supbound} for relevant definitions). Also let $\msk_{2t,U}:=\{M_{2a-1,z_{2a-1}}: a\in [T], \cU_{2a-1,z_{2a-1}}=\phi\}$ and $\msk_{2t,V}:=\{M_{2a-1,z_{2a-1}}: a\in [T], \cV_{2a-1,z_{2a-1}}=\phi\}$. By~\cref{prop:baseprop}, 
	\begin{align}\label{eq:centerprop}
		M_{2a-1,z_{2a-1}}\in \left(\cup_{r=1}^{|\mH_T|} \msk_{2t,\beta_r}\right)\cup \msk_{2t,U} \cup \msk_{2t,V}.
	\end{align}
	Let $\mM_T:=\{M_{2a-1,z_{2a-1}}: a\in [T]\}$. By using~\cref{lem:supbound}, we then have:
	\begin{align}\label{eq:degbound1}
		&\;\;\;\;\mbox{rank}\bigg(\sum_{(\ib,\jb)}R_{z_1,\ldots ,z_{2T}}(\ib,\jb)\bigg)\nonumber \\ &\leq \mbox{rank}\bigg(\sum_{(\ib,\jb)} \mathcal{R}[\bQ]_{N,1+|\msk_{2t,\beta_a}|}(\{\beta_a\}, \msk_{2t,\beta_a})\bigg)\bigg(\bQ^U_{N,|\msk_{2t,U}|}(\msk_{2t,U})\bigg)\bigg(\bQ^V_{N,|\msk_{2t,V}|}(\msk_{2t,V})\bigg)\bigg).
	\end{align}
	Note that the sets $\mH_T$ and $\mM_T$ need not be disjoint. Thus, define $\cC_T:=\mM_T\setminus \mH_T$. Recall the notation in \eqref{eq:obsdef}, \eqref{eq:obsdeg1}, and \eqref{eq:obsdeg2}. Using~\eqref{eq:centerprop}~and~\eqref{eq:degbound1}, we then get:
	\begin{align}\label{eq:degbound11}
		&\;\;\;\;\mbox{rank}\bigg(\sum_{(\ib,\jb)} R_{z_1,\ldots ,z_{2T}}(\ib,\jb)\bigg)\nonumber \\ &\leq \mbox{rank}\bigg(\sum_{(\ib,\jb)\setminus \cC_T} \bigg(\prod_{a=1}^{|\mH_T|} \mathcal{R}[\bQ]_{N,1+|\msk_{2t,\beta_a}|}(\{\beta_a\}, -(\cC_T\cap\msk_{2t,\beta_a}))\bigg)\bigg(\bQ^U_{N,|\msk_{2t,U}|}(-(\cC_T\cap \msk_{2t,U}))\bigg)\nonumber \\ &\qquad\qquad\qquad \bigg(\bQ^V_{N,|\msk_{2t,V}|}(-(\cC_T\cap\msk_{2t,V}))\bigg)\bigg).
	\end{align}
	
	Next define $\tcC_T:=\mM_T\cap \mH_T$, $\tau:=|\tcC_T|$ and enumerate $\tcC_T$ as $\{M_{2\ell_1,z_{2\ell_1-1}},\ldots , M_{2\ell_{\tau}-1,z_{2\ell_{\tau}-1}}\}$ where $\ell_1<\ell_2<\ldots <\ell_{\tau}$. Define $\cF_t:=\{M_{2\ell_{t}-1,z_{2\ell_{t}-1}},\ldots ,M_{2\ell_{\tau}-1,z_{2\ell_{\tau}-1}}\}$ for $t\leq \tau$. Then $\msk_{2t,M_{2\ell_{\tau}-1,z_{2\ell_{\tau}-1}}}\subseteq \cC_T$ by \cref{prop:setinc}, part (d). Moreover, by \eqref{eq:centerprop} and \cref{prop:setinc}, part (d), we have 
	$$M_{2\ell_{\tau}-1,z_{2\ell_{\tau}-1}}\in \left(\cup_{r=1, \beta_r\notin \cF_{\tau}}^{|\mH_T|} \msk_{2t,\beta_r}\right)\cup \msk_{2t,U} \cup \msk_{2t,V}.$$
	Consequently, observe that,
	\begin{align}\label{eq:degbound12}
		&\;\;\;\sum_{(\cup_{a=1}^{|\mH_T|} \{\beta_a\})}\bigg(\prod_{a=1}^{|\mH_T|} \mathcal{R}[\bQ]_{N,1+|\msk_{2t,\beta_a}|}(\{\beta_a\},-(\cC_T\cap\msk_{2t,\beta_a}))\bigg)\bigg(\bQ^U_{N,|\msk_{2t,U}|}(-(\cC_T\cap \msk_{2t,U}))\bigg)\nonumber \\ & \qquad \qquad \bigg(\bQ^V_{N,|\msk_{2t,V}|}(-(\cC_T\cap\msk_{2t,V}))\bigg)\nonumber \\ &\lesssim \sum_{(\cup_{a=1}^{|\mH_T|} \{\beta_a\})\setminus \cF_{\tau}} \bigg(\prod_{\substack{a\in [\mH_T]:\\ \ \beta_a\notin \cF_{\tau}}} \mathcal{R}[\bQ]_{N,1+|\msk_{2t,\beta_a}|}(\{\beta_a\},-((\cC_T\cup \cF_{\tau})\cap\msk_{2t,\beta_a}))\bigg)\bigg(\bQ^U_{N,|\msk_{2t,U}|}(-((\cC_T\cup \cF_{\tau})\cap \msk_{2t,U}))\bigg)\nonumber \\ &\qquad\qquad \bigg(\bQ^V_{N,|\msk_{2t,V}|}(-((\cC_T\cup \cF_{\tau})\cap \msk_{2t,V}))\bigg)\nonumber \\ &\qquad \qquad \bigg(\max_{M_{2\ell_{\tau}-1,z_{2\ell_{\tau}-1}}} \tQ_{N,1+|\msk_{2t,M_{2\ell_{\tau}-1,z_{2\ell_{\tau}-1}}}|}(\{M_{2\ell_{\tau}-1,z_{2\ell_{\tau}-1}}\},-\msk_{2t,M_{2\ell_\tau}-1,z_{2\ell_{\tau}-1}})\bigg)\nonumber\\ &\lesssim  \sum_{(\cup_{a=1}^{|\mH_T|} \{\beta_a\})\setminus \cF_{\tau}} \bigg(\prod_{\substack{a\in [\mH_T]:\\ \ \beta_a\notin \cF_{\tau}}} \mathcal{R}[\bQ]_{N,1+|\msk_{2t,\beta_a}|}(\{\beta_a\},-((\cC_T\cup \cF_{\tau})\cap\msk_{2t,\beta_a}))\bigg)\bigg(\bQ^U_{N,|\msk_{2t,U}|}(-((\cC_T\cup \cF_{\tau})\cap \msk_{2t,U}))\bigg)\nonumber \\ &\qquad \qquad \bigg(\bQ^V_{N,|\msk_{2t,V}|}(-((\cC_T\cup \cF_{\tau})\cap \msk_{2t,V}))\bigg).
	\end{align}
	where the last line follows from~\eqref{eq:obsdegbound}. Then $\msk_{2t,M_{2\ell_{\tau-1}-1,z_{2\ell_{\tau-1}-1}}}\subseteq \cC_T\cup \cF_{\tau}$ by \cref{prop:setinc}, part (d). Moreover, by \eqref{eq:centerprop} and \cref{prop:setinc}, part (d), we have 
	$$M_{2\ell_{\tau-1}-1,z_{2\ell_{\tau-1}-1}}\in \left(\cup_{r=1, \beta_r\notin \cF_{\tau-1}}^{|\mH_T|} \msk_{2t,\beta_r}\right)\cup \msk_{2t,U} \cup \msk_{2t,V}.$$
	Therefore, by repeating the same argument as above, we get
	\begin{align*}
		&\;\;\;\sum_{(\cup_{a=1}^{|\mH_T|} \{\beta_a\})}\bigg(\prod_{a=1}^{|\mH_T|} \mathcal{R}[\bQ]_{N,1+|\msk_{2t,\beta_a}|}(\{\beta_a\},-(\cC_T\cap\msk_{2t,\beta_a}))\bigg)\bigg(\bQ^U_{N,|\msk_{2t,U}|}(-(\cC_T\cap \msk_{2t,U}))\bigg)\nonumber \\ & \qquad \qquad \bigg(\bQ^V_{N,|\msk_{2t,V}|}(-(\cC_T\cap\msk_{2t,V}))\bigg)\\ &\lesssim  \sum_{(\cup_{a=1}^{|\mH_T|} \{\beta_a\})\setminus \cF_{\tau-1}} \bigg(\prod_{\substack{a\in [\mH_T]:\\ \ \beta_a\notin \cF_{\tau-1}}} \mathcal{R}[\bQ]_{N,1+|\msk_{2t,\beta_a}|}(\{\beta_a\},-((\cC_T\cup \cF_{\tau-1})\cap\msk_{2t,\beta_a}))\bigg)\nonumber \\ &\qquad \qquad \bigg(\bQ^U_{N,|\msk_{2t,U}|}(-((\cC_T\cup \cF_{\tau-1})\cap \msk_{2t,U}))\bigg)\bigg(\bQ^V_{N,|\msk_{2t,V}|}(-((\cC_T\cup \cF_{\tau-1})\cap \msk_{2t,V}))\bigg).
	\end{align*}
	which is the same as the right hand side of~\eqref{eq:degbound12} with $\tau$ replaced by $\tau-1$. Proceeding backwards as above, we can replace with $\tau=1$. Observe that $\cC_T\cup\cF_1=\mM_T$. Proceeding recursively as above and using~\eqref{eq:degbound11}, we then get:
	\begin{align*}
		&\;\;\;\;\;\mbox{rank}\bigg(\sum_{(\ib,\jb)} R_{z_1,\ldots ,z_{2T}}(\ib,\jb)\bigg)\\ & \le
		\mbox{rank}\bigg(\sum_{(\ib,\jb) \setminus \mM_T} \bigg(\prod_{a\in [\mH_T]: \beta_a\notin \cF_1}\mathcal{R}[\bQ]_{N,1+|\msk_{2t,\beta_a}|}(\{\beta_a\}, -(\mM_T\cap \msk_{2t,\beta_a}))\bigg)\bigg(\bQ^U_{N,|\msk_{2t,U}|}(-(\mM_T\cap \msk_{2t,U}))\bigg)\bigg) \\ &\qquad \qquad \bigg(\bQ^V_{N,|\msk_{2t,V}|}(-(\mM_T\cap \msk_{2t,V}))\bigg) \\ &\le \mbox{rank}\bigg(\sum_{(\ib,\jb)\setminus \mM_T} 1\bigg)=p+q-T.
	\end{align*}
	Here the last line follows from \eqref{eq:obsdegbound} and \eqref{eq:obsdegbound1}. This proves part (a).
	
	\emph{For part (b).} Note that for $a\in [T]$, the sets $\{(\cE_{2a-2}\setminus M_{2a-2,z_{2a-2}})\setminus \cE_{2a,z_{2a}}\}_{a\in [T]}$ are disjoint. Let us enumerate $\mA_T:=\cup_{a=1}^T ((\cE_{2a-2}\setminus M_{2a-2,z_{2a-2}})\setminus \cE_{2a,z_{2a}}))$ arbitrarily as  $\{\gamma_1,\ldots ,\gamma_{|\mA_T|}\}$. Using~\cref{lem:supbound}, part (c), we consequently get:
	\begin{equation}\label{eq:degbound3}
		\mbox{rank}\bigg(\sum_{(\ib,\jb)} R_{z_1,\ldots ,z_{2T}}(\ib,\jb)\bigg)\leq \mbox{rank}\bigg(\sum_{(\ib,\jb) \setminus (\gamma_1,\ldots ,\gamma_{|\mA_T|})} \sum_{(\gamma_1,\ldots ,\gamma_{|\mA_T|})} \bigg(\prod\limits_{a=1}^{|\mA_T|} \mathcal{R}[\bQ]_{N,1+|\cJ^*_{2T,\gamma_a}|}(\gamma_a , \cJ^*_{2T,\gamma_a})\bigg)\bigg).
	\end{equation}
	
	Recall that we had defined $\mM_T$ as $\{M_{1,k_1},\ldots , M_{2T-1,z_{2T-1}}\}$. Also, by definition of $\mA_T$, we have $\mM_T\cap \mA_T=\phi$. Consequently $\mA_T\cap \cJ^*_{2T,\gamma_a}=\phi$ for any $a\in [T]$. Also note that $\max_{\cJ^*_{2T,\gamma_a}}\sum_{\gamma_a} \mathcal{R}[\bQ]_{N,1+|\cJ^*_{2T,\gamma_a}|}(\gamma_a,\cJ^*_{2T,\gamma_a})\lesssim 1$ by~\cref{lem:smoothcont}, part (2). As $\gamma_a$'s are all distinct, we have:
	\begin{align*}
		\mbox{rank}\bigg(\sum_{(\ib,\jb)} R_{z_1,\ldots ,z_{2T}}(\ib,\jb)\bigg)\le \mbox{rank}\left(\sum_{(\ib,\jb) \setminus (\gamma_1,\ldots ,\gamma_{|\mA_T|})} 1\right) = p+1-|\mA_T|.
	\end{align*}
	This establishes (b).
	
	\subsection{Proof of~\cref{lem:negterm}}\label{sec:pfnegterm}
	The following inequality will be useful throughout this proof:
	\begin{equation}\label{eq:negterm1}
		\frac{k}{2}=\frac{1}{2}\left(q+\sum_{r=1}^p l_r\right)\geq p+\frac{q}{2}.
	\end{equation}
	\emph{Part (i).} Note that if $T>q/2$, then by~\cref{prop:degbound}, $\mathrm{rank}(R_{z_1,\ldots ,z_{2T}})\leq p+q-T<p+q/2\leq k/2$ (by~\eqref{eq:negterm1}).
	
	Next consider the case $T<q/2$. Recall $\cE_{0,z_0}\equiv (j_1,\ldots ,j_q)$ as in the proof of~\cref{prop:boundlen}. As $R_{z_1,\ldots ,z_{2T}}$ is a leaf node, we have $\cE_{2T,z_{2T}}=\phi$ (see step 17 of~\cref{alg:construct_tree}). We consequently get:
	\begin{align*}
		\Bigg|\cup_{a=1}^{T} (\cE_{2a-2,z_{2a-2}}\setminus M_{2a-2,z_{2a-2}})\setminus \cE_{2a,z_{2a}}\Bigg|& = \sum_{a=1}^{T} \Big|(\cE_{2a-2,z_{2a-2}}\setminus M_{2a-2,z_{2a-2}})\setminus \cE_{2a,z_{2a}}\Big|\\ &=\sum_{a=1}^T (|\cE_{2a-2,z_{2a-2}}|-|\cE_{2a,z_{2a}}|-1)=q-T>q/2.
	\end{align*}
	Using the above observation in~\cref{prop:degbound}, we get that $\mathrm{rank}(R_{z_1,\ldots ,z_{2T}})=p+q-\Big|\sum_{a=1}^{T} (\cE_{2a-2,z_{2a-2}}\setminus M_{2a-2,z_{2a-2}})\setminus \cE_{2a,z_{2a}}\Big|<p+q/2\leq k/2$ (see~\eqref{eq:negterm1}). This completes the proof of part (i).
	
	\emph{Part (ii).} Note that, if there exists $p_0\in [p]$ such that $l_{p_0}>2$, then a strict inequality holds in~\eqref{eq:negterm1}, i.e., $k/2>p+q/2$. If $T\neq q/2$, then the conclusion follows from part (i). If $T=q/2$, then by~\cref{prop:degbound}, $\mbox{rank}(R_{z_1,\ldots ,z_{2T}})\leq p+q/2<k/2$. This completes the proof.
	
	\emph{Part (iii).} Without loss of generality, we can restrict to the case $T=|\mA_T|=q/2$. Let $i_c\in \cD_{a_0,k_{a_0}}^c$. Recall the definitions of $\mA_T$, $\{\gamma_1,\ldots ,\gamma_{|\mA_T|}\}$ and $\mM_T$ from the proof of~\cref{prop:degbound}. As $\mA_T\cup \mM_T\subseteq \{j_1,\ldots ,j_q\}$, we therefore have $i_c\notin \mA_T\cup \mM_T$. Recall the definitions of $\cI^\star$ and $\cJ^\star$ from \cref{prop:degbound}, parts (b) and (c). Consequently, using~\cref{lem:supbound}, we get:
	
	\begin{align}\label{eq:negterm4}
		&\;\;\;\;\;\mbox{rank}(R_{z_1,\ldots ,z_{2T}})\nonumber \\ &\leq \mbox{rank}\bigg(\sum_{(\ib,\jb) \setminus (\gamma_1,\ldots ,\gamma_{|\mA_T|}, i_c)} \sum_{(\gamma_1,\ldots ,\gamma_{|\mA_T|},i_c)} \bigg(\prod\limits_{a=1}^{|\mA_T|} \mathcal{R}[\bQ]_{N,1+|\cJ^*_{2T,\gamma_a}|}(\gamma_a, \cJ^*_{2T,\gamma_a})\bigg)\mathcal{R}[\bQ]_{N,1+|\cI^*_{2T,i_c}|}(i_c, \cI^*_{2T,i_c})\bigg)\nonumber \\ &\leq \mbox{rank}\bigg(\sum_{(\ib,\jb) \setminus (\gamma_1,\ldots ,\gamma_{|\mA_T|}, i_c)}  1\bigg)\leq p+q/2-1<k/2,
	\end{align}
	where the last step follows by summing over the indices $(\gamma_1,\ldots ,\gamma_{|\mA_T|})$ first, followed by summing over the index $i_c$ and then using~\eqref{eq:cmean2}. This proves part (iii).
	
	\emph{Part (iv).} Recall that we had defined $\mH_T$ as $\cup_{a=1}^T (\bar{\cD}_{2a-1,z_{2a-1}}\cup \bar{\cE}_{2a-1,z_{2a-1}}^c)$. Assume that there exist $a_0$ such that $M_{2a_0-1,z_{2a_0-1}}\in \mH_T$. As $\mA_T\cap\mM_T=\phi$, we conclude that $M_{2a_0-1,z_{2a_0-1}}\notin \mA_T$. With this observation, the rest of the argument as same as in part (iii), and we leave the details to the reader.
	
	\emph{Part (v).} Suppose there exists $a_0$ such that $\cU_{2a_0-1,z_{2a_0-1}}=\phi$. Recall the definition of $\msk_{2t,U}=\{M_{2a-1,z_{2a-1}}: a\in [T], \cU_{2a-1,z_{2a-1}}=\phi\}$ from the proof of \cref{prop:degbound}. Then $M_{2a_0-1,2a_0-1}\in\msk_{2t,U}$. Recall the definition of $\cJ^\star$ from \cref{prop:degbound}, part (c). Consequently, using~\cref{lem:supbound}, we get:
	\begin{align}\label{eq:negterm4}
		&\;\;\;\;\;\mbox{rank}(R_{z_1,\ldots ,z_{2T}})\nonumber \\ &\leq \mbox{rank}\bigg(\sum_{\substack{(\ib,\jb) \setminus \\ (\gamma_1,\ldots ,\gamma_{|\mA_T|}, M_{2a_0-1,z_{2a_0-1}})}} \sum_{\substack{(\gamma_1,\ldots ,\gamma_{|\mA_T|},\\ M_{2a_0-1,z_{2a_0-1}})}} \bigg(\prod\limits_{a=1}^{|\mA_T|} \mathcal{R}[\bQ]_{N,1+|\cJ^*_{2T,\gamma_a}|}(\gamma_a, \cJ^*_{2T,\gamma_a})\bigg)\bQ^U_{N,|\msk_{2t,U}|}(\msk_{2t,U})\bigg)\nonumber \\ &\leq \mbox{rank}\bigg(\sum_{(\ib,\jb) \setminus (\gamma_1,\ldots ,\gamma_{|\mA_T|}, M_{2a_0-1,z_{2a_0-1}})}  1\bigg)\leq p+q/2-1<k/2,
	\end{align}
	where the last step follows by summing over the indices $(\gamma_1,\ldots ,\gamma_{|\mA_T|})$ first, followed by summing over the index $M_{2a_0-1,z_{2a_0-1}}$ and then using \cref{lem:smoothcont}, part 2 and \cref{lem:supbound}, part (d). This proves part (iii).
	
	\emph{Part (vi).} The proof is the same as that of part (v). So we skip the details for brevity.
	
	\emph{Part (vii).} Without loss of generality, we restrict to the case $T=|\mA_T|=q/2$ and $M_{2a-1,z_{2a-1}}\notin \mH_T$ for any $a\in [T]$. It therefore suffices to show that $\mbox{rank}(R_{z_1,\ldots ,z_{2T}})<k/2$ if there exist $j_{\beta}$ such that
	\begin{equation}\label{eq:degg}
		j_{\beta}\in \left(\bar{\cE}_{2a_0-1,z_{2a_0-1}}\cap \left(\{j_1,\ldots ,j_q\}\setminus \cE_{2a_0-2,z_{2a_0-2}}\right)\right)\setminus \mM_T.\end{equation}
	By~\eqref{eq:degg}, $M_{2a_0-1,z_{2a_0-1}}\in \cJ^*_{2T,j_{\beta}}$. Further, as $j_{\beta}\notin \cE_{2a_0-2,z_{2a_0-2}}$, by~\cref{prop:setinc}, part (e), there exists $a_1<a_0$ such that
	\begin{equation}\label{eq:degclaim} \{M_{2a_1-1,z_{2a_1-1}},M_{2a_0-1,z_{2a_0-1}}\}\subseteq \cJ^*_{2T,j_{\beta}}.\end{equation}
	We split the rest of the proof into two cases:
	
	\emph{Case 1 - $j_{\beta}\notin \mA_T$:} By applying~\cref{lem:supbound}, we get: 
	
	\begin{align}\label{eq:negterm3}
		&\;\;\;\;\;\mbox{rank}(R_{z_1,\ldots ,z_{2T}})\nonumber \\ &\leq \mbox{rank}\bigg(\sum_{\substack{(\ib,\jb) \setminus\\ (\gamma_1,\ldots ,\gamma_{|\mA_T|}, j_{\beta})}} \sum_{(\gamma_1,\ldots ,\gamma_{|\mA_T|},j_{\beta})} \bigg(\prod\limits_{a=1}^{|\mA_T|} \mathcal{R}[\bQ]_{N,1+|\cJ^*_{2T,\gamma_a}|}(\gamma_a,\cJ^*_{2T,\gamma_a})\bigg)\mathcal{R}[\bQ]_{N,1+|\cJ^*_{2T,j_{\beta}}|}(j_{\beta}, \cJ^*_{2T,j_{\beta}})\bigg)\nonumber \\ &\leq \mbox{rank}\bigg(\sum_{(\ib,\jb) \setminus (\gamma_1,\ldots ,\gamma_{|\mA_T|}, j_{\beta})}  1\bigg)\leq p+q/2-1<k/2,	\end{align}
	where the last line follows by first summing over $(\gamma_1,\ldots ,\gamma_{|\mA_T|})$ followed by $j_{\beta}$. This works because $j_{\beta}\notin \mA_T$ and $\mA_T\cap \mM_T=\phi$. We can consequently sum over the indices in $\mA_T$ keeping $j_{\beta}$ fixed. Finally, as $\cJ^*_{2T,j_{\beta}}\neq \phi$ (by~\eqref{eq:degclaim}), we have $\max_{\cJ^*_{2T,j_{\beta}}}\sum_{j_{\beta}} \mathcal{R}[\bQ]_{N,1+|\cJ^*_{2T,j_{\beta}|}}(\{j_{\beta}\}\cup \cJ^*_{2T,j_{\beta}})\lesssim 1$, by \cref{lem:smoothcont} part 2. This establishes~\eqref{eq:negterm3}.
	
	\emph{Case 2 - $j_{\beta}=\gamma_c$ for some $c\leq |\mA_T|$:} Once again, by applying~\cref{lem:supbound}, we get
	\begin{align}\label{eq:negtermnew}
		&\;\;\;\;\;\mbox{rank}(R_{z_1,\ldots ,z_{2T}})\nonumber \\ &\leq \mbox{rank}\bigg(\sum_{\substack{(\ib,\jb) \setminus\\ ((\gamma_1,\ldots ,\gamma_{c-1},\gamma_{c+1},\ldots ,\gamma_{|\mA_T|})\\ \cup  \{M_{2a_1-1,z_{2a_1-1}},M_{2a_0-1,z_{2a_0-1}}\})}} \sum_{\substack{((\gamma_1,\ldots ,\gamma_{c-1},\gamma_{c+1},\ldots ,\gamma_{|\mA_T|})\\ \cup  \{M_{2a_1-1,z_{2a_1-1}},M_{2a_0-1,z_{2a_0-1}}\})}} \mathcal{R}[\bQ]_{N,1+|\cJ^*_{2T,j_{\beta}}|}(j_{\beta}, \cJ^*_{2T,\gamma_c})\bigg)\nonumber \\ &\;\;\;\;\;\bigg(\prod\limits_{a=1,\ a\neq c}^{|\mA_T|} \mathcal{R}[\bQ]_{N,1+|\cJ^*_{2T,\gamma_a}|}(\gamma_a, \cJ^*_{2T,\gamma_a})\bigg)\nonumber \\ &\overset{(a)}{\leq} \mbox{rank}\bigg(\sum_{\substack{(\ib,\jb) \setminus\\ ((\gamma_1,\ldots ,\gamma_{c-1},\gamma_{c+1},\ldots ,\gamma_{|\mA_T|})\\ \cup  \{M_{2a_1-1,z_{2a_1-1}},M_{2a_0-1,z_{2a_0-1}}\})}} \sum_{  \{M_{2a_1-1,z_{2a_1-1}},M_{2a_0-1,z_{2a_0-1}}\}} \mathcal{R}[\bQ]_{N,1+|\cJ^*_{2T,j_{\beta}}|}(j_{\beta}, \cJ^*_{2T,j_{\beta}})\bigg)\nonumber \\ & \overset{(b)}{\leq} \mbox{rank}\bigg(\sum_{\substack{(\ib,\jb) \setminus\\ ((\gamma_1,\ldots ,\gamma_{c-1},\gamma_{c+1},\ldots ,\gamma_{|\mA_T|})\\ \cup  \{M_{2a_1-1,z_{2a_1-1}},M_{2a_0-1,z_{2a_0-1}}\})}} 1\bigg)\leq p+q/2-1<k/2.	
	\end{align}
	Here (a) follows from the fact that $\mA_T\cap \mM_T=\phi$, which implies that we can sum up over $(\gamma_1,\ldots ,\gamma_{c-1},\gamma_{c+1},\ldots ,\gamma_{|\mA_T|})$ keeping $M_{2a_1-1,z_{2a_1-1}}, M_{2a_0-1,z_{2a_0-1}}$ fixed. Finally, (b) follows from \eqref{eq:degclaim}. This completes the proof of part (v).
	
	\emph{Parts (viii) and (ix).} Without loss of generality, we can restrict to the case $T=|\mA_T|=q/2$, $\cup_{a=1}^T \bar{\cD}_{2a-1,z_{2a-1}}=\phi$, $\mM_T\cap \mH_T=\phi$, $\cU_{2a-1,z_{2a-1}}\neq \phi$, $\cV_{2a-1,z_{2a-1}}\neq \phi$ for all $a\in [T]$, and 
	$$\cup_{a=1}^T \left(\bar{\cE}_{2a-1,z_{2a-1}} \cap \left(\{j_1,\ldots ,j_q\}\setminus \cE_{2a-2,z_{2a-2}}\right)\right)=\phi,$$
	from parts (i), (iii), (iv), (vii) above. By the above display, we observe that \begin{equation}\label{eq:dgbd1}
		\bar{\cE}_{2a-1,z_{2a-1}}=\bar{\cE}_{2a-1,z_{2a-1}}\cap \cE_{2a-2,z_{2a-2}}.
	\end{equation}
	We next claim that, for any $a\in [T]$, the following holds:
	\begin{equation}\label{eq:dgbd2}
		\big|\big(\cE_{2a-2,z_{2a-2}}\setminus M_{2a-2,z_{2a-2}}\big)\setminus \cE_{2a,z_{2a}}\big|\geq 1.
	\end{equation}
	First let us complete the proof assuming~\eqref{eq:dgbd2}. Observe that
	$$q/2=|\mA_T|=\sum_{a=1}^T \big|\big(\cE_{2a-2,z_{2a-2}}\setminus M_{2a-2,z_{2a-2}}\big)\setminus \cE_{2a,z_{2a}}\big|\geq \sum_{a=1}^T 1=q/2.$$
	Therefore, equality holds throughout the above display and so $\big|\big(\cE_{2a-2,z_{2a-2}}\setminus M_{2a-2,z_{2a-2}}\big)\setminus \cE_{2a,z_{2a}}\big|=1$. As $\bar{\cD}_{2a-1,z_{2a-1}}=\phi$, $\cU_{2a-1,z_{2a-1}}\neq \phi$, $\cV_{2a-1,z_{2a-1}}=\phi$, we must have $|\cE_{2a-1,z_{2a-1}}|\geq 1$, by \cref{prop:baseprop}, part (h). By~\cref{prop:setinc}, part (f), we have:
	\begin{align}\label{eq:dgbd3}
		1&=\big|\big(\cE_{2a-2,z_{2a-2}}\setminus M_{2a-2,z_{2a-2}}\big)\setminus \cE_{2a,z_{2a}}\big|\nonumber \\ &\geq \big|(\bar{\cE}_{2a-1,z_{2a-1}}\cap \cE_{2a-2,z_{2a-2}})\big| + \big|((\cE_{2a-2,z_{2a-2}}\cap \cE_{2a-1,z_{2a-1}})\setminus \cE_{2a,z_{2a}})\big|\nonumber \\ &\overset{\dagger}{=} \big|\bar{\cE}_{2a-1,z_{2a-1}}\big| + \big|((\cE_{2a-2,z_{2a-2}}\cap \cE_{2a-1,z_{2a-1}})\setminus \cE_{2a,z_{2a}})\big|\geq 1
	\end{align}
	where $\dagger$ follows from~\eqref{eq:dgbd1}. Once again, we must have equality throughout~\eqref{eq:dgbd3}. The equality condition immediately completes the proof.
	
	\emph{Proof of~\eqref{eq:dgbd2}} Suppose that~\eqref{eq:dgbd2} does not hold. By~\cref{prop:setinc}, part (c), this would imply $\cE_{2a,z_{2a}}=(\cE_{2a-2,z_{2a-2}}\setminus M_{2a-2,z_{2a-2}})$. By a similar computation as in~\eqref{eq:dgbd3} would imply $\bar{\cE}_{2a-1,z_{2a-1}}=\phi$, which coupled with $\bar{\cD}_{2a-1,z_{2a-1}}=\phi$, $\cU_{2a-1,z_{2a-1}}\neq \phi$, and $\cV_{2a-1,z_{2a-1}}\neq\phi$, yields a contradiction to \cref{prop:setinc}, part (f), and proves~\eqref{eq:dgbd2}.
	
	\subsection{Proof of \cref{lem:initremove}}\label{sec:pfinitremove}
	We will use the shorthands $\abk,\bbk,\mbk,\obk$ for the index sets $(a_1,\ldots ,a_k)\in [N]^{k}$, $(b_1,\ldots ,b_{k_1})\in [N]^{k_1}$, $(m_1,\ldots ,m_{k_2})\in [N]^{k_2}$, and $(o_1,\ldots ,o_{k_2})\in [N]^{k_2}$. Note that 
	\begin{align*}
		&\;\;\;\;\E T_N^k U_N^{k_1} V_N^{k_2} \\ &=\frac{1}{N^{\frac{k}{2}+k_1+k_2}}\sum_{\abk,\bbk,\mbk,\obk} \E\prod_{r=1}^k (c_{a_r}(g(\si_{a_r})-t_{a_r})) \prod_{r=1}^{k_1}(c_{b_r}^2(g(\si_{b_r})^2-t_{b_r}^2)\prod_{r=1}^{k_2} c_{m_r} c_{o_r}(g(\si_{m_r})-t_{m_r})(t_{o_r}^{m_r}-t_{o_r})
	\end{align*}
	The crux of the statement of \cref{lem:initremove} is to show that the contribution of the summands above, when either of the index sets $\bbk,\mbk,\obk$, overlap with $\abk$, are negligible as $N\to\infty$. To see how, we will first replace each of the unrestricted sums across indices $b_j$ with a sum over $b_j\neq a_1,\ldots ,a_k$. Let us do this inductively. Define $N_{\abk}=[N]\setminus \abk$. Suppose we have already replaced the unrestricted sum over $(b_1,\ldots ,b_{s-1})\in [N]^{s-1}$ with sum over $(b_1,\ldots ,b_{s-1})\in N_{\abk}^{s-1}$, $1\le s\le k_1$. Consider the case where $b_s=a_1$. Let us write $\bbk_{s-1}=(b_1,\ldots , b_{s-1})$, $\bbk_{-s}=(b_{s+1},\ldots ,b_{k_1})$, and $\abk_{-1}=(a_2,\ldots ,a_k)$. The corresponding summands are given by 
	\begin{align*}
		&\;\;\;\;\frac{1}{N^{\frac{k}{2}+k_1+k_2}}\E\sum_{\substack{\abk,\bbk_{s-1}\in N_{\abk}^{s-1},\\ \bbk_{-s},\mbk,\obk}}c_{a_1}^3(g(\si_{a_1})-t_{a_1})(g(\si_{a_1})^2-t_{a_1}^2)\prod_{r=2}^k (c_{a_r}(g(\si_{a_r})-t_{a_r}))\prod_{r=1, r\neq s}^{k_1}c_{b_r}^2(g(\si_{b_r})^2-t_{b_r}^2)\\ &\qquad \qquad \prod_{r=1}^{k_2} c_{m_r} c_{o_r}(g(\si_{m_r})-t_{m_r})(t_{o_r}^{m_r}-t_{o_r}) \\ &=\frac{1}{N^{\frac{k}{2}+k_1+k_2}}\E\sum_{\bbk_{s-1},\bbk_{-s}} \left(\sum_{a_1\in [N]\setminus \bbk_{s-1}} c_{a_1}^3(g(\si_{a_1})-t_{a_1})(g(\si_{a_1})^2-t_{a_1}^2)\right) \left(\sum_{\abk_{-1}\in ([N]\setminus \bbk_{s-1})^{k-1}} \prod_{r=2}^k c_{a_r}(g(\si_{a_r})-t_{a_r})\right) \\ &\qquad \qquad \prod_{r=1,r\neq s}^{k_1} c_{b_r}^2(g(\si_{b_r})^2-t_{b_r}^2)\sum_{\mbk,\obk} \prod_{r=1}^{k_2} c_{m_r} c_{o_r}(g(\si_{m_r})-t_{m_r})(t_{o_r}^{m_r}-t_{o_r}) \\ &\overset{(i)}{\lesssim} \frac{1}{N^{\frac{1}{2}+k_1+k_2}}\sum_{\bbk_{s-1},\bbk_{-s}}\left(\sum_{a_1\in [N]\setminus \bbk_{s-1}} 1\right)\E\bigg|\frac{1}{\sqrt{N}}\sum_{a\notin \bbk_{s-1}} c_a(g(\si_a)-t_a)\bigg|^{k-1}\prod_{r=1,r\neq s}^{k_1} c_{b_r}^2(g(\si_{b_r})^2-t_{b_r}^2) \\ &\qquad\qquad \bigg|\sum_{\mbk,\obk} \prod_{r=1}^{k_2}\Q_{N,2}(o_r,b_r)\bigg| \\ &\overset{(ii)}{\lesssim} \frac{N^{k_1+k_2}}{N^{1/2+k_1+k_2}}= O(N^{-1/2}).\end{align*}
	Here (i) follows from \cref{as:coeffalt}, \eqref{eq:cmean1}, and the fact that $g(\cdot)$ is bounded. Next, (ii) follows from \eqref{eq:cmean2} and \cref{lem:auxtail}, part (a). Therefore, the contribution of the terms where the indices $\bbk$ overlap non-trivially with $\abk$, are all negligible. 
	
	Let us now show that the contribution of the terms when either of the vectors $\mbk,\obk$ overlaps with $\abk$, is again negligible. For notational simplicity, we will only show that the contributions when either $m_1=a_1$ or $o_1=a_1$ are negligible. We can now assume that $\bbk$ does not overlap with $\abk$. First, let us assume that $m_1=a_1$ and $o_1,(m_2,o_2),\ldots ,(m_{k_2},o_{k_2})$ are unrestricted. Let us write $\mbk_{-1}=(m_2,\ldots ,m_{k_2})$ and $\obk_{-1}=(o_2,\ldots ,o_{k_2})$. The corresponding summands are given by 
	\begin{align*}
		&\;\;\;\;\frac{1}{N^{\frac{k}{2}+k_1+k_2}}\E \sum_{\substack{\abk,\bbk\in N_{\abk}^k, \\ \mbk_{-1}, \obk}} c_{a_1}^2 c_{o_1}(g(\si_{a_1})-t_{a_1})^2(t_{o_1}^{a_1}-t_{o_1}) \prod_{r=2}^k c_{a_r}(g(\si_{a_r})-t_{a_r})\prod_{r=1}^{k_1} c_{b_r}^2 (g(\si_{b_r})^2-t_{b_r}^2) \\ &\qquad\qquad \prod_{r=2}^{k_2} c_{m_r} c_{o_r}(g(\si_{m_r})-t_{m_r})(t_{o_r}^{m_r}-t_{o_r}) \\ &=\frac{1}{N^{\frac{k}{2}+k_1+k_2}}\E\sum_{\bbk}\left(\sum_{a_1\in [N]\setminus \bbk, o_1} c_{a_1}^2 c_{o_1}(g(\si_{a_1})-t_{a_1})^2(t_{o_1}^{a_1}-t_{o_1})\right)\left(\sum_{\abk_{-1}\in ([N]\setminus \bbk)^{k-1}} \prod_{r=2}^{k} c_{a_r}(g(\si_{a_r})-t_{a_r})\right) \\ &\qquad \qquad \prod_{r=1}^{k_1} c_{b_r}^2 (g(\si_{b_r})^2-t_{b_r}^2) \sum_{\mbk_{-1},\obk_{-1}}\prod_{r=2}^{k_2} c_{m_r} c_{o_r}(g(\si_{m_r})-t_{m_r})(t_{o_r}^{m_r}-t_{o_r}) \\ &\overset{(iii)}{\lesssim} \frac{1}{N^{\frac{1}{2}+k_1+k_2}}\left(\sum_{a_1,o_1}\Q_{N,2}(o_1,a_1)\right)\E\bigg|\frac{1}{\sqrt{N}}\sum_{a\notin [N]\setminus \bbk} c_a(g(\si_a)-t_a)\bigg|^{k-1} \prod_{r=1}^{k_1} \big|c_{b_r}^2 (g(\si_{b_r})^2-t_{b_r}^2)\big| \\ &\qquad \qquad \prod_{r=2}^{k_2}\left(\sum_{m_r,o_r} \Q_{N,2}(o_r,m_r)\right) \\ &\overset{(iv)}{\lesssim} \frac{N^{k_1+k_2}}{N^{\frac{1}{2}+k_1+k_2}}=O(N^{-1/2}).
	\end{align*}
	Here (iii) follows from \cref{as:coeffalt}, \eqref{eq:cmean1}, and the fact that $g(\cdot)$ is bounded. Also (iv) follows from \eqref{eq:cmean2} and \cref{lem:fixsol}, part (a). This implies the contribution when $m_1\in \abk$ is negligible. Next, we assume $m_1\notin \abk$ and $o_1=a_1$, while $\mbk_{-1},\obk_{-1}$ are all unrestricted. The corresponding summands are given by 
	\begin{align*}
		&\;\;\;\;\frac{1}{N^{\frac{k}{2}+k_1+k_2}}\E \sum_{\substack{\abk,\bbk\in N_{\abk}^k, \\ m_1\notin \abk, \mbk_{-1}, \obk_{-1}}} c_{a_1}^2 c_{m_1}(g(\si_{a_1})-t_{a_1})(g(\si_{m_1})-t_{m_1})(t_{a_1}^{m_1}-t_{a_1}) \prod_{r=2}^k c_{a_r}(g(\si_{a_r})-t_{a_r}) \\ &\qquad\qquad \prod_{r=1}^{k_1} c_{b_r}^2 (g(\si_{b_r})^2-t_{b_r}^2)\prod_{r=2}^{k_2} c_{m_r} c_{o_r}(g(\si_{m_r})-t_{m_r})(t_{o_r}^{m_r}-t_{o_r}) \\ &=\frac{1}{N^{\frac{k}{2}+k_1+k_2}}\E\sum_{\bbk,m_1}\left(\sum_{a_1\in [N]\setminus \bbk, m_1\neq a_1} c_{a_1}^2 c_{m_1}(g(\si_{a_1})-t_{a_1})(g(\si_{m_1})-t_{m_1})(t_{a_1}^{m_1}-t_{a_1})\right) \\ &\qquad \qquad  \left(\sum_{\abk_{-1}\in ([N]\setminus (\bbk,m_1))^{k-1}} \prod_{r=2}^{k} c_{a_r}(g(\si_{a_r})-t_{a_r})\right) \prod_{r=1}^{k_1} c_{b_r}^2 (g(\si_{b_r})^2-t_{b_r}^2) \\ &\qquad\qquad\sum_{\mbk_{-1},\obk_{-1}}\prod_{r=2}^{k_2} c_{m_r} c_{o_r}(g(\si_{m_r})-t_{m_r})(t_{o_r}^{m_r}-t_{o_r}) \\ &\overset{(v)}{\lesssim} \frac{1}{N^{\frac{1}{2}+k_1+k_2}}\left(\sum_{a_1,o_1}\Q_{N,2}(a_1,m_1)\right)\E\bigg|\frac{1}{\sqrt{N}}\sum_{a\notin [N]\setminus (\bbk,m_1)} c_a(g(\si_a)-t_a)\bigg|^{k-1} \prod_{r=1}^{k_1} \big|c_{b_r}^2 (g(\si_{b_r})^2-t_{b_r}^2)\big| \\ &\qquad \qquad \prod_{r=2}^{k_2}\left(\sum_{m_r,o_r} \Q_{N,2}(o_r,m_r)\right) \\ &\overset{(vi)}{\lesssim} \frac{N^{k_1+k_2}}{N^{\frac{1}{2}+k_1+k_2}}=O(N^{-1/2}).
	\end{align*}
	As before, (v) follows from \cref{as:coeffalt}, \eqref{eq:cmean1}, and the fact that $g(\cdot)$ is bounded. Also (vi) follows from \eqref{eq:cmean2} and \cref{lem:auxtail}, part (a). This completes the proof.
	\section{Proofs of Applications}\label{sec:pfappl}
	This Section is devoted to the proofs of results from Sections \ref{sec:ising}, \ref{sec:highising}, and \ref{sec:ergm}.
	
	\subsection{Proofs from \cref{sec:ising}}
	\begin{proof}[Proof of \cref{theo:conmainis}]
		Recall the definitions of $U_N$ and $V_N$ from \eqref{eq:randvar}. By an application of \cref{theo:CLTmain}, the proof of \cref{theo:conmainis} will follow once we establish \cref{as:cmean}.
		To wit, recall the definition $m_i=\sum_{j=1,j\neq i}^N \A_N(i,j)\si_j$ from \eqref{eq:avisdef}. Therefore $m_i^k=\sum_{j=1,j\neq i,k}^N \A_N(i,j)\si_j$. By \cref{as:bddrowsum}, we note that 
		$$\max_{1\le i\le N}\sum_{k=1}^N |m_i-m_i^k|\le \max_{1\le i\le N}\sum_{k=1}^N \A_N(i,k)\lesssim 1, \qquad \qquad \max_{1\le k\le N}\sum_{i=1}^N |m_i-m_i^k|\le \max_{1\le k\le N}\sum_{i=1}^N \A_N(i,k)\lesssim 1.$$
		Recall the definition of $\Xi$ from \eqref{eq:candef1}. As $\Xi'$ has uniformly bounded derivatives of all orders, an application of \cref{lem:smoothcont} the establishes \cref{as:cmean}.
	\end{proof}
	
	In order to prove the remaining results from \cref{sec:ising}, it will be useful to consider the following corollary of \cref{theo:conmainis}.
	
	We present a corollary to \cref{theo:conmainis} that helps simplify $U_N+V_N$ (see \eqref{eq:randvar} with $g(x)=x$) under model \eqref{eq:model} when $\A_N$ satisfies the Mean-Field condition \cref{as:mfield}. This will be helpful in proving all the results in \cref{sec:ising}.
	\begin{corollary}\label{cor:isingmean}
		Consider the same assumptions as in \cref{theo:conmainis}. In addition suppose that \cref{as:mfield} holds. Define 
		\begin{align}\label{eq:varterm}
			v_N^2:=\left(\frac{1}{N}\sum_{i=1}^N c_i^2 \Xi''(\beta m_i+B)-\frac{\beta}{N}\sum_{i\neq j} c_i c_j \A_N(i,j)\Xi''(\beta m_i+B)\Xi''(\beta m_j+B)\right)\vee a_N,
		\end{align}
		for a strictly positive sequence $a_N\to 0$. 
		Then the following holds: 
		$$\frac{1}{v_N}\sum_{i=1}^N c_i(\si_i-\Xi'(\beta m_i+B)) \overset{w}{\longrightarrow} N(0,1).$$
		for any strictly positive sequence $a_N\to 0$. 
	\end{corollary}
	
	\begin{proof}
		By \cref{theo:conmainis} and \eqref{eq:bddzero}, it suffices to show that 
		\begin{align}\label{eq:uvradiate}
			U_N-\frac{1}{N}\sum_{i=1}^N c_i^2\Xi''(\beta m_i+B) \overset{\P}{\longrightarrow}0, \quad \mbox{and}\quad  V_N+\frac{\beta}{N}\sum_{i\neq j}c_ic_j\A_N(i,j)\Xi''(\beta m_i+B)\Xi''(\beta m_j+B)\overset{\P}{\longrightarrow} 0.
		\end{align}
		By \eqref{eq:cbd6} and \eqref{eq:cbd7}, we can assume without loss of generality $\mc$ satisfies \cref{as:coeffalt}. 
		
		Let us begin with the first display of \eqref{eq:uvradiate}. Note that $\E[\si_i^2|\si_j,j\neq i]=\Xi''(\beta m_i+B)+(\Xi'(\beta m_i+B))^2$. Therefore, by \cref{lem:auxtail}, part (a), we have 
		$$U_N-\frac{1}{N}\sum_{i=1}^N c_i^2 \Xi''(\beta m_i+B) = \frac{1}{N}\sum_{i=1}^N c_i^2 (\si_i^2-\E[\si_i^2|\si_j,j\neq i]) \overset{\P}{\longrightarrow} 0.$$
		We move on to the second display of \eqref{eq:uvradiate}. Direct calculation shows that 
		\begin{align*}
			V_N&=-\frac{\beta}{N}\sum_{i\neq j} c_i c_j \A_N(i,j)(\si_i^2-\si_i\Xi'(\beta m_i+B))\Xi''(\beta m_j^i+B) + \frac{1}{N}\lVert \A_N\rVert_F^2.
		\end{align*}
		As a result, we have:
		\begin{align}\label{eq:two1}
			&\;\;\;\;V_N+\frac{\beta}{N}\sum_{i\neq j}c_ic_j\A_N(i,j)\Xi''(\beta m_i+B)\Xi''(\beta m_j+B) \nonumber \\ &=-\underbrace{\frac{\beta}{N}\sum_{i\neq j} c_i c_j \A_N(i,j)(\si_i^2-\si_i\Xi'(\beta m_i+B)-\Xi''(\beta m_i+B))\Xi''(\beta m_j^i+B)}_{R_N}+o(1),
		\end{align}
		where the last display uses \cref{as:mfield}. Note that $\E R_N=0$. Further 
		\begin{align*}
			|\E R_N^2| &=\frac{\beta^2}{N^2}\bigg|\sum_{j_1\neq i_1, j_2\neq i_2} c_{i_1}c_{i_2}c_{j_1}c_{j_2}\A_N(i_1,j_1)\A_N(i_2,j_2)\E\bigg[(\si_{i_1}^2-\si_{i_1}\Xi'(\beta m_{i_1}+B)-\Xi''(\beta m_{i_1}+B)) \\ & \qquad (\si_{i_2}^2-\si_{i_2}\Xi'(\beta m_{i_2}+B)-\Xi''(\beta m_{i_2}+B))\Xi''(\beta m_{j_1}^{i_1}+B)\Xi''(\beta m_{j_2}^{i_2}+B)\bigg]\bigg| \\ &\lesssim N^{-1}+\frac{\beta^2}{N^2}\bigg|\sum_{j_1\neq i_1,j_2\neq i_2,i_1\neq i_2}\A_N(i_1,j_1)\A_N(i_2,j_2)\E\bigg[(\si_{i_1}^2-\si_{i_1}\Xi'(\beta m_{i_1}+B)-\Xi''(\beta m_{i_1}+B))\\ & \qquad (\si_{i_2}^2-\si_{i_2}\Xi'(\beta m_{i_2}^{i_1}+B)-\Xi''(\beta m_{i_2}^{i_1}+B))\Xi''(\beta m_{j_1}^{i_1}+B)\Xi''(\beta m_{j_2}^{i_1,i_2}+B)\bigg]\bigg|.
		\end{align*}
		The inner expectation in the above display equals $0$. The last inequality uses \cref{as:bddrowsum}. Therefore $R_N\overset{\P}{\longrightarrow} 0$. This completes the proof.
	\end{proof}
	\begin{proof}[Proof of \cref{thm:jointCLT}]
		Recall the notation $\mca_f$ and $\mcb_f$ from \eqref{eq:amat} and \eqref{eq:bmat} respectively. We begin with the following observation from \cite[Corollary 1.5]{bhattacharya2023gibbs} --- 
		\begin{align}\label{eq:check}
			\frac{1}{N}\sum_{i=1}^N \left(m_i-\fs\left(\frac{i}{N}\right)\right)^2\overset{\P}{\longrightarrow}0.
		\end{align}
		Define the following functions in $(\beta,B)$ --- 
		$$g_i(\beta,B):=\begin{pmatrix} m_i(\si_i-\Xi'(\beta m_i+B)) \\  \si_i-\Xi'(\beta m_i+B)\end{pmatrix}.$$
	\end{proof}
	By \eqref{eq:check}, we note that 
	$$-\frac{1}{N}\sum_{i=1}^N \nabla_{(\beta,B)} g_i(\beta,B)\overset{\P}{\longrightarrow} \begin{pmatrix} \int_0^1 \fs^2(x)\Xi''(\beta \fs(x)+B)\,dx & \int_0^1 \fs(x)\Xi''(\beta \fs(x)+B)\,dx \\ \int_0^1 \fs(x)\Xi''(\beta \fs(x)+B)\,dx & \int_0^1 \Xi''(\beta \fs(x)+B)\,dx \end{pmatrix}=\mca_{\fs}.$$
	By \cite[Theorem 1.11]{ghosal2018joint}, $\sqrt{N}(\hbem-\beta,\hbm-B)=O_{\P}(1)$ and $\mca_{\fs}$ is invertible. 
	
	Next we will derive the weak limit of $N^{-1/2}\sum_{i=1}^N g_i(\beta,B)$. We proceed using the Cram\'{e}r-Wold device. First note that by \eqref{eq:check} and \cref{lem:auxtail}, part (b), we have that for each $a,b\in\R$, the following holds: 
	\begin{align}\label{eq:probequivalence}
		\begin{pmatrix} a \\ b \end{pmatrix}^{\top} N^{-1/2}\sum_{i=1}^N g_i(\beta,B)=\frac{1}{\sqrt{N}}\sum_{i=1}^N \left(a \fs\left(\frac{i}{N}\right)+b\right)(\si_i-\Xi'(\beta m_i+B))+o_{\P}(1).
	\end{align}
	Using \eqref{eq:probequivalence}, we need to derive a CLT for $N^{-1/2}\sum_{i=1}^N (a \fs(i/N)+b)(\si_i-\Xi'(\beta m_i+B))$. We will use \cref{theo:conmain} for this. To achieve this, we need to identify the limit of  $U_N+V_N$ where $(U_N,V_N)$ are defined as in \eqref{eq:randvar} with $g(x)=x$ and $c_i=a \fs(i/N)+b$. By \eqref{eq:uvradiate}, we have 
	$$\begin{pmatrix} U_N-\frac{1}{N}\sum_{i=1}^N c_i^2 \Xi''(\beta m_i+B) \\ V_N+\frac{\beta}{N}\sum_{i\neq j} c_i c_j \A_N(i,j)\Xi''(\beta m_i+B)\Xi''(\beta m_j+B)\end{pmatrix} \overset{\P}{\longrightarrow} \begin{pmatrix} 0 \\ 0 \end{pmatrix}.$$
	Also by \eqref{eq:check}, we have that 
	\begin{align*}
		&\;\;\;\frac{1}{N}\sum_{i=1}^N c_i^2 \Xi''(\beta m_i+B)-\frac{\beta}{N}\sum_{i\neq j} c_i c_j\A_N(i,j)\Xi''(\beta m_i+B)\Xi''(\beta m_j+B) \\ &\overset{\P}{\longrightarrow} a^2 \mcb_{\fs}(1,1)+b^2 \mcb_{\fs}(2,2) + 2ab \mcb_{\fs}(1,2).
	\end{align*}
	We refer the reader to \eqref{eq:bmat} for relevant definitions. Combining te above displays with \cref{theo:conmain}, we get: 
	$$N^{-1/2}\sum_{i=1}^N g_i(\beta,B) \overset{w}{\longrightarrow} N(\mathbf{0}_2,\mcb_{\fs}).$$
	The conclusion now follows from \cref{prop:CLTMPLE}. To apply the result, we note that (A1) and (A3) follow from \cite[Theorem 1.11]{ghosal2018joint}, and (A2) has been proved above.
	\begin{proof}[Proof of \cref{theo:regclt}]
		By \cite[Lemma 2.1, part (b)]{deb2020fluctuations}, we have 
		\begin{align}\label{eq:curtaincall}
			\frac{1}{N}\sum_{i=1}^N (m_i-t_{\vrh})^2\overset{\P}{\longrightarrow} 1.
		\end{align}
		Next we look at the variance term $v_N$ in \eqref{eq:varterm}. Also assume that $\Xi''(\beta t_{\vrh}+B)(\upsilon_1-\upsilon_2\Xi''(\beta t_{\vrh}+B))>0$.  
		By leveraging \cref{cor:isingmean}, it suffices to show that $v_N^2\to \Xi''(\beta t_{\vrh}+B)(\upsilon_1-\upsilon_2 \Xi''(\beta t_{\vrh}+B))$. To wit, note that by \eqref{eq:curtaincall}, we have
		$$\frac{1}{N}\sum_{i=1}^N c_i^2\Xi''(\beta m_i+B)\overset{\P}{\longrightarrow} \upsilon_1 \Xi''(\beta t_{\vrh}+B), \quad \frac{1}{N}\sum_{i\neq j} c_i c_j \A_N(i,j)\Xi''(\beta m_i+B)\Xi''(\beta m_j+B) \overset{\P}{\longrightarrow} \upsilon_2(\Xi''(\beta t_{\vrh}+B))^2.$$
		As $\Xi''(\beta t_{\vrh}+B)(\upsilon_1-\upsilon_2\Xi''(\beta t_{\vrh}+B))>0$, the above display implies that $v_N^2\to \Xi''(\beta t_{\vrh}+B)(\upsilon_1-\upsilon_2 \Xi''(\beta t_{\vrh}+B))$. When $\Xi''(\beta t_{\vrh}+B)(\upsilon_1-\upsilon_2 \Xi''(\beta t_{\vrh}+B))=0$, the conclusion follows by repeating the same second moment calculation as in \cref{lem:auxtail}, part (b). We omit the details for brevity.
	\end{proof}
	
	\begin{proof}[Proof of \cref{thm:margbeta}]
		First let us show that $\hbem$ exists and  $\hbem\overset{\P}{\longrightarrow} \beta$. Consider the map $\beta \mapsto h_N(\beta)$ where 
		$$h_N(\tbt):=\frac{1}{N}\sum_{i=1}^N m_i(\si_i-\Xi'(\tbt m_i+B)),$$
		for $\tbt\ge 0$. Then $h_N(\cdot)$ is strictly decreasing. By \eqref{eq:curtaincall}, we have 
		$$h_N(\tbt)\overset{\P}{\longrightarrow} h(\tbt), \qquad \mbox{where}\qquad h(\tbt):=t_{\vrh}(t_{\vrh}-\Xi'(\tbt t_{\vrh}+B)).$$
		Now $h(\tb)$ is strictly decreasing and has a unique root at $\tbt=\beta$. Fix an arbitrary $\epsilon>0$. Then $h(\beta-\epsilon)>0>h(\beta+\epsilon)$. As $h_N(\tbt)\overset{\P}{\longrightarrow}h(\tbt)$, we have $h_N(\beta-\epsilon)>0>h_N(\beta+\epsilon)$ with probability converging to $1$. As $h_N(\cdot)$ is strictly decreasing, there exists unique $\hbem$ such that $h_N(\hbem)=0$ and $\hbem\in (\beta-\epsilon,\beta+\epsilon)$ with probability converging to $1$. As $\epsilon>0$ is arbitrary, $\hbem\overset{\P}{\longrightarrow}\beta$.  
		
		Now we will establish the asymptotic normality of $\hbem$ based on \cref{prop:CLTMPLE}. First note that by using \eqref{eq:curtaincall}, we have 
		$$-h_N'(\beta)\overset{\P}{\longrightarrow} t_{\vrh}^2 \Xi''(\beta t_{\vrh}+B).$$
		Next by combining \cref{lem:auxtail}, part (b), and \cref{theo:regclt}, we have: 
		$$\frac{1}{\sqrt{N}}\sum_{i=1}^N m_i(\si_i-\Xi'(\beta m_i+B))=\frac{t_{\vrh}}{\sqrt{N}}\sum_{i=1}^N (\si_i-\Xi'(\beta t_{\vrh}+B))+o_{\P}(1)\overset{w}{\longrightarrow} N(0,t_{\vrh}^2\Xi''(\beta t_{\vrh}+B)(1-\beta \Xi''(\beta t_{\vrh}+B))).$$
		The conclusion now follows by invoking \cref{prop:CLTMPLE}, \eqref{eq:limtheo}.
	\end{proof}
	
	\begin{proof}[Proof of \cref{thm:margB}]
		The existence of $\hbm$ and its consistency follow the same way as in the proof of \cref{thm:margbeta}. We omit the details for brevity. Define the map $\tB\mapsto H_N(\tB)$ where 
		$$H_N(\tB):=\frac{1}{N}\sum_{i=1}^N (\si_i-\Xi'(\beta m_i+\tB)).$$
		Once again by using \eqref{eq:curtaincall}, we have:
		$$-H_N'(B)\overset{\P}{\longrightarrow} \Xi''(\beta t_{\vrh}+B).$$
		From \cref{theo:regclt}, we have: 
		$$\frac{1}{\sqrt{N}}\sum_{i=1}^N (\si_i-\Xi'(\beta m_i+B))\overset{w}{\longrightarrow} N(0,\Xi''(\beta t_{\vrh}+B)(1-\beta \Xi''(\beta t_{\vrh}+B))).$$
		The conclusion follows by invoking \cref{prop:CLTMPLE}, \eqref{eq:limtheo}.
	\end{proof}
	\begin{proof}[Proof of \cref{prop:bipartg}]
		Recall the notion of cut norm from \cref{def:defirst}. With $\A_N$ chosen as the scaled adjacency matrix of a complete bipartite graph, as in \cref{prop:bipartg}, we have $d_{\square}(W_{N\A_N},W)\to 0$ where 
		\begin{align}\label{eq:baje_W}
			W(x,y)=&2\text{ if }(x,y)\in (0,.5)\times (.5\times 1) \text{ or }(x,y)\in (.5,1)\times (0,.5), \nonumber\\
			=&0\text{ otherwise }.
		\end{align}
		With $W(\cdot,\cdot)$ as in \eqref{eq:baje_W}, elementary calculus shows that with $\beta<0$ and large enough in absolute value, \eqref{eq:optimprob} admits exactly two optimizers which are of the form 
		$$\fs(x)=\begin{cases} t_1 & \mbox{if}\,\,\, 0<x\le 0.5 \\ t_2 & \mbox{if}\,\,\, 0.5<x\le 1\end{cases}, \qquad \fs(x)=\begin{cases} t_2 & \mbox{if}\,\,\, 0<x\le 0.5 \\ t_1 & \mbox{if}\,\,\, 0.5<x\le 1\end{cases},$$
		where $t_1,t_2$ are of different signs and magnitudes. Recall the definition of $U_N$ (with $g(x)=x$) from \eqref{eq:randvar} and that of $\Xi$ from \eqref{eq:candef1}. From \cite[Corollary 1.5]{bhattacharya2023gibbs}, we have
		\begin{align}\label{eq:mixcon}
			\min\left\{\frac{1}{N}\left(\sum_{i=1}^{N/2}|m_i-t_1|+\sum_{i=N/2+1}^N |m_i-t_2|\right),\frac{1}{N}\left(\sum_{i=1}^{N/2}|m_i-t_2|+\sum_{i=N/2+1}^N |m_i-t_1|\right)\right\}\overset{\P}{\longrightarrow}0.
		\end{align}
		By using \eqref{eq:uvradiate}, \eqref{eq:mixcon}, and the symmetry across the two communities would imply that 
		$$U_N=\frac{1}{N}\sum_{i=1}^{N/2} \Xi''(\beta m_i+B) + o_{\P}(1) \overset{w}{\longrightarrow} \frac{1}{2}\delta_{\frac{1}{2}\Xi''(\beta t_1+B)} + \frac{1}{2}\delta_{\frac{1}{2}\Xi''(\beta t_2+B)},$$
		which is a two component mixture. By using \eqref{eq:uvradiate} again, we also have $V_N\overset{\P}{\longrightarrow} 0$. The conclusion now follows by invoking \cref{theo:conmain}.
	\end{proof}
	
	\begin{proof}[Proof of \cref{prop:bipartpseudo}]
		Define 
		$$\mcH_N(\tlh,\tB):=\begin{pmatrix} \frac{1}{N}\sum_{i=1}^{N/2} (\si_i-\Xi'(\beta m_i+\tlh+\tB)) \\ \frac{1}{N}\sum_{i=1}^{N/2} (\si_i-\Xi'(\beta m_i+\tlh+\tB)+\frac{1}{N}\sum_{i=N/2+1}^N (\si_i-\Xi'(\beta m_i+B))\end{pmatrix}.$$
		Observe that as $(\tlh,\tB)$ lies in a compact set $K$, the Jacobian of $\mH_N$ given by 
		$$-\nabla \mH_N(\tlh,\tB)=\begin{pmatrix} \frac{1}{N}\sum_{i=1}^{N/2}\Xi''(\beta m_i+\tlh+\tB) & \frac{1}{N}\sum_{i=1}^{N/2}\Xi''(\beta m_i+\tlh+\tB) \\ \frac{1}{N}\sum_{i=1}^{N/2}\Xi''(\beta m_i+\tlh+\tB) & \frac{1}{N}\sum_{i=1}^{N/2}\Xi''(\beta m_i+\tlh+\tB)+\frac{1}{N}\sum_{i=N/2+1}^N \Xi''(\beta m_i+\tB)\end{pmatrix},$$
		has eigenvalues that are uniformly upper and lower bounded on $K$. 
		
		Moreover, $\mH_N(h,B)\overset{P}{\longrightarrow} 0$ by \cref{lem:auxtail}, part (a). Therefore by \cref{prop:ConsisMPLE}, $(\hh,\hbm)\overset{\P}{\longrightarrow} (0,B)$. 
		
		\noindent We will now use \cref{prop:CLTMPLE} to derive the asymptotic distribution of $(\hh,\hbm)$. Fix arbitrary $a,b\in \R$. Define $c_i=a\mathbf{1}(1\le i\le N/2)+b$. Recall the definition of $U_N$ and $V_N$ from \eqref{eq:randvar} with $\mc$ as defined above. Note that they can be simplified as 
		\begin{align*}
			U_N=\frac{1}{N}\sum_{i=1}^N c_i^2(\si_i^2-t_i^2)&=\frac{(a+b)^2}{N}\sum_{i=1}^{N/2} \Xi''(\beta m_i+B)+\frac{b^2}{N}\sum_{i=N/2+1}^N \Xi''(\beta m_i+B)+o_{\P}(1)\\ &=-\begin{pmatrix} a \\ b\end{pmatrix}^{\top} \nabla \mcH_N(0,B) \begin{pmatrix} a \\ b\end{pmatrix} + o_{\P}(1).
		\end{align*}
		by \cref{lem:auxtail}, part (a). Further by \eqref{eq:uvradiate}, we also get: 
		\begin{align*}
			V_N&=-\frac{\beta}{N}\sum_{1\le i\neq j\le N} (a\mathbf{1}(1\le i\le N/2)+b)(a\mathbf{1}(1\le i\le N/2)+b)\A_N(i,j)\Xi''(\beta m_i+B)\Xi''(\beta m_j+B)\\ &=-\left(\frac{4ab\beta}{N^2}+\frac{2b^2\beta}{N^2}\right)\sum_{1\le i\le N/2, \, N/2+1\le j\le N} \Xi''(\beta m_i+B)\Xi''(\beta m_j+B).
		\end{align*}
		Recall the definitions of $\tto$ and $\ttt$ from \cref{prop:bipartpseudo}. Define $$\kappa_{1,2}:=\begin{pmatrix}\frac{a^2}{2}\tto+\frac{b^2}{2}(\ttt+\tto)+ab\tto \\ -b\beta (a+b)\tto\ttt \end{pmatrix}.$$
		Define $\kappa_{2,1}$ as above by reversing the roles of $\tto$ and $\ttt$. 
		Then by using \eqref{eq:mixcon}, we get:
		\begin{align*}
			\begin{pmatrix} U_N \\ V_N \end{pmatrix} = \begin{pmatrix} -\begin{pmatrix} a \\ b \end{pmatrix}^{\top}\nabla \mcH_N(0,B) \begin{pmatrix} a\\ b\end{pmatrix} \\ V_N\end{pmatrix}\overset{w}{\longrightarrow}  \xi \delta_{\kappa_{1,2}}+(1-\xi)\delta_{\kappa_{2,1}},
		\end{align*}
		where $\xi$ is Rademacher. By using \cref{theo:conmain} and the above display yields 
		\begin{align*}
			\begin{pmatrix} \frac{1}{\sqrt{N}}\sum_{1\le i\le N/2} (\si_i-\Xi'(\beta m_i+B)) \\ \frac{1}{\sqrt{N}}\sum_{i=1}^N (\si_i-\Xi'(\beta m_i+B))\end{pmatrix}\overset{w}{\longrightarrow} & \xi N\left(\begin{pmatrix} 0 \\ 0\end{pmatrix},\begin{pmatrix} \frac{1}{2}\tto & \frac{1}{2}(\tto-\beta \tto\ttt) \\ \frac{1}{2}(\tto-\beta \tto\ttt) & \frac{1}{2}(\tto+\ttt)-\beta \tto\ttt\end{pmatrix}\right) \\ &+(1-\xi)N\left(\begin{pmatrix} 0 \\ 0\end{pmatrix},\begin{pmatrix} \frac{1}{2}\ttt & \frac{1}{2}(\ttt-\beta \tto\ttt) \\ \frac{1}{2}(\ttt-\beta \tto\ttt) & \frac{1}{2}(\tto+\ttt)-\beta \tto\ttt\end{pmatrix}\right).
		\end{align*}
		The conclusion follows by combining the two displays above with \cref{prop:CLTMPLE}.
	\end{proof}
	
	\subsection{Proofs from \cref{sec:highising}}
	\begin{proof}[Proof of \cref{thm:highordclt}]
		By invoking \cref{theo:CLTmain}, it suffices to show that \cref{as:cmean} holds. Fix $\{j_1,\ldots ,j_k\}$ and let $\tilde{\cS}\subseteq [N]$ be such that $\tilde{\cS}\cap\{j_1,\ldots ,j_k\}=\phi$. Write  $$\E[\si_i|\si_{\ell},\ell\neq i]=\Xi'(\beta m_i+B),$$
		where $m_i$s are defined in \eqref{eq:m_ihigh}. For convenience of the reader, we recall it here.
		\begin{align*}
			m_i=\frac{1}{N^{v-1}}\sum_{(i_2,\ldots ,i_v)\in \cS(N,v,i)}\mathrm{Sym}[\A_N](i,i_2,\ldots ,i_v)\left(\prod_{a=2}^v \si_{i_a}\right), \quad \mbox{for} \,\, i\in [N].
		\end{align*}
		Therefore
		$$\sum_{D\subseteq \{j_1,\ldots ,j_k\}} (-1)^{|D|} m_i^{\tilde{\cS}\cup D}=\frac{1}{N^{v-1}}\sum_{D\subseteq \{j_1,\ldots ,j_k\}} (-1)^{|D|}\sum_{(i_2,\ldots ,i_v)\in \cS(N,v,i)}\mathrm{Sym}[\A_N](i,i_2,\ldots ,i_v)\left(\prod_{a=2}^v \si_{i_a}\right)^{\tilde{\cS}\cup D}.$$
		Therefore $m_i$s are polynomials of degree $k$. So, for each summand, if there exists some $j_{\ell}$ such that $j_{\ell}\notin (i_2,\ldots ,i_v)$, then the corresponding summand equals $0$. This immediately implies that the left hand side of the above display equals $0$ if $k\ge v$. And for $k<v$, we have 
		\begin{align}\label{eq:mismooth}
			\bigg|\sum_{D\subseteq \{j_1,\ldots ,j_k\}} (-1)^{|D|} m_i^{\tilde{\cS}\cup D}\bigg|\lesssim \sum_{(i_2,\ldots ,i_{v-k})\in S(N,v,\{i,j_1,\ldots ,j_k\})}\mathrm{Sym}[\A_N](i,j_1,\ldots ,j_k,i_2,\ldots ,i_{v-k}), 
		\end{align}
		where $S(N,v,\{i,j_1,\ldots ,j_k\})$ denotes the set of all distinct tuples of $[N]^{v-k-1}$ such that none of the elements equal to $\{i,j_1,\ldots ,j_k\}$. \cref{as:cmean} now follows from combining \eqref{eq:mismooth} with \cref{lem:smoothcont}.
	\end{proof}
	
	\begin{proof}[Proof of \cref{thm:jointCLThigh}]
		The proof of this Theorem is exactly the same as that of \cref{thm:jointCLT} except for the invertibility of $\mca_{\fs}$. Therefore, for brevity, we will only prove that $\mca_{\fs}$ is invertible under the assumptions of the Theorem. As $B>0$, by replacing a function $f:[0,1]\to [-1,1]$ with $|f|$, it follows that the unique $\fs$ that optimizes \eqref{eq:highoptim} must be non-negative almost everywhere. Also $f\equiv 0$ is not an optimizer of \eqref{eq:highoptim} as $B>0$. Recall the definition of $\mca_{\fs}$ from \eqref{eq:amat}. Then by the Cauchy-Schwartz inequality $\mca_{\fs}$ is singular if and only if $\fs$ is constant everywhere. However under the irregularity assumption \cref{as:irretensor} $\fs$ is not a constant function by \cite[Theorem 1.2(ii)]{bhattacharya2023gibbs}. Therefore $\mca_{\fs}$ must be invertible. This completes the proof.
	\end{proof}
	
	\subsection{Proofs from \cref{sec:ergm}}
	\begin{proof}[Proof of \cref{thm:ergmclt}]
		Once again, by \cref{theo:CLTmain}, the conclusion will follow if we can verify \cref{as:cmean}. Without loss of generality, we will assume that $\tilde{\cS}=\phi$. Recall from \eqref{eq:conergm} that
		\begin{equation}\label{eq:whichterm}
			\E[Y_{ij}|Y_{-ij}]=L(\eta_{ij}), \quad \eta_{ij}:=\sum_{m=1}^k \frac{\beta_m}{N^{v_m-2}}\sum_{(a,b)\in E(H_m)}\sum_{\substack{(k_1,\ldots ,k_{v_m}) \textrm{ are distinct, }\\ \{k_a,k_b\}=\{i,j\}}}\prod_{(p,q)\in E(H_m)\setminus (a,b)} Y_{k_p k_q}.
		\end{equation}
		Fix the edges $\mathfrak{E}_1=(i,j)$ and let $\mathfrak{E}_{\ell}=(i_{\ell},j_{\ell})$ for $2\le \ell\le r$. Let $\mathrm{CV}(\mathfrak{E}_1,\ldots ,\mathfrak{E}_r)$ denote the number of distinct vertices within the edge set $\mathfrak{E}_1,\ldots ,\mathfrak{E}_r$. Define the sequence of tensors $\Q_{N,r}$ defined by 
		$$\Q_{N,r}(\mathfrak{E}_1,\ldots ,\mathfrak{E}_r)=\frac{1}{N^{\mathrm{CV}(\mathfrak{E}_1,\ldots ,\mathfrak{E}_r)-2}}.$$
		It is easy to check that the max row sums of the above tensors are bounded for all $r$. 
		Therefore the left hand side of the above display can be bounded by 
		\begin{align*}
			&\;\;\;\;\sum_{D\subseteq \{\mathfrak{E}_2,\ldots ,\mathfrak{E}_r\}} (-1)^{|D|}\eta_{\mathfrak{E}_1}^{\mathfrak{E}_2,\ldots ,\mathfrak{E}_r}\\ &=\sum_{m=1}^k \frac{\beta_m}{N^{v_m-2}}\sum_{D\subseteq \{\mathfrak{E}_2,\ldots ,\mathfrak{E}_r\}}(-1)^{|D|}\left(\sum_{(a,b)\in E(H_m)}\sum_{\substack{(k_1,\ldots ,k_{v_m}) \textrm{ are distinct, }\\ \{k_a,k_b\}=\{i,j\}}}\prod_{(p,q)\in E(H_m)\setminus (a,b)} Y_{k_p k_q}\right)^{\mathfrak{E}_2,\ldots ,\mathfrak{E}_r}. 
		\end{align*}
		Therefore all the vertices covered by $\mathfrak{E}_2,\ldots ,\mathfrak{E}_r$ must be covered by one of the $k_{\ell}$s. As $\{k_a,k_b\}=\{i,j\}$, the above claim restricts $\mathrm{CV}(\mathfrak{e}_1,\ldots ,\mathfrak{E}_r)$ many of the $k_{\ell}$s. As a result, we have: 
		\begin{align*}
			\bigg|\sum_{D\subseteq \{\mathfrak{E}_2,\ldots ,\mathfrak{E}_r\}} (-1)^{|D|}\eta_{\mathfrak{E}_1}^{\mathfrak{E}_2,\ldots ,\mathfrak{E}_r}\bigg|&\lesssim \sum_{m=1}^k \frac{\beta_m}{N^{v_m-2}}N^{v_m-\mathrm{CV}(\mathfrak{E}_1,\ldots ,\mathfrak{E}_r)}\lesssim \frac{1}{N^{\mathrm{CV}(\mathfrak{E}_1,\ldots ,\mathfrak{E}_r)-2}}=\Q_{N,r}(\mathfrak{E}_1,\ldots ,\mathfrak{E}_r).
		\end{align*}
		This completes the proof.
	\end{proof}
	
	\begin{proof}[Proof of \cref{cor:ergmclt}]
		Using \cref{thm:ergmclt}, we only need to find the weak limits of $\Une$ and $\Vne$ under the sub-critical regime. We will leverage the fact that in the sub-critical regime, draws from the model \eqref{eq:ergm} are equivalent (for weak limits) to Erd\H{o}s-R\'{e}nyi random graphs with edge probability $\ps$. In particular, by using \cite[Theorem 1.6]{Reinert2019}, we have: 
		\begin{align}\label{eq:etaweak}
			\frac{1}{{N\choose 2}}\sum_{1\le i<j\le N} \delta_{\eta_{ij}}\overset{w}{\longrightarrow}2\sum_{m=1}^k \beta_m e_m (\ps)^{m-1}.
		\end{align}
		
		\emph{Limit of $\Une$}. By \cref{lem:auxtail}, part (a), we have 
		$$\Une=\frac{1}{{N \choose 2}}\sum_{1\le i<j\le N} (Y_{ij}-L^2(\eta_{ij}))=\frac{1}{{N\choose 2}}\sum_{1\le i<j\le N} (L(\eta_{ij})-L^2(\eta_{ij}))+o_{\PPe}(1).$$
		By \eqref{eq:etaweak}, we then get: 
		$$\Une\overset{\PPe}{\longrightarrow} \ps(1-\ps).$$
		
		\emph{Limit of $\Vne$}. Through direct computations, we have: 
		\begin{align*}
			\Vne&=\frac{1}{{N\choose 2}}\sum_{\substack{(i_1,j_1)\neq (i_2,j_2) \\ \in \mathcal{I}}}(Y_{i_1 j_1}-L(\eta_{i_1 j_1}))(L(\eta_{i_2 j_2}^{(i_1,j_1)})-L(\eta_{i_2 j_2})) \\ &=\frac{1}{{N\choose 2}}\sum_{\substack{(i_1,j_1)\neq (i_2,j_2) \\ \in \mathcal{I}}}(Y_{i_1 j_1}-L(\eta_{i_1 j_1}))(\eta_{i_2 j_2}^{(i_1,j_1)}-\eta_{i_2 j_2})L'(\eta_{i_2 j_2}))+o_{\PPe}(1).
		\end{align*}
		Next observe that 
		\begin{align*}
			&\;\;\;\;\eta_{i_2 j_2}^{(i_1,j_1)}-\eta_{i_2 j_2} \\ &=-Y_{i_1 j_1}\sum_{m=1}^k \frac{\beta_m}{N^{v_m-2}}\sum_{\substack{(a,b),(c,d)\\ \in E(H_m)}}\sum_{\substack{(k_1,\ldots ,k_{v_m}) \textrm{ are distinct, }\\ \{k_a,k_b\}=\{i_2,j_2\}, \{k_c,k_d\}=\{i_1,j_1\}}}\prod_{(p,q)\in E(H_m)\setminus ((a,b)\cup (c,d))} Y_{k_p k_q}.
		\end{align*}
		Combining the equations above with \cref{lem:auxtail}, part (a), we then get: 
		\begin{align*}
			&\;\;\;\;\Vne\\ &=-\frac{1}{{N\choose 2}}\sum_{\substack{(i_1,j_1),(i_2,j_2)\\ \in\mathcal{I}}} (L(\eta_{i_1 j_1})-L^2(\eta_{i_1 j_1}))(L(\eta_{i_2 j_2})-L^2(\eta_{i_2 j_2}))\frac{\beta_m}{N^{v_m-2}}\\ &\qquad\qquad\qquad\sum_{\substack{(a,b),(c,d)\\ \in E(H_m)}}\sum_{\substack{(k_1,\ldots ,k_{v_m}) \textrm{ are distinct, }\\ \{k_a,k_b\}=\{i_2,j_2\}, \{k_c,k_d\}=\{i_1,j_1\}}}\prod_{(p,q)\in E(H_m)\setminus ((a,b)\cup (c,d))} Y_{k_p k_q}+o_{\PPe}(1) \\ &=2(\ps(1-\ps))^2\sum_{m=1}^k \beta_m (\ps)^{e_m-2}e_m(e_m-1)+o_{\PPe}(1),
		\end{align*}
		where the last line follows from \eqref{eq:etaweak}. As $\vphb'(\ps)=2\sum_{m=1}^k \beta_m (\ps)^{e_m-2}e_m(e_m-1)$. This completes the proof.
	\end{proof}
	
	\begin{proof}[Proof of \cref{thm:MPLergmCLT}]
		Recall from \eqref{eq:plf} that the pseudolikelihood function is given by 
		$$\mathrm{PL}(\beta_1):=\sum_{(i,j)\in\mathcal{I}} \left(Y_{ij}\eta_{ij}(\beta_1)-\log(1+\exp(\eta_{ij}(\beta_1))\right).$$
		Therefore 
		$$\mathrm{PL}'(\beta_1)=2\sum_{(i,j)\in\mathcal{I}}(Y_{ij}-L(\eta_{ij}), \qquad \mathrm{PL}''(\beta_1)=-4\sum_{(i,j)\in\mathcal{I}}L(\eta_{ij})(1-L(\eta_{ij})).$$
		As $K$ is a known compact set, $\hbpm$ exists and $\hbpm\overset{\PPe}{\longrightarrow}\beta_1$ by \cref{prop:ConsisMPLE}. As a result, 
		the conclusion in \eqref{eq:firstconc} follows from \cref{prop:CLTMPLE}. 
		
		For the conclusion in \eqref{eq:secondconc}, note that 
		$$\frac{1}{{N\choose 2}}\sum_{(i,j)\in\mathcal{I}} L(\eta_{ij})(1-L(\eta_{ij}))\overset{\PPe}{\longrightarrow}\ps(1-\ps).$$
		The conclusion now follows by combining the above display with \eqref{eq:firstconc} and \cref{cor:ergmclt}.
	\end{proof}
	\section{Proof of auxiliary results}\label{sec:pfauxi}
	This section is devoted to proving some auxiliary results from earlier in the paper whose proofs were deferred.
	
	\begin{proof}[Proof of \cref{prop:CLTMPLE}]
		By (A3), there exists a sequence $r_N\to 0$ slow enough such that $\mathbb{P}_{\theta_0}(\lVert \htmp-\theta_0\rVert\ge r_N)\to 0$.  Define $B(\theta_0;r_N):=\{\theta:\lVert \theta-\theta_0\rVert\le \epsilon\}$. Then for all $N$ large enough, $B(\theta_0;r_N)$ is contained in the interior of the parameter space $\Theta$. Therefore without loss of generality, we can always operate under the event $\htmp\in B(\theta_0;r_N)$. Note that
		$$\sum_{i=1}^N \nabla f_i(\htmp)=0.$$
		By a first order Taylor expansion of the left hand side, we observe that there exists $\tilde{\theta}\in B(\theta_0;r_N)$ (as both $\theta_0,\htmp\in B(\theta_0;r_N)$) such that 
		$$\left(\frac{1}{N}\sum_{i=1}^N \nabla^2 f_i(\tilde{\theta})\right)\sqrt{N}(\htmp-\theta_0)=-\frac{1}{\sqrt{N}}\sum_{i=1}^N \nabla f_i(\theta_0).$$
		By (A1), $$\left(\frac{1}{N}\sum_{i=1}^N \nabla^2 f_i(\tilde{\theta})\right)^{-1}\left(\frac{1}{N}\sum_{i=1}^N \nabla^2 f_i(\theta_0)\right)\overset{\mathbb{P}_{\theta_0}}{\longrightarrow} \mathbf{I}_p.$$
		Therefore $\sqrt{N}(\htmp-\theta_0)=O_p(1)$. This implies 
		$$\left(\frac{1}{N}\sum_{i=1}^N \nabla^2 f_i(\theta_0)\right)\sqrt{N}(\htmp-\theta_0)=-\frac{1}{\sqrt{N}}\sum_{i=1}^N \nabla f_i(\theta_0)+o_{\mathbb{P}_{\theta_0}}(1).$$
		The conclusion now follows by using (A2).
	\end{proof}
	
	\begin{proof}[Proof of \cref{prop:ConsisMPLE}]
		As $\sum_{i=1}^N \nabla f_i(\htmp)=0$, by (B1), we have:
		\begin{align*}
			&\;\;\;\;\left\langle \frac{1}{N}\sum_{i=1}^N \nabla f_i(\htmp)-\frac{1}{N}\sum_{i=1}^N \nabla f_i(\theta_0),\htmp-\theta_0\right\rangle \le -\alpha \lVert \htmp-\theta_0\rVert^2 \\ &\implies \left\lVert \frac{1}{N}\sum_{i=1}^N \nabla f_i(\theta_0)\right\rVert \lVert \htmp-\theta_0\rVert \ge \alpha \lVert \htmp-\theta_0\rVert^2.  
		\end{align*}
		The last inequality follows from Cauchy-Schwartz. The conclusion now follows from (B2).
	\end{proof}
	\begin{proof}[Proof of~\cref{lem:auxtail}]\label{Sec:concproof}
		
		\emph{Part (a)}. Let $\ms$ be drawn according to~\eqref{eq:pivotstat} and suppose $\tms$ is drawn by moving one step in the Glauber dynamics, i.e., let $I$ be a random variable which is discrete uniform on $\{1,2,\ldots,N\}$, and replace the $I$-th coordinate of $\ms$ by an element drawn from the conditional distribution of $\sigma_I$ given the rest of the $\sigma_j$'s. It is easy to see that $(\ms,\tms)$ forms an exchangeable pair of random variables. Next, define an anti-symmetric function $F(\mathbf{x},\mathbf{y}):=\sum_{i=1}^N d_i(g(x_i)-g(y_i))$, which yields that
		$$\E\left(F(\ms,\tms)|\ms\right)=\frac{1}{N}\sum_{i=1}^N d_i(g(\sigma_i)-t_i)=:f(\ms).$$
		Observe that
		$$f(\ms)-f(\tms)=\frac{1}{N}d_I(g(\sigma_I)-g(\tilde{\sigma}_I))-\frac{1}{N}\sum_{i\neq I} d_i(t_i-\tilde{t}_i),$$
		where $\tilde{t}_i$ is defined as in~\eqref{eq:consig2} with $\ms$ replaced by $\tms$. Also note that, by~\cref{as:cmean}, $|t_i-\tilde{t}_i|\le 2\Q_{N,2}(i,I)$ for all $i\neq I$. By using these observations, it is easy to see that
		\begin{align*}
			&\E\left[|(f(\ms)-f(\tms))F(\ms,\tms)|\big|\ms\right]\\ &=\E\left[\frac{1}{N}d_I^2(g(\si_I)-g(\tilde{\si}_I))^2+\frac{1}{N}\sum_{i\neq I} |d_i||d_I||t_i-\tilde{t}_i||g(\si_I)-g(\tilde{\si}_I)|\big|\ms\right]\\ &\lesssim \frac{1}{N^2}\sum_{i=1}^N d_i^2+\frac{1}{N^2}\sum_{i\neq j}|d_i||d_j|\Q_{N,2}(i,j)\lesssim \frac{1}{N^2}\sum_{i=1}^N d_i^2.
		\end{align*}
		By invoking~\cite[Theorem 3.3]{Cha2005}, we get the desired conclusion.
		
		\vspace{0.1in}
		
		\emph{Part (b)}. Recall the definition of $t_i^j$, $i\neq j$ from~\eqref{eq:consig}. Observe that
		$$\E\left(\sum_{i=1}^N d_i(g(\si_i)-t_i)r_i\right)^2=\E\left(\sum_{i=1}^N d_i^2(g(\si_i)-t_i)^2r_i^2\right)+\sum_{i\neq j}d_id_j\E\left((g(\si_i)-t_i)(g(\si_j)-t_j)r_ir_j\right).$$
		The first term in the above display is clearly $\lesssim N$ under the assumptions of~\cref{lem:auxtail}. Focusing on the second term, note that for $i\neq j$, we have:
		$$(g(\si_i)-t_i)(g(\si_j)-t_j)r_ir_j=(g(\si_i)-t_i^j)(g(\si_j)-t_j)r_i^jr_j+O\left(\Q_{N,2}(i,j)\right),$$
		where the above follows from \cref{as:cmean}. As 
		$$\E (g(\si_i)-t_i^j)(g(\si_j)-t_j)r_i^j r_j = 0$$
		for $i\neq j$. Combining the above displays we get:
		$$\E\left(\sum_{i=1}^N d_i(g(\si_i)-t_i)r_i\right)^2\lesssim N+\sum_{i\neq j}|d_i||d_j|\Q_{N,2}(i,j)\le N + \lambda_1(\Q_{N,2})\sum_{i=1}^N d_i^2 \lesssim N,$$
		thereby completing the proof.
	\end{proof}
	
	\begin{proof}[Proof of~\cref{lem:fixsol}]
		Consider the following sequence of probability measures:
		$$\frac{d\vrh_{\theta}}{d\vrh}(x)=\exp(\theta x-\Xi(\theta))$$
		for $\theta\in\R$. By standard properties of exponential families, $\Xi''(\theta)=\mbox{Var}_{\vrh_{\theta}}(X)>0$ as $\vrh$ is assumed to be non-degenerate. Therefore $\Xi'(\cdot)$ is one-to-one and $(\Xi')^{-1}(\cdot)$ is well defined. Further, it is easy to check that $\phi(\cdot)$ is maximized in the interior of the support of $(\Xi')^{-1}(\cdot)$ (see e.g.,~\cite[Lemma 1(ii)]{mukherjee2021variational}). Consequently, any maximizer (local or global) of $\phi(\cdot)$ must satisfy
		\begin{equation*}
			\tph(x)=0,\quad \mbox{where}\,\,\, \tph(x):=x-\Xi'(rx+s).
		\end{equation*}
		As the case $r=0$ is trivial, we will consider $r>0$ throughout the rest of the proof.
		
		\emph{Proof of (a).} Suppose $(r,s)\in\Theta_{11}$. Note that $C(0)=0$. As the probability measure $\vrh$ is symmetric around 0, we also have $\Xi'(0)=0$. Therefore $\tph(0)=0$. Further, observe that
		\begin{align}\label{eq:fixsol1}
			\tph'(x)=1-r \Xi''(rx),\qquad \phi''(x)=-r^2 \Xi'''(rx).
		\end{align}
		We split the argument into two cases: (i) $r<(\Xi''(0))^{-1}$, and (ii) $r=(\Xi''(0))^{-1}$.
		
		\emph{Case (i).} Note that $\Xi'(\cdot)$ and hence $\tph(\cdot)$ are both odd functions. It then suffices to show that $\tph(x)>0$ for $x>0$. We proceed by contradiction. Assume that there exists $x_0>0$ such that $\tph(x_0)=0$. First, observe that $\tph'(0)=1-r\Xi''(0)>0$ which implies that $0$ is a local maxima of $\phi(\cdot)$. Further $\lim\limits_{x\to\infty} \tph(x)=\infty$, which implies that $\tph(\cdot)$ must have at least two positive roots (recall $0$ is already shown to be a root of $\tph(\cdot)$). By Rolle's Theorem, $\tph''(\cdot)$ must have at least one positive root. Consequently, by~\eqref{eq:fixsol1}, $\Xi'''(\cdot)$ must have a positive root. As $\Xi'''(\cdot)$ is an odd function, then Assumption~\eqref{eq:secasn} implies $\Xi'''(\cdot)$ must be $0$ in a neighborhood of $0$. This forces $\vrh$ to be Gaussian, which is a contradiction! This completes the proof for (i).
		
		\emph{Case (ii).} For $(r,s)\in\Theta_{11}$, note that $\phi(\cdot)$ implicitly depends on $r$. Therefore writing $\phi(r;x)\equiv \phi(x)$, we have from case (i) that $\phi(r,x)< \phi(r,0)$ for all $r<(\Xi''(0))^{-1}$ and all $x$. By continuity, this implies $\phi((\Xi''(0))^{-1};x)\leq \phi((\Xi''(0))^{-1};0)$ and consequently $0$ is a global maximizer of $\phi(\cdot)\equiv \phi(r;\cdot)$ for $r=(\Xi''(0))^{-1}$. As a result $\phi(\cdot)$ is negative at some point close to $0$ which again implies that either $0$ is the unique maximizer of $\phi(\cdot)$ or $\tph(x)=0$ has at least two positive solutions. The rest of the argument is same as in case (i).
		
		\emph{Proof of (b).} By symmetry, it is enough to prove part (b) for $s>0$. First note that $\Xi'(s)>0$ which implies $\tph(0)<0$. As $\lim\limits_{x\to\infty} \tph(x)=\infty$, either $\tph(\cdot)$ has a unique positive root or at least $3$ positive roots. If the latter holds, then $\Xi'''(\cdot)$ must have a positive root, which gives a contradiction by the same argument as used in the proof of part (a)(i). This implies $\phi(\cdot)$ has a unique positive maximizer, say at $t_{\vrh}$. Also $\Xi'''(r t_{\vrh}+s)\neq 0\implies \tph''(t_{\vrh})\neq 0$. Consequently, we must have $\tph'(t_{\vrh})=1-r\Xi'(r t_{\vrh + s})>0$.
		
		\emph{Proof of (c).} In this case $\tph(0)=0$ and $\tph'(0)<0$. Therefore, $\tph(\cdot)$ either has a unique positive root or at least $3$ positive roots. The rest of the argument is same as in the other parts of the lemma, so we omit them for brevity.
	\end{proof}
	
\end{document}